\documentclass[10pt,a4paper]{amsart}
\usepackage[toc,page]{appendix}
\usepackage{a4}
\usepackage[active]{srcltx}
\usepackage{amsmath,bbm,esint}
\usepackage{amssymb}
\usepackage{amsfonts}
\usepackage{mathrsfs} 
\usepackage{graphicx}
\usepackage{tikz}
\usepackage{color}
\usepackage[colorlinks=true,urlcolor=blue,
citecolor=red,linkcolor=blue,linktocpage,pdfpagelabels,bookmarksnumbered,bookmarksopen]{hyperref}
\usepackage{xcolor}

\usepackage[footskip=.6in]{geometry}
\usepackage{fancyhdr,floatpag}

\usepackage{caption}
\usepackage{subcaption}
\graphicspath{{./graphics/}}

\floatpagestyle{fancy}

\def\a{\alpha}
\def\b{\beta}
\def\d{\delta}
\def\g{\gamma}

\def\l{\lambda}

\def\vphi{\varphi}

\def\n{\nu}
\def\s{\sigma}

\def\t{\tau}
\def\e{\varepsilon}
\def\Om{\Omega}

\newcommand{\hmu}{{\hat \mu}}

\newcommand{\ophi}{{\overline \phi}} 

\newcommand{\cB}{{\mathcal B}}
\newcommand{\cC}{{C}}   

\newcommand{\cE}{{\mathcal E}}
\newcommand{\cF}{{\mathcal F}}

\newcommand{\cL}{{\mathcal L}}
\newcommand{\cK}{{\mathcal K}}
\newcommand{\cI}{{\mathcal I}}
\newcommand{\I}{{\mathcal I}}
\newcommand{\cJ}{{\mathcal J}}

\newcommand{\cS}{{S}} 

\newcommand{\cU}{{\mathcal U}}

\newcommand{\fF}{{\mathfrak F}}

\newcommand{\fX}{{\mathfrak X}}

\newcommand{\trho}{{\widetilde\rho}}

\newcommand{\R}{{\mathbb R}}

\newcommand{\N}{\mathbb{N}}

\newcommand{\Rinf}{{\R \cup \{+\infty\}}}
\newcommand{\Rinfm}{{\R \cup \{-\infty\}}}
\newcommand{\Rinfb}{{\R \cup \{\pm \infty\}}}

\newcommand{\E}{{\textbf E}}
\newcommand{\ov}{\overline}

\newcommand{\K}{{\mathcal K}}

\newcommand{\weakly}{\ensuremath{\rightharpoonup}}
\newcommand{\weaklys}{\stackrel{\star}{\rightharpoonup}}
\newcommand{\da}{\ensuremath{\downarrow}}

\newcommand{\mres}{\mathbin{\vrule height 1.6ex depth 0pt width
0.13ex\vrule height 0.13ex depth 0pt width 1.3ex}}

\newcommand{\sB}{\mathscr{B}}

\newcommand{\sD}{\mathscr{D}}

\newcommand{\sL}{\mathscr{L}}

\newcommand{\sH}{{\mathscr H}}
\newcommand{\sM}{{\mathscr M}}
\newcommand{\sP}{{\mathscr P}}

\newcommand{\one}{\mathbbm{1}}

\newcommand{\ds}{\displaystyle}

\newcommand{\spt}{{\rm{spt}}}

\newcommand{\diver}{\nabla\cdot}

\renewcommand{\ae}{{\rm{a.e.}}}
\newcommand{\dd}{\hspace{0.7pt}{\rm d}}
\newcommand{\id}{{\rm id}}

\newcommand{\tE}{\widetilde{\textbf E}}

\renewcommand{\v}{{\textbf{v}}}

\renewcommand{\n}{{\textbf n}}

\newcommand{\be}{\begin{equation}}
\newcommand{\ee}{\end{equation}}

\newcommand{\ba}{\begin{array}}
\newcommand{\ea}{\end{array}}

\newcommand{\pt}{{p^\t}}

\def\osc{\mathop{\textnormal{osc}}\limits}%

\newcommand{\rhot}{{\rho^\t}}

\newcommand{\Et}{{\E^\t}}

\newcommand{\Sa}{{S_a}}
\newcommand{\Sb}{{S_b}}

\newcommand{\Ro}{{\R^+ \setminus \{1\}}}
\newcommand{\Rp}{{\R^+}}
\newcommand{\Rpz}{{[0, +\infty)}}

\newtheorem{remark}{\textbf{Remark}}[section]
\newtheorem{theorem}{\textbf{Theorem}}[section]
\newtheorem{lemma}[theorem]{\textbf{Lemma}}
\newtheorem{corollary}[theorem]{\textbf{Corollary}}
\newtheorem{proposition}[theorem]{\textbf{Proposition}}

\newtheorem{definition}[remark]{\textbf{Definition}}
\newtheorem{example}[remark]{\textbf{Example}}

\newtheorem{assumption}[remark]{\textbf{Assumption}}

\newtheorem{innercustomgeneric}{\customgenericname}
\providecommand{\customgenericname}{}
\newcommand{\newcustomtheorem}[2]{%
  \newenvironment{#1}[1]
  {%
   \renewcommand\customgenericname{#2}%
   \renewcommand\theinnercustomgeneric{##1}%
   \innercustomgeneric
  }
  {\endinnercustomgeneric}
}

\newcustomtheorem{customthm}{Framework}

\numberwithin{equation}{section}

\usepackage{comment}

\DeclareMathOperator*{\argmin}{\arg\!\min}
\def\liminf{\mathop{\lim\,\inf}\limits}%
\def\limsup{\mathop{\lim\,\sup}\limits}%
\def\argmin{\mathop{\arg\,\min}\limits}%

\title[]{Degenerate nonlinear parabolic equations with discontinuous diffusion coefficients}  

\author[D. Kwon]{Dohyun Kwon}  
\date{\today}
\address{Department of Mathematics, University of Wisconsin-Madison, 480 Lincoln Dr., Madison, WI 53706, USA}
\email{dkwon7@wisc.edu}

\author[A.R. M\'esz\'aros]{Alp\'ar Rich\'ard M\'esz\'aros}  
\date{\today}
\address{Department of Mathematical Sciences, Durham University, Durham DH1 3LE, UK}
\email{alpar.r.meszaros@durham.ac.uk}

\begin{document}
\maketitle

\begin{abstract}
This paper is devoted to the study of some nonlinear parabolic equations with discontinuous diffusion intensities. Such problems appear naturally in physical and biological models. Our analysis is based on variational techniques and in particular on gradient flows in the space of probability measures equipped with the distance arising in the Monge-Kantorovich optimal transport problem. The associated internal energy functionals in general fail to be differentiable, therefore classical results do not apply directly in our setting. We study the combination of both linear and porous medium type diffusions and we show the existence and uniqueness of the solutions in the sense of distributions in suitable Sobolev spaces. Our notion of solution allows us to give a fine characterization of the emerging  critical regions, observed previously in numerical experiments. A link to a three phase free boundary problem is also pointed out.
\end{abstract}

2020 Mathematics Subject Classification: 35K65; 49Q22; 46N10; 35D30.

\tableofcontents

\section{Introduction}
In this paper we investigate a class of degenerate nonlinear parabolic equations, with discontinuous diffusion intensities. These can be written formally as the Cauchy problem for the unknown $\rho:[0,T]\times\Om\to[0,+\infty)$
\begin{align}\label{dd}
\begin{cases}
\partial_t\rho -\Delta\vphi(\rho)  -\nabla\cdot(\nabla\Phi\rho) =0, &\hbox{ in } (0,T) \times \Om,\\
(\nabla \vphi(\rho)+\nabla\Phi\rho)  \cdot \n = 0 , &\hbox{ on }  (0,T) \times \partial \Om,\\
\rho(0,\cdot)=\rho_0, &\hbox{ in } \Om,
\end{cases}
\end{align}
where $T>0$ is a given time horizon, $\Om \subset \R^d$ is the closure of a bounded convex open set with smooth boundary, $\Phi:\Om\to\R$ is a given Lipschitz continuous potential function, $\rho_0\in\sP(\Om)$ is a nonnegative Borel probability measure and the diffusion intensity function $\vphi:[0,+\infty)\to\R$ is supposed to have a discontinuity at $\rho=1$. Therefore, $\vphi$ is extended to be a multi-valued function at the discontinuity and in addition, it is supposed to be monotone in the sense that if $\eta^i\in\vphi(\rho^i)$, then
\begin{align*}
(\eta^1-\eta^2)(\rho^1-\rho^2)\ge 0.
\end{align*}
Our aim is to identify a large class of potentials $\Phi$, nonlinearities $\vphi$ and initial data $\rho_0$, for which we show the well-posedness of \eqref{dd} in a suitable distributional sense. Furthermore, we aim to describe some fine properties of the solutions. Let us remark that our results are expected to be valid also in the case of $\Om=\R^d$, without running into many technical difficulties, provided we work in the space of measures having enough uniform moment bounds, just as in the original works \cite{JKO,otto}.

\medskip

Such problems appear naturally in physical and biological models. Let us briefly describe two of these. In \cite{BanJan}, the authors study so-called phenomena of {\it self-organized criticality}. These arise typically in {\it sandpile models}, in which the sand particles are subject to a constant diffusion only at regions where their density is greater than a given threshold, otherwise they remain still. At the macroscopic level, in the cited reference such models were described by equations similar to \eqref{dd}, with $\Phi=0$ and $\vphi(\rho)=0$, if $\rho<\rho_c$ and $\vphi=const$ if $\rho\ge \rho_c$ (where $\rho_c$ is a given threshold value). Via an approximation procedure and numerical investigations, the authors observe the growth (in time) of the {\it critical region}, where $\rho=\rho_c$, therefore, they conclude that  particles following this diffusion law `self-organize into criticality'. Our main results in this paper will rigorously confirm such phenomena. 

In \cite{ChoKim} the authors study diffusion models for biological organisms that increase their motility when food or other resource is insufficient. They refer to such phenomena as {\it starvation driven diffusion}. At the mathematical level, their model consists in a system of reaction-diffusion equations for two species, where the diffusion rates are discontinuous functions depending on the (food supply)/(food demand) ratio in the global population. In this model, a Lotka-Volterra type competition is implemented and a particular example is provided when one species follows the starvation driven diffusion and the other follows the linear diffusion. The authors conclude, by means of numerical simulations, that in heterogeneous environments the starvation driven diffusion turns out to be a better survival strategy than the linear one. Therefore, by this conclusion the authors would like to underline also the fact that in biological models, discontinuous diffusion rates might appear in a very natural way, resulting many times in a better description of competing biological systems. 

\medskip

Degenerate nonlinear parabolic problems like \eqref{dd} received a lot of attention in the past couple of decades. For a non-exhaustive list of classical works on this subject we refer to \cite{BenBreCra, BenCra, BenBocHer, CafEva, Car} and the references therein. In majority of the literature, however, the nonlinearity $\vphi$ is taken to be a continuous function.

To the best of our knowledge, except in particular cases involving linear type diffusions and/or bounded initial data (see for instance in \cite{BlaRockRus, BarRockRuss,BarRoc18}), our model problem in its full generality has not been addressed previously in the literature. The  solution obtained in the aforementioned references heuristically can be written as pairs $(\rho,\eta_\rho)$ belonging to well-chosen function spaces, such that
$$\partial_t\rho-\Delta(\eta_\rho)-\nabla\cdot(\nabla\Phi\rho)=0$$
is fulfilled either in the  distributional or entropic sense and $\rho(t,x)\in\eta_\rho(t,x)$ a.e. 


\medskip

In this paper, we rely on the gradient flow structure of \eqref{dd} in the space of probability measures, when equipped with the distance $W_2$ arising in the Monge-Kantorovich optimal transport problem. To \eqref{dd}, we associate an entropy functional $\cE:\sP(\Om)\to\R\cup\{+\infty\}$, defined as 
\begin{align}\label{energy}
\ds\cE(\rho):=\left\{
\begin{array}{ll}
\ds\int_\Om \cS(\rho(x))\dd x + \int_\Om\Phi(x)\dd\rho(x), & {\rm{if}}\ S(\rho)\in L^1(\Om),\\
+\infty, & {\rm{otherwise}},
\end{array}
\right.
\end{align}
where $S:[0,+\infty)\to\R$ is a given function.
At the formal level, the relationship between $\vphi$ and $S$ can be written as
\begin{align*}
\vphi(\rho) = \rho \cS'(\rho) - \cS(\rho) + \cS(1) \hbox{ and }  \vphi'(\rho) = \rho S''(\rho), \ \ {\rm{if}}\ \rho\neq 1.
\end{align*}
We observe that the discontinuity of $\vphi$ at $\rho=1$ corresponds to the non-differentiability of $S$ at $\rho=1$. Furthermore, as $\vphi$ is monotone, we impose that $S$ is convex and the multiple values of $\vphi$ can be represented by the subdifferential of $S$. 
In this sense, throughout the paper we consider $S$ to be given which satisfies the following assumption.
\begin{assumption}
\label{as:gen}
$\cS : \Rpz \rightarrow \R$ is continuous, strictly convex and superlinear, in the sense that $\lim_{\rho\to+\infty}S(\rho)/\rho=+\infty$. Furthermore, $\cS$ is twice continuously differentiable in $(0,+\infty)\setminus\{1\}$. 
\end{assumption}

Let us notice that in this manuscript the internal energy part of the functional $\cE$ in general will satisfy the well-known condition introduced by McCann (\cite{mccann}), so it will be displacement convex. But this energy fails to be differentiable on $(\sP(\Om),W_2)$. Furthermore, in general we do not impose $\l$-convexity assumptions on the potential $\Phi$ (so the potential energy in general fails to be displacement $\l$-convex). Because of these two deficiencies, the classical results from \cite{AmbGigSav} do not apply directly in our setting. The lack of geodesic $\l$-convexity in the context of Wasserstein gradient flows typically poses serious obstructions (as we can see for instance in \cite{DifMat14,MatMcCSav,KimMes18}). Even though the existence of the gradient flow of $\cE$ in $(\sP(\Om),W_2)$ is expected, the fine characterization of the density curves, their velocities and the {\it critical region} $\{\rho=1\}$, in as general settings as possible, is a challenging task. Because of the same reasons, an approach by maximal monotone operators as in \cite{BarRoc18} would not be satisfactory in our setting either. In this context, ours seems to be the first contribution which gives fine characterization of the gradient flows of a general class of non-differentiable internal energies in $(\sP(\Om),W_2)$.

\medskip

In our analysis, we rely on the classical {\it minimizing movements} scheme of De Giorgi (see also \cite{JKO} and \cite{San:17-GF}). This, for a given $\rho_0\in\sP(\Om)$ (and for a small parameter $\t>0$ and $N\in\N$ such that $N\t=T$) iteratively constructs $\left(\rho_k\right)_{k=0}^N$ as
\begin{equation}\label{eq:MM_intro}
\rho_{k+1}={\rm{argmin}}\left\{\cE(\rho)+\frac{1}{2\t}W_2^2(\rho_k,\rho):\ \rho\in\sP(\Om)\right\},\ k\in\{0,\dots,N-1\}.
\end{equation}

\medskip

In order to write down the first order necessary optimality conditions associated to \eqref{eq:MM_intro}, in Section \ref{sec:MM_opt} as our first contribution in this paper, we give a precise characterization of the subdifferential of $\cE$ in $(\sP(\Om),W_2)$ (cf. \cite{AmbGigSav}) in various settings (depending on the growth condition of $S$ and the summability of $\rho_0$). Our analysis in this section relies on classical results from convex analysis, carefully adapted to \eqref{eq:MM_intro}. As an intermediate result, we show (see Lemma \ref{lem:apriori-2star}) that optimizers of the problem \eqref{eq:MM_intro} enjoy higher summability estimates than the a priori ones coming from the growth condition of $S$ at $+\infty$.

In order to give a precise description of the optimality conditions associated to \eqref{eq:MM_intro}, we introduce a function $p_k$ which encodes the `transition' between the phases $\{\rho_k<1\}$ and $\{\rho_k>1\}$ through the {\it critical region} $\{\rho_k=1\}$. This is very much inspired by the derivation of the {\it pressure variable} in recent models studying crowd movements under density constraints (see in \cite{MauRouSan1}, \cite{DimMauSan16}, \cite{MesSan}). Because of this similarity, throughout the paper, we sometimes use the abused terminology of {\it pressure} to refer to the variable $p$. Interestingly, numerical experiments suggest (see Figure~\ref{fig:jko}) that the critical region emerges in general already after one minimizing movement iteration.

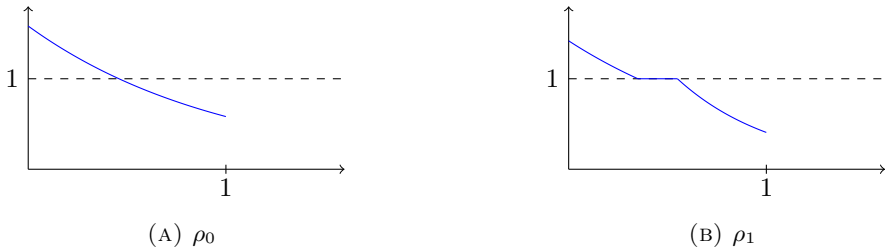
\begin{figure}[h]
	\centering
	\begin{subfigure}[t]{0.3\textwidth}
\begin{tikzpicture}[xscale=1.3,yscale=0.6]
\draw [scale=2,->] (0,0) -- (1.6,0);
\draw [scale=2,->] (0,0) -- (0,1.8);
\draw [scale=2,dashed] (0,1) -- (1.6,1);
\draw[scale=2,domain=0:1,smooth,variable=\x,blue] plot ({\x},{exp(-\x)/(1-exp(-1))});
\node [left] at (0,2) {$1$};
\node[below] at (2,0){$1$};
\draw[scale=2] (1,-0.05) -- (1,0.05);
\end{tikzpicture}
		\caption{$\rho_0$}
	\end{subfigure}
	\hspace{2cm}
	\begin{subfigure}[t]{0.3\textwidth}
\begin{tikzpicture}[xscale=1.3,yscale=0.6]
\draw [scale=2,->] (0,0) -- (1.6,0);
\draw [scale=2,->] (0,0) -- (0,1.8);
\draw [scale=2,dashed] (0,1) -- (1.6,1);
\draw[scale=2,domain=0:0.35,smooth,variable=\x,blue] plot ({\x},{exp(0.35-\x)});
\draw[scale=2,domain=0.35:0.55,smooth,variable=\x,blue] plot ({\x},{1});
\draw[scale=2,domain=0.55:1,smooth,variable=\x,blue] plot ({\x},{exp(1.1-2*\x)});
\node [left] at (0,2) {$1$};
\node[below] at (2,0){$1$};
\draw[scale=2] (1,-0.05) -- (1,0.05);
\end{tikzpicture}		
		\caption{$\rho_1$}
	\end{subfigure}
	\caption{One minimizing movement step in 1D, for $\Phi(x) = 2x$, $\Omega=[0,1]$ and $S$ in \eqref{S:log_intro}}
	\label{fig:jko}
\end{figure}

After obtaining the necessary compactness results, we pass to the limit with the time discretization parameter $\t\da 0$ and we recover a PDE (which precisely describes the weak distributional solutions of \eqref{dd}) satisfied by the limit quantities $(\rho,p)$. This formally reads as 
\be\label{eq:main1}
\left\{
\ba{ll}
\partial_t\rho - \Delta (L_S(\rho,p)) -\nabla\cdot(\nabla\Phi\rho) =0, & {\rm{in}}\ (0,T)\times\Om,\\
\rho(0,\cdot)=\rho_0, & {\rm{in}}\ \Om,\\
(\nabla (L_S(\rho,p) )+\nabla\Phi\rho ) \cdot \n = 0 , & {\rm{in}}\ [0,T] \times \partial \Om.
\ea
\right.
\ee
Here, the operator $L_S$ is defined pointwisely for functions $(\rho,p):[0,T]\times \Om\to\R$ by
\begin{align}
\label{eq:ls}
L_\cS(\rho,p)(t,x):= \left[ \rho(t,x) \cS'(\rho(t,x)) - \cS(\rho(t,x)) + \cS(1) \right] \one_{ \{ \rho \neq 1 \} }(t,x) + p(t,x) \one_{\{ \rho=1 \}}(t,x)
\end{align}
and the pressure variable $p : [0,T]\times \Om \to \R$ satisfies
\begin{align}
\label{eq:main2}
\begin{cases}
p=\cS'(1-) &\hbox{ {\rm{if}} } 0 \leq \rho < 1,\\
p \in [\cS'(1-),\cS'(1+)] &\hbox{ {\rm{if}} } \rho = 1,\\
p=\cS'(1+) &\hbox{ {\rm{if}} } \rho > 1.
\end{cases}
\end{align}
Formally, \eqref{eq:main1} and \eqref{eq:main2} correspond to the {\it three phase free boundary problem}
$$\Delta p=-\Delta\Phi,\ {\rm{in}}\ \{\rho=1\},\ \ \ p=\cS'(1-)\ {\rm{in}}\ \{\rho<1\}\ {\rm{and}}\ p=S'(1+)\ {\rm{in}}\ \{\rho>1\}.$$

Throughout the paper we distinguish cases depending on the diffusion rates in the two phases $\{\rho<1\}$ and $\{\rho>1\}$. We consider the combination of linear and porous medium type diffusions, which correspond to a behavior as $S(\rho)\sim\rho\log(\rho)$ and $S(\rho)\sim\rho^m$ (for $m>1$), in $\{\rho<1\}$ and $\{\rho>1\}$. So, typical examples we have in mind include 
\begin{align*}
\cS(\rho):= \begin{cases}
\rho \log\rho, &{\rm{if\ }} \rho \in [0,1],\\
\frac{\rho^m}{m-1}-\frac{1}{m-1}, &{\rm{if\ }} \rho\in(1,+\infty),
\end{cases}
\ \ {\rm{or}}\ \ 
\cS(\rho):= \begin{cases}
\frac{\rho^m}{m-1}, &{\rm{if\ }} \rho \in [0,1],\\
\frac{\rho^r}{r-1}-\frac{1}{r-1}+\frac{1}{m-1}, &{\rm{if\ }} \rho\in(1,+\infty),
\end{cases}
{\rm{for}}\ m>r>1.
\end{align*}
Energies of only logarithmic type or power like ones with the same power on both phases will also be considered (as in \eqref{S:log_intro} or \eqref{S:m_intro}). The analysis in the case of general energies is quite involved. In the same time, optimizers in the minimizing movements scheme possess different characteristics in the case of logarithmic type and porous medium type internal energies (such as fully supported vs. not fully supported; Lipschitz continuous vs. not Lipschitz continuous, etc.). As a result of this, we have to use different arguments to obtain the needed estimates. Therefore, to keep the paper as much readable as possible, we carefully break the cases (depending on the behavior of the internal energies) into specific sections.

In order to emphasize the main ideas of the paper, we present two toy problems in details. These turn out to be building blocks of our analysis for more general cases. Section \ref{sec:gfl} is devoted to the case when the entropy is of logarithmic type on both phases $\{\rho<1\}$ and $\{\rho>1\}$ and in particular $S$ is given by
\begin{align}\label{S:log_intro}
\cS(\rho):= \begin{cases}
\rho \log\rho, &{\rm{if\ }} \rho \in [0,1],\\
2\rho\log\rho, &{\rm{if\ }} \rho\in(1,+\infty).
\end{cases}
\end{align}
In this case, it turns out that the solution $(\rho,p)$ satisfies $p = 1$ in $\{ \rho < 1\}$, $p \in [1,2]$ in $\{\rho=1\}$, $p = 2$ in $\{ \rho > 1\}$ and we have the simplified expression $L_S(\rho, p) = p \rho$.

Similarly, Subsection \ref{sec:51} presents the analysis in the case when $S$ is given by
\begin{align}\label{S:m_intro}
\ds\cS(\rho):= \begin{cases}
\ds\frac{\rho^m}{m-1}, &\hbox{ for } \rho \in [0,1],\\[5pt]
\ds\frac{2 \rho^m}{m-1} - \frac{1}{m-1} , &\hbox{ for } \rho\in(1,+\infty),
\end{cases}
\end{align}
for some $m>1$. For this energy, the first equation of \eqref{eq:main1} can be written as 
\begin{align*}
\partial_t\rho - \nabla \cdot \left(\rho \left[\nabla \left( \rho^{m-1} p \right)+\nabla\Phi\right] \right) =0, \ \ & {\rm{in}}\ (0,T)\times\Om.
\end{align*}
Furthermore, $p=\frac{m}{m-1}$ in $\{ \rho < 1\}$, $p \in [\frac{m}{m-1},\frac{2m}{m-1}]$ in $\{\rho=1\}$ and $p = \frac{2m}{m-1}$ in $\{ \rho > 1\}$.

\medskip

Starting with Section \ref{sec:gfn}, we consider general entropies. Assumptions are made on the growth of $S$ in the two different phases $\{\rho<1\}$ and $\{\rho>1\}$. First, we impose   
\begin{assumption}\label{as:intro_main1}
\begin{align}
\label{eq:g1}
\cS : \Rpz \rightarrow \R {\rm{\ satisfies\ }} \frac{\rho^{m-2}}{\s_2} < \cS''(\rho)  &{\rm{\ if}}\ \rho\in(0,1) {\rm{\ for\ some\ }} m \geq 1 {\rm{\ and\ }} \s_2>0.
\end{align}
\end{assumption}

The imposed summability assumption on the initial data $\rho_0\in\sP(\Om)$  plays also a crucial role in our analysis. If $\rho_0\in L^\infty(\Om)$, it turns out that the entire iterated sequence $(\rho_k)_{k=1}^N$ obtained in the scheme \eqref{eq:MM_intro} remains essentially uniformly bounded, provided the potential $\Phi$ is regular enough. This fact does not depend on the differentiability of $S$ and it is well-known in the literature (see \cite{OTAM}). In this case,  imposing only the assumption \eqref{eq:g1} on $S$ is enough to obtain the well-posedness of \eqref{eq:main1}-\eqref{eq:main2}.

The other `extreme' case is when we only impose that $\rho_0$ has finite energy, i.e. $\cE(\rho_0)<+\infty.$ We show that the iterated sequence will have improved summability estimates for $k\in\{1,\dots,N\}$ (see in Lemma \ref{lem:apriori-2star}), provided $S$ satisfies the additional growth condition \eqref{eq:1as22}-\eqref{eq:2as22} below. These summability estimates on the iterated sequence will be enough to obtain the necessary a priori estimates and pass to the limit as $\t\da 0$ to obtain a weak solution to \eqref{eq:main1}-\eqref{eq:main2}.

As a consequence of these arguments, we will always distinguish two cases with respect to the previous two summability assumptions when stating our main results. Our main result in the case of $\rho_0\in L^\infty(\Om)$ reads as:

\begin{theorem}[Theorems \ref{thm:exi}, \ref{thm:gm}, \ref{thm:exip}, \ref{thm:gmpg} and Theorem \ref{thm:L1contr}]
\label{thm:main1}
Suppose that Assumptions \ref{as:gen}-\ref{as:intro_main1} hold and $\Phi$ satisfies \eqref{as:Phi2}. For $\rho_0 \in L^\infty(\Om)$, there exists $\rho\in L^{\infty}([0,T]\times\Om)$, $\rho^{m}\in L^2([0,T];H^1(\Om))$ and $p\in L^2([0,T]; H^1(\Om))\cap L^\infty([0,T]\times\Om)$ such that $(\rho,p)$ is a unique solution of \eqref{eq:main1}-\eqref{eq:main2} in the sense of distributions. 
\end{theorem}

\medskip

For general initial data such that $\cE(\rho_0) < +\infty$ we shall impose the following additional growth condition on $S$. 
\begin{assumption}\label{as:intro-19}
\begin{subequations}
\label{eq:g2}
\begin{align}
\label{eq:2as22}
\cS : \Rpz \rightarrow \R {\rm{\ satisfies\ }} \frac{\rho^{r-2}}{\s_1} \leq \cS''(\rho)  &{\rm{\ if\ }} \rho\in(1,+\infty) {\rm{\ and\ }}\\
\label{eq:1as22}
\cS''(\rho) \leq \s_1 \rho^{r-2} &{\rm{\ if\ }} \rho\in(1,+\infty) {\rm{\ for\ some\ }} r, \s_1 \geq 1.
\end{align}
\end{subequations}
\end{assumption}
Notice that under \eqref{eq:g2} and $r>1$, $\cE(\rho_0) < +\infty$  is equivalent to $\rho_0 \in L^r(\Om)$. Similarly to Theorem \ref{thm:main1}, we can formulate the corresponding well-posedness result.

\begin{theorem}[Theorems \ref{thm:exi}, \ref{thm:gm}, \ref{thm:exip}, \ref{thm:gmpg} and Theorem \ref{thm:L1contr}]\label{thm:main2}
Suppose that Assumptions \ref{as:gen}, \ref{as:intro_main1} and \ref{as:intro-19} are fulfilled and
\begin{align}\label{eq:1main2}
m <  r + \frac{\beta}{2}
\end{align} 
hold true for $\beta>1$ (its precise value is given in \eqref{eq:beta}). For $\rho_0 \in \sP(\Om)$ such that $\cE(\rho_0) < +\infty$, there exists $\rho\in L^{\beta}([0,T]\times\Om)$ and $p\in L^2([0,T]; H^1(\Om))\cap L^\infty([0,T]\times\Om)$ such that $(\rho,p)$ is a solution of \eqref{eq:main1}-\eqref{eq:main2} in the sense of distributions.  Furthermore, we have
\begin{align*}
\rho^{m - \frac12} \in L^2([0,T]; H^1(\Om)), {\rm{\ if\ }} m \leq r {\rm{\ and\ }} \rho^{m-\frac12} \in L^q([0,T]; W^{1,q}(\Om)) {\rm{\ if\ }} r < m <  r + \frac{\beta}{2}
\end{align*}
for some $q \in (1,2)$. If in addition $\beta \geq 2r$, then the pair $(\rho,p)$ is unique.
\end{theorem}

Let us comment on the additional technical assumption \eqref{eq:1main2} in the previous theorem. While this condition has to be required for purely technical reasons and we do not claim anything about its sharpness, we believe that it could be physically motivated. This would just mean that for unbounded initial data, the diffusion rate on the region $\{\rho<1\}$ cannot be `too much slower' than the one on the region $\{\rho>1\}.$  With other words, `too fast' diffusion rate on $\{\rho>1\}$ and `too slow' diffusion on $\{\rho<1\}$ might result in unphysical phenomena and in non-existence of solutions. 

It worth also noticing that the previous phenomenon is not expected for bounded solutions. Also, in particular from the definition of $\b$ in \eqref{eq:beta}, we see that  $\b<+\infty$ can be arbitrary large if $d=2$ and $\b=+\infty$, if $d=1$. Therefore, in such cases the previous theorem holds true without the additional assumption \eqref{eq:1main2}. The same is true in the case when $1\le m\le r.$

\medskip

Let us make a brief comment also on the proof of the previous theorems. In the case when the diffusion rates are equal on the two phases $\{\rho<1\}$ and $\{\rho>1\}$, i.e. $m=r$, the derivation of the optimality conditions already gives us enough a priori estimates on gradients of suitable powers of the density variable. Then, these are enough to obtain the strong compactness of the interpolated curves of the discrete in time densities and pass to the limit as $\t\da 0$. The situation is way more challenging in the case when $m\neq r$. In these situations, we actually obtain the required estimates on the gradients of the discrete in time densities raised on a carefully chosen `intermediate' power (depending on both $m$ and $r$). This idea seems to be crucial in our analysis and this is one of the most technical parts of the paper. 

\medskip

It worth to comment also on the fact that in Theorem \ref{thm:main2} we obtain improved summability estimates of the density variable, even if one merely imposes $L^r$ summability on $\rho_0$ and the diffusion rate in $\{\rho>1\}$ is $r$, we obtain $\rho\in L^\beta([0,T]\times\Om)$ (and $\b$ given in \eqref{eq:beta} satisfies $\b>r$; in particular $\b<+\infty$ is arbitrary large for $d=2$ and $\b=+\infty$ for $d=1$). This improved summability estimate (with respect to the summability of the initial data) seems to be well-known in the case of standard porous medium equations (for instance in the case of $\Phi=0$, this is a consequence of \cite[Theorem 8.7]{Vaz07}). Our proof, which is based on purely optimal transport techniques, implies this estimate in particular also in the classical porous medium equation.

\medskip

When studying the well-posedness of the system \eqref{eq:main1}-\eqref{eq:main2}, one can ask the natural question whether these PDEs can be represented as continuity equations. Under suitable additional assumptions, this is always the case, as we can show in Theorem~\ref{cor:modg}, Theorem~\ref{cor:emodp} and Theorem~\ref{cor:gmpgg} when \eqref{eq:main1} also reads as 
\be\label{eq:main3}
\left\{
\ba{ll}
\partial_t\rho - \nabla \cdot \left(\rho \nabla \left( \cS'(\rho)\one_{ \{ \rho \neq 1 \} } + p\one_{\{ \rho=1 \}} \right) \right) -\nabla\cdot(\rho\nabla\Phi) =0, & {\rm{in}}\ (0,T)\times\Om,\\
\rho(0,\cdot)=\rho_0, & {\rm{in}}\ \Om,\\
\rho\left[ \nabla \left( \cS'(\rho)\one_{ \{ \rho \neq 1 \} } + p\one_{\{ \rho=1 \}} \right) +\nabla\Phi\right]\cdot \n = 0 , & {\rm{in}}\ [0,T] \times \partial \Om.
\ea
\right.
\ee
We underline that the required additional assumptions are needed to guarantee Sobolev estimates on $S'(\rho)$. We can summarize our results in this direction as follows. 
\begin{theorem}[Theorems~\ref{cor:modg} and \ref{cor:gmpgg}]
\label{thm:main3}
Let us suppose that we are in the setting of Theorem~\ref{thm:main2} and  $(\rho,p)$ is the solution of\eqref{eq:main1}-\eqref{eq:main2}. If we additionally assume
\begin{align}
\label{eq:1main3}
m < r+\frac{1}{2}
\end{align}
and
\begin{align}
\label{eq:2main3}
\beta > 2 {\rm{\ and\ }} m < \frac{\beta}{2} + \frac12, 
\end{align}
then $(\rho,p)$ is a weak solution of \eqref{eq:main3} in the sense of distribution. The uniqueness of the solution holds under the same assumption as in Theorem~\ref{thm:main2}. If in addition $\rho_0\in L^\infty(\Om)$ and $\Phi$ satisfies \eqref{as:Phi2}, we can drop \eqref{eq:2main3} from the statement.
\end{theorem}
In the same way as in Theorem \ref{thm:main2} (by the definition of $\beta$ in \eqref{eq:beta}), \eqref{eq:2main3} holds for any $m, r\geq 1$ if $d=1$ or $d=2$. Moreover, when $r=m$, then the second inequality in \eqref{eq:2main3} is satisfied for all $m\geq1$ and $\beta >2$ is equivalent to $m > \frac{3d-4}{2d}$. 

\medskip

The attentive reader could observe that in the statements of Theorem \ref{thm:main1} and \ref{thm:main2} we included the corresponding uniqueness results as well. Indeed, Section \ref{sec:uniqueness} is entirely devoted to this issue and in particular we obtain an $L^1$ contraction result for the density variable $\rho$ (see in Theorem \ref{thm:L1contr}), implying its uniqueness. This will then imply the uniqueness of the corresponding $p$ variable as well. Our approach is inspired by \cite[Section 3]{DiMMes} and  \cite[Theorem 6.5]{Vaz07}, and as expected, the {\it monotonicity} of the operator $L_S$ (see Lemma \ref{lem:mon}) plays a crucial role in our argument. By the `double degeneracy' of our problem, neither of the previously mentioned two approaches apply directly and a very careful combination of the two is required to obtain the desired $L^1$ contraction. Similarly as in \cite[Theorem 6.5]{Vaz07}, in this analysis an additional summability assumption is needed on the density variable. Due to the extra $L^\beta$ summability obtained in Theorem \ref{thm:main2} or in the case of bounded solutions as in Theorem \ref{thm:main1}, this is automatically fulfilled in many cases. Let us mention that we expect a $W_2$-type contraction argument (in the spirit of \cite{CarMcCVil03,CarMcCVil06}) to hold in our setting as well, provided we impose some convexity assumptions on the potential $\Phi$. The results from \cite{BolCar} imply that the Wasserstein contraction is equivalent to the geodesic convexity of the  energy, provided the energy is smooth. Since the internal energies considered in this article, even though geodesically convex, in general fail to be differentiable, the results from \cite{BolCar} do not apply directly in our setting. These investigations represent the subject of future study. Finally, since in general we do not impose ($\l$-)convexity assumptions on $\Phi$, our energy will in general lack the displacement ($\l$-)convexity property, which also motivates our search for an $L^1$-contraction instead.

\medskip

Section \ref{sec:disc} is devoted to further discussions on the models studied in this paper. In particular, we discuss examples where the emergence of the critical region $\{\rho=1\}$ can be observed for positive times, even if that was not present in the case of the initial data, i.e. $\sL^d(\{\rho_0=1\})=0.$ We illustrate this in dimension one and we describe stationary solutions (minimizers of the free energy) corresponding to suitable potential functions $\Phi$, where the critical region is present. As we mentioned before, our problems can be linked to three phase free boundary problems, and in this section we also derive these ones formally.

\medskip

We end our paper with three small appendices, where we collected some well-known facts (or consequences of well-known results) from the theory of optimal transport and convex analysis. Here, we present also a suitable version of the classical Aubin-Lions lemma, which is repeatedly used throughout the paper to obtain compactness of families of time dependent functions in Lebesgue spaces.

\medskip

\medskip

\section{The minimizing movements scheme, optimality conditions and properties of the energy}\label{sec:MM_opt}

Throughout the paper $\Om\subset\R^d$ is given, as the closure of a bounded, convex open set with smooth boundary. $\sP(\Om)$ denotes the space of Borel probability measures on $\Om$ and $\sL^d$ stands the Lebesgue measure on $\R^d$. We also use the notation $\sP^{\rm{ac}}(\Om):=\left\{\mu\in\sP(\Om):\ \mu\ll\sL^d\mres\Om \right\}.$ $T>0$ is a fixed time horizon and we often use the notations $Q:= [0,T] \times \Omega$ and $\R^+:=(0,+\infty)$.

As $\cS'$ is strictly increasing in $\Ro$ from Assumption~\ref{as:gen}, $\cS'(0+)$ and $\cS'(1\pm)$ are well-defined in $\Rinfm$ and $\R$, respectively, as follows.
\begin{align}
\label{eq:slr}
\cS'(0+) := \lim_{\e \to 0+} \cS'(\e),\quad \cS'(1-) := \lim_{\e \to 1-} \cS'(\e) \hbox{ and } \cS'(1+) := \lim_{\e \to 1+} \cS'(\e).
\end{align}
In particular, we have that $S'(1-)\leq S'(1+).$

We define the corresponding internal energy $\cJ : \sP(\Om) \rightarrow \Rinf$ by 
\begin{align}
\label{eq:j}
\cJ(\rho) := 
\begin{cases}
\ds\int_\Om \cS(\rho(x))\dd x\quad &\hbox{ if } \rho \in \sP^{\rm{ac}}(\Om),\\ 
+\infty &\hbox{ otherwise. }
\end{cases}
\end{align}
Furthermore, we suppose that there is given $\Phi:\Om\to\R$ a 
potential function  in $W^{1,\infty}(\Om)$ and the associated potential energy $\cF:\sP(\Om)\to\R$ given by
$$
\cF(\rho):=\int_\Om\Phi(x)\dd\rho(x).
$$

Let $\rho_0\in\sP(\Om)$ be given and consider a time discretization parameter $\t>0$ and $N\in\N$ such that $N\t = T$. We define the {\it minimizing movements} $(\rho_{k})_{k=1}^N$ 
of $\cJ+\cF$ as follows: for  $k \in \{1,\dots,N\}$ set,
\begin{align}\label{eq:step}
\rho_{k} := \argmin\limits_{\rho\in \sP(\Om)}\left\{\cJ(\rho)+\cF(\rho)+\frac{1}{2\t}W_2^2(\rho,\rho_{k-1})\right\}.
\end{align}
Note that the existence and uniqueness of the solutions in the minimization problems \eqref{eq:step} follow from standard compactness, lower semicontinuity and convexity arguments (similarly as in \cite[Proposition 8.5]{OTAM}, for instance).

In what follows, in our analysis we differentiate two cases with respect to the summability assumption on $\rho_0$. Since these need slightly different arguments, we separate them in two different subsections. In particular, if one assumes $L^\infty$ summability on $\rho_0$, the presented results will hold true under no additional assumptions on $S$ (other than in Assumption \ref{as:gen}). However, in \eqref{eq:step} we can allow general measure initial data, in which case an additional growth condition (see \eqref{eq:g2}) has to be imposed on $S$ in order to obtain the same optimality conditions.

\subsection{Optimality conditions for $\rho_0\in L^\infty(\Om)$}

\begin{lemma}\label{lem:Linfty}
Suppose that Assumption \ref{as:gen} takes place and $\rho_0\in L^\infty(\Om)$. If $\Phi$ is non-constant, let us assume that $\Phi\in C^1(\ov\Om)$ and
\begin{align}\label{as:Phi2}
\nabla\Phi(x_0)\cdot\n(x_0)>0, \ \ \forall x_0\in\partial\Om \ \ {\rm{and}}\ \ \nabla\Phi\in BV(\Om;\R^d)\ {\rm{with}}\ [\Delta\Phi]_+\in L^\infty(\Om)
\end{align}
where $\n$ stands for the outward normal vector to $\partial\Om$ and $[\Delta\Phi]_+$ denotes the positive part of the measure $\Delta\Phi$. Let $(\rho_k)_{k=1}^N$ be constructed via the scheme \eqref{eq:step}. Then we have
$$\|\rho_k\|_{L^\infty}\le\|\rho_{k-1}\|_{L^\infty}\left(1+\t\|[\Delta\Phi]_+\|_{L^\infty}\right)^d\le \|\rho_0\|_{L^\infty}\left(1+\t\|[\Delta\Phi]_+\|_{L^\infty}\right)^{kd}\le \|\rho_0\|_{L^\infty} e^{dT\|[\Delta\Phi]_+\|_{L^\infty}},$$ 
$\forall k\in\{1,\dots, N\}.$
\end{lemma}

\begin{remark}
Let us notice that the second part of assumption \eqref{as:Phi2} is sharp and it is very much related to the ones imposed in the work of Ambrosio (see \cite{Amb}), as an improvement of the classical DiPerna-Lions theory (\cite{DipLio}), on transport equations with $BV$ vector fields. 
\end{remark}

\begin{proof}[Proof of Lemma \ref{lem:Linfty}]
The proof of this result in the case when $\Phi\equiv 0$ is essentially the same as the proof of \cite[Proposition 7.32]{OTAM} (since that proof is not assuming any differentiability on $S$).

\medskip

For general $\Phi$, we use some ideas from the proof of \cite[Theorem 1]{CarSan18}. Let us approximate $S$ with a sequence $(S_\e)_{\e>0}$ of smooth convex functions such that $S_\e ''\ge c_\e>0$ for any $\e>0$ with $S_\e'(0+)=-\infty.$ Let $\Phi_\e$ be a smooth approximation of $\Phi$ which satisfies \eqref{as:Phi2} and such that $\Phi_\e\to\Phi$, $\nabla\Phi_\e\to\nabla\Phi$, uniformly as $\e\da 0$ and $\|[\Delta\Phi_\e]_+\|_{L^\infty}\le \|[\Delta\Phi]_+\|_{L^\infty}$, for $\e>0$. Let $\rho_k^\e$ be the unique solution of \eqref{eq:step}, when we replace $S$ with $S_\e$ and $\Phi$ by $\Phi_\e$. Writing down the optimality conditions we obtain 
$$S_\e'(\rho_k^\e)+\Phi_\e+\frac{\phi_k^\e}{\t}=C \ \ae,$$
where $\phi_k^\e\in\cK(\rho_k^\e,\rho_{k-1}).$ Let us suppose that $\phi_k^\e\in C^{2,\alpha}(\Om)$, otherwise we approximate $\rho_{k-1}$ by strictly positive $C^{0,\alpha}$ measures (and $\rho_k^\e$ is Lipschitz continuous and strictly positive), and we use Caffarelli's regularity theory to deduce the desired regularity for the potential.

Now, let $x_0$ a maximum point of $\rho_k^\e$. From the previous equality, since $S_\e'$ is strictly increasing, we certainly have that $x_0$ is a minimum point of $\Phi_\e+\frac{\phi_k^\e}{\t}$.

We claim that $x_0\notin\partial\Om$. Indeed, if $x_0$ would belong to $\partial\Om$, we would have that 
$$(\nabla\phi_k^\e(x_0)+\t\nabla\Phi_\e(x_0))\cdot\n(x_0)\le 0.$$ However, by the convexity of $\Om$, we have that $(x_0-\nabla\phi_k^\e(x_0))\cdot \n(x_0)\le 0$, from where $\nabla\phi_k^\e(x_0)\cdot \n(x_0)\ge 0$. This fact together with the assumption \eqref{as:Phi2} yields a contradiction. Indeed, from the uniform convergence of $\nabla\Phi_\e\to\nabla\Phi$, we have that 
$$\nabla\Phi_\e(x_0)\cdot\n\ge\nabla\Phi(x_0)\cdot\n-\e>0,$$
for sufficiently small $\e>0$.

Therefore, the maximum point $x_0$ of $\rho_k^\e$ belongs to the interior of $\Om$. This implies that $\Delta\phi_k^\e(x_0)+\t\Delta\Phi_\e(x_0)\ge 0.$ Using the Monge-Amp\`ere equation we find
\begin{align*}
\|\rho_k^\e\|_{L^\infty}&=\rho_k^\e(x_0)=\rho_{k-1}(x_0-\nabla\phi_k^\e(x_0))\det\left(I_d-D^2\phi_k^\e(x_0)\right)\le\|\rho_{k-1}\|_{L^\infty}(1-\Delta\phi_k^\e(x_0))^d\\
&\le \|\rho_{k-1}\|_{L^\infty}(1+\t\Delta\Phi_\e(x_0))^d\le \|\rho_{k-1}\|_{L^\infty}(1+\t\|[\Delta\Phi_\e]_+\|_{L^\infty})^d\le \|\rho_{k-1}\|_{L^\infty}(1+\t\|[\Delta\Phi]_+\|_{L^\infty})^d\\
&\le\|\rho_{0}\|_{L^\infty}(1+\t\|[\Delta\Phi]_+\|_{L^\infty})^{kd}\le \|\rho_0\|_{L^\infty} e^{dT\|[\Delta\Phi]_+\|_{L^\infty}},
\end{align*}
where in the first inequality we have used the inequality between the arithmetic and geometric means. Since the last three bounds depend only on the data, these will also remain valid also in the limit $\e\da 0$ (since the minimizers of both the approximated and the original problems are unique). Therefore the thesis of the lemma follows.
\end{proof}

Now, we state the main result of this subsection on the first order necessary optimality conditions for the problems in \eqref{eq:step}. 

\begin{theorem}
\label{thm:opg}
Suppose that $\rho_0\in L^\infty(\Om)$. For all $k \in \{1,\dots,N\}$, there exists $\cC=\cC(k) \in \R$ and $\ophi_{k} \in \cK(\rho_{k}, \rho_{k-1})$ such that
\begin{align}
\label{eq:1opg}
\begin{cases}
\cC - \frac{\ophi_{k}}{\tau} -\Phi \leq \cS'(0+)	&\hbox{ {\rm{in}} } \{\rho_{k} =0 \},
\\ \cC - \frac{\ophi_{k}}{\tau} -\Phi \in [\cS'(1-),\cS'(1+)]	&\hbox{ {\rm{in}} } \{ \rho_{k} =1 \},
\\ \cC - \frac{\ophi_{k}}{\tau} -\Phi = \cS' \circ \rho_{k} &\hbox{ {\rm{otherwise.}} }
\end{cases}
\end{align}
Here, $\cK(\rho_k,\rho_{k-1})$ is given in Definition~\ref{def:kan}. Also, $\cS'(0+)$ and $\cS'(1\pm)$ are given in \eqref{eq:slr}. Note that if if $S'(0+)=-\infty$, then $\rho_k > 0$ a.e. (see Lemma~\ref{lem:postg}), and in this case the first inequality in \eqref{eq:1opg} is not present.
\end{theorem}

The proof of the previous results relies on the precise derivation of the subdifferential of the corresponding objective functional in \eqref{eq:step}. Let us point out that the subdifferential of sum is not always the sum of subdifferentials (see for instance \cite[Example 7.22]{OTAM}). Therefore, we need to carefully choose the domain of definition of $\cJ$. In the spirit of Lemma \ref{lem:Linfty}, we consider it as a functional on $L^\infty(\Om)$ instead of $\sP(\Om)$. The additive property of subdifferentials on $L^\infty(\Om)$ holds under suitable conditions (cf. \cite{EkeTem}). 

\begin{remark}
Let us underline that in our analysis we rely on the classical subdifferential calculus in $L^p$ spaces rather than directly computing Wasserstein subdifferentials (cf. \cite{AmbGigSav}). This is mainly because of the already available powerful classical results on precise representations of subdifferentials of integral functionals, such as $\cJ$, on $L^p$ spaces (cf. \cite{Roc68, Roc71}). In the same time, this framework is well suited also for computing the subdifferential of $\rho\mapsto W_2^2(\cdot,\rho_{k-1})$ (see for instance in \cite{OTAM}). 

It worth mentioning, however, that at the heuristic level  there is an intimate link between the two notions of subdifferentials, namely: if $\xi\in\partial\cJ(\rho)$ (i.e. $\xi$ is an element of the classical $L^p$ subdifferential) is sufficiently regular, then $\nabla\xi\in\partial_{W_2}\cJ(\rho)$ (that is, $\nabla\xi$ is an element of the Wasserstein subdifferential). Nevertheless, this connection at this point remains only formal, because typically we do not have any a priori information on the regularity of $\xi$ to justify this link.
\end{remark}

\begin{proposition}
\label{prop:add} For all $k \in \{1,\dots, N\}$ we have 
\begin{align}
\label{eq:1add}
\left. \partial \left(\cJ(\rho)+\cF(\rho)+\frac{1}{2\t}W_2^2(\rho,\rho_{k-1}) \right)\right |_{\rho = \rho_{k}} = \partial \cJ(\rho_{k}) + \Phi + \frac{1}{2\t} \partial(W_2^2(\rho,\rho_{k-1})) |_{\rho = \rho_{k}}.
\end{align}
\end{proposition}

\begin{proof}
To simplify the writing, we consider only the case $k=1$. Let us check that $\cJ$ and $W_2^2(\cdot,\rho_{0})$ satisfy the assumptions in Lemma~\ref{lem:subsum}. The convexity of $\cS$ implies that of $\cJ$. Also, the continuity
of $\cJ$ in $L^\infty(\Om)$ follows from the continuity of $S$. From Lemma~\ref{lem:equ}, we conclude $\cJ \in \Gamma(L^\infty(\Om))$. We have the same conclusion for the functional $\cF$ (which is actually linear in $\rho$).

\medskip

Let us show that $W_2^2(\cdot,\rho_{0}) \in \Gamma(L^\infty(\Om))$. Define $H : L^1(\Om) \rightarrow \Rinf$ by
\begin{align}\label{def:H}
H(\phi):= -\int_{\Om} \phi^c \dd\rho_0.
\end{align}
Proposition~\ref{prop:dual} implies that $H^* : L^\infty(\Om) \rightarrow \Rinf$ is given (in the sense of  \eqref{def:lt}) by
\begin{align*}
H^* = \frac{1}{2} W_2^2(\cdot,\rho_{0}) \hbox{ on } L^\infty(\Om).
\end{align*}
We conclude $W_2^2(\cdot,\rho_{0}) \in \Gamma(L^\infty(\Om))$.

\medskip

Lastly, choose $A\subseteq\Om$ a Borel set such that $\sL^d(A)\neq 1$ and define
\begin{align}
\label{eq:hmu}
\hmu := \frac{1}{\sL^d(A)}\one_{A}.
\end{align}
$\cJ(\hmu)$, $\cF(\hmu)$ and $W_2^2(\hmu,\rho_{0})$ are finite. Furthermore, by the continuity of $\cS$ in $\Rp$, $\cJ$ is continuous at $\hmu$. In the same way $\cF$ is also continuous at $\hmu$. Thus, we conclude \eqref{eq:1add} from Lemma~\ref{lem:subsum}.
\end{proof}

Next, let us find the subdifferential of $W_2^2(\cdot,\rho_{k-1})$. While this subdifferential is expected to be the set of Kantarovich potentials $\cK(\rho_k,\rho_{k-1})$, it is not straight forward to conclude about this as we consider the subdifferential for the functional on $L^\infty(\Om)$. We rely on the ideas from \cite[Proposition 7.17]{OTAM}, tailored to our setting.

\begin{lemma}
\label{lem:subdual}
\cite[Lemma~7.15]{OTAM}
Let $\fX$ be a Banach space and $H : \fX \rightarrow \Rinf$ be convex and lower semicontinuous. Set $H^*(y) = \sup\limits_{x \in \fX} \{\langle x,y \rangle_{\fX,\fX^*} - H(x)\}$. Then, we have
\begin{align}
\label{eq:subdual}
\partial H^*(y) = \arg\max\limits_{x \in \fX} \left\{\langle x,y \rangle_{\fX,\fX^*} - H(x)\right\}.
\end{align}
\end{lemma}

\begin{lemma}\label{lem:Hlsc}
$H:L^1(\Om)\to\R\cup\{+\infty\}$ given in \eqref{def:H} is convex and l.s.c.
\end{lemma}

\begin{proof}
The proof of convexity of $H$ is the same as in \cite[Proposition 7.17]{OTAM}, where one needs to change only the definition of $\vphi^c$ using essential infima.

Let us show now that $H$ is l.s.c. For this, let $\vphi\in L^1(\Om)$ and  $(\vphi_n)_{n\in\N}$ a sequence in $L^1(\Om)$ such that $\vphi_n\to\vphi$ strongly in $L^1(\Om)$ as $n\to+\infty$.

Notice first that by definition, 
$$-\vphi^c(y)\ge\vphi(y),\ {\rm{a.e.\ in\ }}\Om,$$
from where we have that $H(\vphi)>-\infty.$ 
Because of the strong $L^1$ convergence, we know that there exists a subsequence of $(\vphi_n)_{n\in\N}$ (that we do not relabel), which is converging pointwise a.e. in $\Om$ to $\vphi$. We shall work with this sequence from now on.

Writing the previous inequality for $\vphi^c_n$ and $\vphi_n$, we have that 
$$\liminf_{n\to+\infty}-\vphi^c_n(y)\ge\liminf_{n\to+\infty}\vphi_n(y)=\vphi(y),\ {\rm{a.e.\ in\ }}\Om,$$
where we used the fact that $\vphi_n(y)\to\vphi(y)$ a.e. in $\Om$, as $n\to+\infty$. 

Let us define $g:\Om\to\R\cup\{+\infty\}$ as $g(y):=\liminf_{n\to+\infty}-\vphi^c_n(y)$. Notice that this is measurable function. Indeed, $(-\vphi^c_n)_{n\in\N}$ is a sequence of measurable functions (infima of measurable functions), and using Fatou's lemma for the non-negative sequence of measurable functions $(-\vphi^c_n-\vphi_n)_{n\in\N}$, one concludes that $g$ is measurable and
$$\int_{\Om}\vphi(y)\rho_{0}(y)\dd y\le\int_\Om g(y)\rho_{0}(y)\dd y\le\liminf_{n\to+\infty} H(\vphi_n).$$

{\it Claim.} $\vphi(y)\le-\vphi^c(y)\le g(y)$ for a.e. $y\in\Om$.

{\it Proof of the claim.} Actually the first inequality was shown before, thus we show only the second one.
Thus, by Egorov's theorem, we have that for any $\d>0$ there exists a measurable set $B_\d\subseteq\Om$ such that $\sL^d(B_\d)<\d$ and $(\vphi_{n})_{n\in\N}$ converges uniformly to $\vphi$ as $n\to+\infty$ on $\Om\setminus B_\d$. Let us fix a small $\d>0$. We have furthermore that for any $\e>0$ there exists $N_\e\in\N$ such that 
$$\vphi(x)-\e\le\vphi_n(x)\le\vphi(x)+\e$$
for a.e. $x\in\Om\setminus B_\d$ and $n\ge N_\e$. Because of this, we have the following chain of inequalities for all $n\ge N_\e$
\begin{align*}
-\vphi^c_n(y)&=\sup_{x\in\Om}\left\{\vphi_n(x)-|x-y|^2\right\}\ge\sup_{x\in\Om\setminus B_\d}\left\{\vphi_n(x)-|x-y|^2\right\}\ge\sup_{x\in\Om\setminus B_\d}\left\{\vphi(x)-\e-|x-y|^2\right\}.
\end{align*}
Taking $\liminf_{n\to+\infty}$ of both sides, one obtains
$$g(y)\ge \sup_{x\in\Om\setminus B_\d}\left\{\vphi(x)-\e-|x-y|^2\right\}$$
for a.e. $y\in\Om$. By the arbitrariness of $\e$ and $\d$ (in this order), one gets that
$$g(y)\ge \sup_{x\in\Om}\left\{\vphi(x)-|x-y|^2\right\}=-\vphi^c(y),$$
as we claimed. 

Notice that we have proved the following: if $(\vphi_n)_{n\in\N}$ is converging to $\vphi$ in $L^1(\Om)$, then there exists a subsequence $(\vphi_{n_j})_{j\in\N}$  of the original sequence such that 
$$H(\vphi)\le\liminf_{j\to+\infty}H(\vphi_{n_j}).$$
This statement actually implies the l.s.c. of $H$ on the full sequence. Indeed, observe that by the definition of $\liminf$, there exists a subsequence $(\vphi_{n_k})_{k\in\N}$ of the original sequence such that 
$$\liminf_{n\to+\infty} H(\vphi_n)=\lim_{k\to+\infty} H(\vphi_{n_k}).$$
We have shown previously that there exists a subsequence $(\vphi_{n_{k_j}})_{j\in\N}$ of $(\vphi_{n_k})_{k\in\N}$ such that 
$$H(\vphi)\le\liminf_{j\to+\infty} H(\vphi_{n_{k_j}}).$$
On the other hand 
$$\liminf_{j\to+\infty} H(\vphi_{n_{k_j}})=\lim_{j\to+\infty} H(\vphi_{n_{k_j}})=\lim_{k\to+\infty} H(\vphi_{n_k})=\liminf_{n\to+\infty} H(\vphi_n),$$
thus the l.s.c. of $H$ follows.
\end{proof}

\medskip
 
\begin{proposition}
\label{prop:wul}
For all $k \in \{1,\dots, N\}$ we have 
\begin{align}
\label{eq:add21}
\frac{1}{2}\partial(W_2^2(\rho,\rho_{k-1})) {\big{|}}_{\rho = \rho_{k}} = \cK(\rho_{k}, \rho_{k-1}).
\end{align}
\end{proposition}

\begin{proof}
To simplify the notation, we set $k=1$. Recall from Proposition~\ref{prop:dual} that
\begin{align*}
\frac{1}{2}\partial(W_2^2(\rho,\rho_{0})) |_{\rho = \rho_{1}} = \partial H^*(\rho_{1})
\end{align*}
for $H$ given in \eqref{def:H}. From Lemma~\ref{lem:subdual} and Lemma~\ref{lem:Hlsc}, it holds that
\begin{align*}
\partial H^*(\rho_{1}) = {\rm{argmax}}_{\phi \in L^1(\Om)} \left\{\int_{\Om} \phi \dd\rho_{1} + \int_{\Om} \phi^c \dd\rho_{0}\right\}.
\end{align*}
From Definition~\ref{def:kan}, we conclude.
\end{proof}

Lastly, let us compute the subdifferential of $\cJ$ based on \cite{Roc71}. Before, we need the following preparatory result.

\begin{lemma}
\label{lem:roc}
\cite[Corollary 1B]{Roc71}
Let $\psi$ and $\Psi$ be given as in \eqref{eq:psi}. Assume that $\psi(\mu(x))$ is majorized by a summable function of $x$ for at least one $\mu \in L^\infty(\Om)$ and that $\psi^*(\zeta(x))$ is majorized by a summable function of $x$ for at least one $\zeta \in L^1(\Om)$. Then, an element $\xi \in L^\infty(\Om)^*$ belongs to $\partial \Psi(\mu)$ given in \eqref{eq:2psi} if and only if $\xi^{\rm{ac}}(x) \in \partial \psi(\mu(x))$ for a.e. $x\in \Om$  where $\xi^{\rm{ac}}$ is the absolutely continuous component of $\xi$, and the singular component $\xi^{\rm{s}}$ of $\xi$ attains its maximum at $\mu$ over $$\{ \nu \in L^\infty(\Om) : \Psi(\nu) < + \infty \}.$$
\end{lemma}

\begin{proposition}
\label{prop:sul}
For $\rho_{k}$ is given in \eqref{eq:step}, if $\xi \in \partial \cJ(\rho_{k}) \cap L^1(\Om)$, then it holds that
\begin{align}
\label{eq:1sul}
\xi \in
\begin{cases}
[-\infty,\cS'(0+)]	&{\rm{a.e.\ in}}\ \{ \rho_{k} = 0 \},
\\ [\cS'(1-),\cS'(1+)]	&{\rm{a.e.\ in}}\  \{ \rho_{k} = 1 \},
\\ \cS' \circ \rho_{k} &{\rm{a.e.\ in}}\  \{ \rho_{k} \neq 1 \}.
\end{cases}
\end{align}
\end{proposition}

\begin{proof}
Let us show that $\cS$ and $\cS^*$ satisfies assumptions on Lemma~\ref{lem:roc}. Let $\mu=\zeta=\frac{1}{\sL^d(\Om)}\one_{\Om}$, then $\cS(\mu)$ is finite, and thus in $L^1(\Om)$. On the other hand, as $\cS$ is superlinear, $\cS^* < +\infty$ in $\Rpz$. Therefore, for any constant $c \in \R$, $\cS^*(c) \in L^1(\Om)$.

\medskip

By Lemma~\ref{lem:roc}, $\xi^{\rm{ac}}(x) \in \partial \cS(\rho_{k}(x))$ a.e., where $\xi^{\rm{ac}}$ is the absolutely continuous part of $\xi$.  From the direct computation of $\partial \cS(\rho_{k}(x))$, we conclude that $\xi^{\rm{ac}}$ satisfies the right hand side of \eqref{eq:1sul}.
As $\xi \in L^1(\Om)$, the singular part of $\xi$ is zero, $\xi^{\rm{ac}} = \xi$ and we conclude \eqref{eq:1sul}.
\end{proof}

\begin{proof}[Proof of Theorem~\ref{thm:opg}]
We only consider the case that $k=1$. By the optimality of $\rho_1$ in \eqref{eq:step}, it holds that
\begin{align*}
0 \in \partial \left(\cJ(\rho_{1})+\cF(\rho_1)+\frac{1}{2\t}W_2^2(\rho_{1},\rho_{0}) \right). 
\end{align*}
From Proposition~\ref{prop:add} and Proposition~\ref{prop:wul}, there exists $\xi \in \partial \cJ(\rho_{1})$, $\ophi_{1} \in \cK(\rho_{1}, \rho_{0})$ and $\cC \in \R$ such that 
\begin{align*}
\xi + \frac{\ophi_{1}}{\tau}+\Phi - \cC = 0 \hbox{ a.e. on } \Om .
\end{align*}
As $\ophi_{1},\Phi \in L^1(\Om)$, $\xi \in \partial \cJ(\rho_{1}) \cap L^1(\Om)$, Proposition~\ref{prop:sul} implies \eqref{eq:1opg}.
\end{proof}

\subsection{Optimality conditions for $\rho_0\in\sP(\Om)$ having finite energy}
In this subsection we are imposing \eqref{eq:g2}.
Let us show first that $\cJ$ satisfying the additional assumption in \eqref{eq:g2} defines a continuous functional on $L^r(\Om).$ In the previous subsection, the continuity of $\cJ$ in $L^\infty(\Om)$ directly follows from the continuity of $S$.

\begin{lemma}
\label{lem:cj}
Let $\cJ$ be given in \eqref{eq:j} satisfying \eqref{eq:1as22}. Then $\cJ$ is continuous in $L^s(\Om)$ for all 
\begin{align}\label{eq:1cj}
s > r {\rm{\ if\ }} r=1,\ {\rm{and}}\ s \geq r\ {\rm{if}}\ r>1.  
\end{align}
\end{lemma}

\begin{proof}

From \eqref{eq:1as22}, there exists $c>0$ such that for all $\rho \in \Rpz$ (since $S$ is also continuous, hence uniformly bounded on $[0,1]$)
\begin{align}
\label{eq:cj01}
|S(\rho)| \leq c (\rho^{s}+1).
\end{align}
for all $s$ satisfying \eqref{eq:1cj}.

Consider a sequence $\{ \mu_i \}_{i \in \N} \subset L^s(\Om)$ such that
\begin{align}
\label{eq:cj10}
\mu_i \to \mu \hbox{ in } L^s(\Om) \hbox{ as } i \to \infty
\end{align}
These exists a subsequence $\{ \mu_{i_j} \}_{j \in \N} \subset \{ \mu_i \}_{i \in \N}$ such that
\begin{align}
\label{eq:cj11}
\mu_{i_j} \to \mu \hbox{ a.e.} \hbox{ as } j \to \infty.
\end{align}

\medskip

From \eqref{eq:cj01}, it holds that for all $j \in \N$
\begin{align}
\label{eq:cj12}
0 \leq c (|\mu_{i_j}|^s +1) - |S(\mu_{i_j})| \leq c (|\mu_{i_j}|^s +1) \pm S(\mu_{i_j}).
\end{align}
Let us apply Fatou's lemma into $c ( |\mu_{i_j}|^s +1) + S(\mu_{i_j})$. From \eqref{eq:cj10}, \eqref{eq:cj11} and the continuity of $S$, it holds that
\begin{align*}
\int_\Om c (|\mu(x)|^s +1) + S(\mu(x)) \dd x &\leq \liminf_{j \to \infty} \int_\Om  c (|\mu_{i_j}(x)|^s +1) + S(\mu_{i_j}(x))\dd x,\\ 
&\leq \int_\Om c (|\mu(x)|^s+1)\dd x + \liminf_{j \to \infty} \int_\Om  S(\mu_{i_j})\dd x.
\end{align*}
and we have
\begin{align*}
\cJ(\mu) \leq \liminf_{j \to \infty} \cJ(\mu_{i_j}).
\end{align*}
Similarly to the argument at the end of the proof of Lemma \ref{lem:Hlsc}, we conclude the lower semicontinuity along the full sequence, therefore
\begin{align}
\label{eq:cj14}
\cJ(\mu) \leq \liminf_{i \to \infty} \cJ(\mu_{i}).
\end{align}

\medskip

Applying Fatou's lemma again into $c (|\mu_{i_j}|^s +1) - S(\mu_{i_j})$, we get 
\begin{align}
\label{eq:cj15}
\cJ(\mu) \geq \limsup_{j \to \infty} \cJ(\mu_{i_j}),
\end{align}
and as before, we deduce the upper semicontinuity along the full sequence. Therefore \eqref{eq:cj14} and \eqref{eq:cj15} imply that $\cJ(\mu) = \lim\limits_{j \to \infty} \cJ(\mu_{i_j})$ and $\cJ$ is continuous in $L^s(\Om)$.
\end{proof}

In what follows, we show that the minimizers of the of the minimizing movements scheme \eqref{eq:step} enjoy higher order summability estimates (which are independent of $\rho_0$, but depend on $\t$). These will play a crucial role later when deriving the optimality conditions. 

\begin{lemma}\label{lem:apriori-2star}
Suppose that $S$ satisfies Assumption \ref{as:gen} and \eqref{eq:2as22}. Let $\rho_k\in\sP(\Om)$ be the minimizer in \eqref{eq:step}.  Then $\rho_k\in L^{\b}(\Om)$, where $\b:=(2r-1)d/(d-2)$, if $d\ge 3$. If $d=2$ then the statement is true for any $\b<+\infty$ and $\b=+\infty$ if $d=1$. In particular, there exists $C>0$ depending only on $\Om$, $\|\nabla\Phi\|_{L^\infty}$ and $\beta$ such that if $\b<+\infty$, then
\begin{align}\label{estim:Lbeta}
\int_{\Om}(\rho_k)^\b\dd x\le C + \frac{C}{\t^2}W_2^2(\rho_{k},\rho_{k-1}).
\end{align}
Otherwise, for $d=1$, 
$$
\|\rho_k\|_{L^\infty}\le C.
$$

\end{lemma}

\begin{remark}
Let us notice that the previous lemma gives an improvement on the summability of $\rho_k$. Indeed, in case when the internal energy is of logarithmic entropy type, we know a priori that $\rho_k\in L^1(\Om)$, while in the case of power like entropies, we have a priori $\rho_k\in L^r(\Om).$ In contrast to these, we clearly improve the summability exponents in both cases.
\end{remark}

\begin{proof}[Proof of Lemma \ref{lem:apriori-2star}]
For $\e>0$ let $S_\e:[0,+\infty)\to\R$ smooth, strictly convex such that $S_\e''\ge c_\e>0$ (for some $c_\e>0$), $S_\e'(0+)=-\infty$ and $S_\e\to S$ uniformly as $\e\to 0$.  Let $\rho_k^\e$ be the unique minimizer of the problem
\begin{equation}\label{eq:approx}
\inf_{\rho\in\sP(\Om)}\left\{\cE_\e(\rho):=\int_{\Om}S_\e(\rho)\dd x+\cF(\rho)+\frac{1}{2\t}W_2^2(\rho,\rho_{k-1})\right\}.
\end{equation}
By the assumptions on $S_\e$, classical results imply that $\rho_k^\e$ is Lipschitz continuous. 

\medskip

Without loss of generality, we can assume that $S_\e$ satisfies the growth \eqref{eq:2as22} if $\rho>2$.
We can write the optimality condition
\begin{equation}\label{opt:approx}
S_\e''(\rho_k^\e)\nabla\rho_k^\e+ \nabla \Phi+\frac{\nabla\vphi_k^\e}{\t}=0\ \ \ae,
\end{equation}
where $\vphi_k^\e$ is a Kantorovich potential in the transport of $\rho_k^\e$ onto $\rho_{k-1}.$
From here, there exists a constant $C>0$ (depending only on $r$ and $\s_1$) such that 
$$\int_{\Om}|S_\e''(\rho_k^\e)\nabla\rho_k^\e|^2\rho_k^\e\dd x\le C\left(\|\nabla\Phi\|_{L^\infty}^2+\frac{1}{\t^2}W_2^2(\rho_{k}^\e,\rho_{k-1})\right).$$
And in particular, for any $\ell> 2$, we have by setting  $\Om_\ell:=\{\rho_k^\e>\ell\}$,
\begin{equation}\label{eq:H1-estim}
\int_{\Om_\ell}|\nabla(\rho_k^\e)^{r-1/2}|^2\dd x \le C\left(\|\nabla\Phi\|_{L^\infty}^2+\frac{1}{\t^2}W_2^2(\rho_{k}^\e,\rho_{k-1})\right).
\end{equation}
We know that the optimizers $\rho_k^\e$ are Lipschitz continuous on their supports, therefore the super-level sets $\Om_\ell$ are open.

Moreover, once again using the fact that $\rho_k^\e$ is Lipschitz continuous, we have that there exists $\d>0$ such that
$${\rm{dist}}(\partial\Om_\ell,\ov\Om_{2\ell})\ge 2\d.$$
Indeed, otherwise if one supposes the contrary, then for any $n\in\N$, there exist $x_n\in\partial\Om_\ell$ and $y_n\in\Om_{2\ell}$ such that ${\rm{dist}}(x_n,y_n)<\frac1n$, then one would have that $|\rho_k^\e(x_n)-\rho_k^\e(y_n)|\le\frac{1}{n}\|\nabla\rho_k^\e\|_{L^\infty(\Om_\ell)}\to 0,$ as $n\to+\infty.$ However, this would be a contradiction since $\rho_k^\e(x_n)=\ell$ and $\rho_k^\e(y_n)\ge 2\ell$.

\medskip

Now, by defining $\Om_{\ell,\d}:=\{\chi_{\Om_{2\ell}}\star\eta_\d>s\}$ for some $s\in(0,1/2)$ to be set later (where $\eta_\d:\R^d\to\R$ is a mollifier obtained from a  smooth even kernel $\eta:\R^d\to\R$ -- such that $\int_{\R^d}\eta\dd x=1$, $\eta\ge 0$ and $\spt(\eta)\subset B_1(0)$ -- by $\eta_\d:=(1/\d^d)\eta(\cdot/\d))$, we have that $\Om_{2\ell}\subset\ov\Om_{\ell,\d}\subset\Om_\ell$, $\Om_{\ell,\d}$ is an open set, and by Sard's theorem it has smooth boundary for $\sL^1$-a.e. $s\in(0,1/2).$ We choose such an $s$.

\medskip 

We have in particular from  \eqref{eq:H1-estim}  that
$$\int_{\Om_{\ell,\d}}|\nabla(\rho_k^\e)^{r-1/2}|^2\dd x \le C\left(\|\nabla\Phi\|_{L^\infty}^2+\frac{1}{\t^2}W_2^2(\rho_{k}^\e,\rho_{k-1})\right),$$
and so the Sobolev embedding theorem implies (since $\rho_k^\e$ is only uniformly bounded in $L^r(\Om)$) that $(\rho_k^\e)^{r-1/2}\in L^{2^*}(\Om_{\ell,\d})$ from where $\rho_k^\e\in L^{\b}(\Om_{\ell,\d}),$ where $\b:=2^*(r-1/2)$, if $d\ge 3$ and $\b<+\infty$ arbitrary if $d=2$ and $\b$ can be taken $+\infty$ if $d=1$. He we use the notation $2^*=2d/(d-2)$.

From the above construction we can claim that $\rho_k^\e\in L^{\b}(\Om)$. Indeed, we have 
\begin{align*}
\int_{\Om}(\rho_k^\e)^\b\dd x&=\int_{\{\rho_k^\e\le\ell\}}(\rho_k^\e)^\b\dd x+\int_{\Om_{\ell,\d}}(\rho_k^\e)^\b\dd x+\int_{\Om_\ell\setminus \Om_{\ell,\d}}(\rho_k^\e)^\b\dd x\\
&\le(2^\b+1)\ell^\b\sL^d(\Om) + C\left(\|\nabla\Phi\|_{L^\infty}^2+\frac{1}{\t^2}W_2^2(\rho_{k}^\e,\rho_{k-1})\right).
\end{align*}
Let us underline that this bound only depends on $W_2^2(\rho_{k}^\e,\rho_{k-1})$. Clearly, the previous inequality is valid only if $\beta<+\infty$. In the case of $d=1$, i.e. when $\beta=+\infty$, we first perform the computation for $\b<+\infty$ finite, and obtain the desired bound after taking $\b$-root of the previous inequality and sending $\b\to+\infty$.

Now, it is easy to see that because $S_\e\to S$ uniformly, we have that the objective functional in \eqref{eq:approx} $\Gamma$-convergences to the objective functional in the original problem as $\e\da 0$, w.r.t. the weak-$*$ convergence of probability measures. Indeed, take a sequence $(\rho^\e)_{\e>0}$ and $\rho$ in $\sP(\Om)$ such that $\rho^\e\weaklys\rho$ as $\e\da 0$. Notice that by the construction of the approximation $S_\e$, if $\cE_\e(\rho^\e)\le C$ (for a constant independent of $\e$), then we have that $(\rho^\e)_{\e>0}$ is uniformly bounded in $L^r(\Om)$. By the uniform convergence $S_\e\to S$, we have that for any $\d>0$ there exists $\e_0$ such that
$$S(\rho^\e)\le S_\e(\rho^\e)+\delta,\ \forall \e<\e_0.$$
Therefore 
$$\cE(\rho)\le\liminf_{\e\da 0}\cE(\rho^\e)\le\liminf_{\e\da 0}\cE_\e(\rho^\e)+\d\sL^d(\Om),$$
so the $\Gamma$-liminf inequality follows by the lower semicontinuity of the energy $\cE$ and the arbitrariness of $\d>0$. For the $\Gamma$-limsup inequality, we use a constant sequence $\rho^\e=\rho$ as a recovery sequence such that $\cE_\e(\rho)$ is finite for all $\e>0$. Clearly $\lim_{\e\da 0}\cE_\e(\rho)=\cE(\rho)$.

Finally, since both $\rho_k$ and $\rho_k^\e$, the solutions of the original and the approximated problems, respectively are unique, when $\e\da 0$ we find that $\rho_k$ also has the $L^\b(\Om)$ bound. The thesis of the lemma follows.

\end{proof}

Let us notice that in Lemma \ref{lem:apriori-2star} the $L^\b$ bounds on $\rho_k$ depends only on $\frac{1}{\t^2}W_2^2(\rho_k,\rho_{k-1})$ and the data. Therefore, when considering the piecewise constant interpolated curves $(\rho^\t)_{\t>0}$ (see their precise definition in \eqref{eq:int} below), and integrating them in time and space, we find the following very important lemma.

\begin{lemma}
\label{lem:uni}
Suppose that $\rho_0\in \sP(\Om)$ with $\cJ(\rho_0) < +\infty$ and \eqref{eq:g2} hold. The curves $(\rho^\t)_{\t>0}$ are uniformly bounded in $L^\b(Q)$ for $\beta$ given in 
\begin{align}
\label{eq:beta}
\beta :=
\begin{cases}
(2r-1)\frac{d}{d-2} &\hbox{ if } d \geq 3,\\
(0,\infty) &\hbox{ if } d=2\\
+\infty & \hbox{ if } d=1.
\end{cases}
\end{align}
\end{lemma}
\begin{proof}
Let $\b$ as in the statement of the lemma and let $(\rho^\t)_{\t>0}$ stand for the piecewise constant interpolations as defined in \eqref{eq:int}. Then, Lemma \ref{lem:apriori-2star} and \eqref{estim:Lbeta} imply that 
\begin{align*}
\int_0^T\int_{\Om}(\rho^\t)^\b\dd x\dd t=\t\sum_{k=1}^N\int_\Om(\rho_k)^\b\dd x\le \t N C + C\sum_{k=1}^N\frac{1}{\t}W_2^2(\rho_k,\rho_{k-1}),
\end{align*}
where $C>0$ depends only on the data and $\Om$. Since $\t N=T$ and $\sum_{k=1}^N\frac{1}{\t}W_2^2(\rho_k,\rho_{k-1})$ is uniformly bounded (see Lemma \ref{lem:gfi}), we conclude.

\end{proof}

Under the above assumption, we show a result parallel to Theorem~\ref{thm:opg}.

\begin{theorem}
\label{thm:opgr}
Suppose that $\rho_0\in \sP(\Om)$ such that $\cE(\rho_0) < +\infty$ and \eqref{eq:g2} hold. Then, for all $k \in \{1,\dots,N\}$ there exists $\cC=\cC(k) \in \R$ and $\ophi_{k} \in \cK(\rho_{k}, \rho_{k-1})$ satisfying \eqref{eq:1opg}. Here, $\cK(\rho_k,\rho_{k-1})$ and $\rho_{k}$ are given in Definition~\ref{def:kan} and \eqref{eq:step}, respectively.
\end{theorem}

We recall the following lemma from \cite{Roc68} and \cite{Roc71} and compute the subdifferential of $\cJ$ explicitly. In comparison to the previous subsection, it holds that $(L^r(\Om))^* = L^{r'}(\Om)$ for $r \in (1,+\infty)$ where $r':= \frac{r}{r-1}$ and thus the argument below is simpler than Lemma~\ref{lem:roc}.

\begin{lemma}
\label{lem:rocr}
\cite[Theorem 2]{Roc68}, \cite[Equations (1.11) \& (1.12)]{Roc71}
Let $\psi$ and $\Psi$ be given as in \eqref{eq:psi}. Assume that $\psi(\mu(x))$ is majorized by a summable function of $x$ for at least one $\mu \in L^\infty(\Om)$ and that $\psi^*(\zeta(x))$ is majorized by a summable function of $x$ for at least one $\zeta \in L^1(\Om)$. Then, an element $\xi \in L^r(\Om)^*$ belongs to $\partial \Psi(\mu)$ given in \eqref{eq:2psi} if and only if $\xi(x) \in \partial \psi(\mu(x))$ for a.e. $x\in \Om$.  
\end{lemma}

\begin{proof}[Proof of Theorem~\ref{thm:opgr}]

Let us set $k=1$. The first part of the proof is parallel to Proposition~\ref{prop:add} and Proposition~\ref{prop:wul}. Let us show
\begin{align}
\label{eq:opgr11}
\left. \partial \left(\cJ(\rho)+\cF(\rho)+\frac{1}{2\t}W_2^2(\rho,\rho_{0}) \right)\right |_{\rho = \rho_{1}} = \partial \cJ(\rho_{1}) + \Phi+ \frac{1}{\t} 
\K(\rho_{1}, \rho_{0}),
\end{align}
where $\K$ is given in Definition~\ref{def:kan} and the subdifferential is defined in Definition~\ref{def:subdif}. Recall $\Gamma(\cdot)$ from Definition~\ref{def:gam} and its equivalent property in Lemma~\ref{lem:equ}. Note that $\cJ \in \Gamma(L^r(\Om))$ follows from the convexity of $S$ and Lemma~\ref{lem:cj}. The same is true for $\cF$.

Let us underline that it is crucial that we have a priori bounds on the optimizers of \eqref{eq:step} in $L^\beta(\Om)$ for some $\b>1$. Indeed, Lemma \ref{lem:apriori-2star} yields that even if $r=1$ (which corresponds to the logarithmic entropy type interaction energy), we have that the optimizers satisfy $\rho_k\in L^{\b}(\Om).$ In this case, without loss of generality, one considers the continuity of $\cJ$ and $\cF$ in $L^{\beta}(\Om)$. Otherwise, we gain $L^r(\Om)$ bounds simply from the growth condition on $S$ at $+\infty$, hence we can also refer to the continuity of $\cJ$ in this space.

Furthermore, from Proposition~\ref{prop:dual}, we have
\begin{align*}
H^* = \frac{1}{2}W_2^2(\cdot,\rho_{0}) \hbox{ on } L^\b(\Om) 
\end{align*}
for $H : L^{\b'}(\Om) \to \Rinf$ given in \eqref{def:H} and  $\b':= \frac{\b}{\b-1}$. Thus we get $W_2^2(\cdot,\rho_{0}) \in \Gamma(L^\b(\Om))$. Lastly, by the parallel argument in Lemma~\ref{lem:Hlsc}, $H$ is also in $\Gamma(L^{\b'}(\Om))$.
From Lemma~\ref{lem:subsum} and Lemma~\ref{lem:subdual}, we conclude \eqref{eq:opgr11}.

\medskip

The rest of the proof is parallel to that of Theorem~\ref{thm:opg}. From \eqref{eq:opgr11} and Lemma~\ref{lem:rocr}, there exists $\xi \in \partial \cJ(\rho_{1})$ satisfying \eqref{eq:1sul}, $\ophi_{1} \in \cK(\rho_{1}, \rho_{0})$ and $\cC \in \R$ such that 
\begin{align*}
\xi + \frac{\ophi_{1}}{\tau}+\Phi - \cC = 0 \hbox{ a.e. on } \Om .
\end{align*}
and we conclude \eqref{eq:1opg}.
\end{proof}

To give a fine characterization of the critical regions $\{\rho=1\}$ arising in our models, we introduce a new {\it scalar pressure field} $p$ defined via the  subdifferential of $\cJ$ and its `nontrivial' value on the set $\{\rho=1\}$ is due to the non-differentiability of $S$ at $s_0=1$. This construction is inspired by recent works on crowd motion models with hard congestion effects (see for instance \cite{MauRouSan1, MesSan}).

\begin{definition}\label{def:p} Let $(\rho_k)_{k=1}^N$ be given by the minimizing movement scheme \eqref{eq:step} and let $\ophi_k\in\cK(\rho_k,\rho_{k-1})$.
For $k \in \{1,\dots,N\}$, let us define $p_{k} : \Om \to \R$ 
by
\begin{align}\label{eq:p}
p_{k} = p_{k}(\cdot ; \tau) :=
\begin{cases}
\max\left\{ \cC - \frac{\ophi_{k}}{\t}-\Phi,\cS'(1-) \right\} &\hbox{ {\rm{in}} } \rho_{k}^{-1}([0,1)),\\[5pt]
\cC - \frac{\ophi_{k}}{\t}-\Phi &\hbox{ \rm{in} } \rho_{k}^{-1}(\{1\}),\\[5pt]
\min\left\{ \cC - \frac{\ophi_{k}}{\t}-\Phi,\cS'(1+) \right\} &\hbox{ {\rm{in}} } \rho_{k}^{-1}((1,+\infty)).
\end{cases}
\end{align}
where the constant $C\in\R$ might be different at each step. We observe that by the convexity of $S$ and \eqref{eq:1opg} shown in Theorem \ref{thm:opgr}, $p_k$ can be written in compact form as 
\begin{align}\label{p:compact}
p_k=\min\left\{\max\left\{ \cC - \frac{\ophi_{k}}{\t}-\Phi,\cS'(1-) \right\}, \cS'(1+)\right\}.
\end{align}

\end{definition}

\begin{lemma}
\label{lem:lip_new}
For $k \in \{1,\dots, N\}$, $\ophi_{k}\in\cK(\rho_k,\rho_{k-1})$ and $p_k$ are Lipschitz continuous in $\Om$.
\end{lemma}
\begin{proof}
From \cite[Theorem 1.17]{OTAM} we have that $\ophi_{k}$ shares the modulus of continuity of the cost $(x,y)\mapsto |x-y|^2$. On the one hand, as $\ov\Om$ is compact, we conclude that $\ophi_{k}$ is Lipschitz continuous. On the other hand, {\color{red}\eqref{p:compact} implies} that $p_{k}$ is Lipschitz continuous.
\end{proof}

\begin{remark}
Let us remark that $p_k$ is given in the same manner as in \eqref{eq:p} or equivalently in \eqref{p:compact} throughout the manuscript. As the Lipschitz continuity of the Kantorovich potential is independent of the energy, the above lemma holds true for all the models we consider in this paper.
\end{remark}

In the following lemma we deduce more properties of the optimizers of the JKO scheme \eqref{eq:step}.

\begin{lemma}\label{lem:lipg-gen}
Suppose that Assumption \ref{as:gen} takes place. Let $(\rho_k)_{k=1}^N$ be obtained via the minimizing movement scheme \eqref{eq:step}. For $k \in \{1,\dots,N\}$, let $\ophi_{k}\in\cK(\rho_k,\rho_{k-1})$ given in Theorem~\ref{thm:opgr}. Then, we have

(i)
\begin{align}
\label{eq:1lipg-gen}
\rho_{k} = 
\begin{cases}
0, & {\rm{in}}\ f_k^{-1}((-\infty, \cS'(0+)]),\\ 
1, &  {\rm{in}}\ f_k^{-1}([\cS'(1-), \cS'(1+)]),\\ 
(\cS')^{-1} \circ f_k,  &{\rm{otherwise}}, 
\end{cases}
\end{align} 
a.e. in $\Omega$, where  $f_k:= \cC - \frac{\ophi_{k}}{\tau}-\Phi$, and $\cS'(0+)$ and $\cS'(1\pm)$ are given in \eqref{eq:slr};

(ii) $\rho_k$ is continuous in $\Omega$;

(iii) the formula
\begin{align}\label{deriv:formula-gen}
| \nabla \rho_{k}| =  \frac{|\nabla f_{k}| }{\cS'' (\rho_{k})},
\end{align}
holds true a.e. in $\rho_{k}^{-1}(\Ro)$.

(iv) If in addition we suppose that Assumption \ref{as:intro_main1} takes place with some $\s_2>0$ and $m\in[1,2]$, then
$\rho_{k}$ is Lipschitz continuous in $\Omega$ with a Lipschitz constant that might degenerate when $\t\da 0$.
\end{lemma}

\begin{proof}
(i)
By Assumption~\ref{as:gen}, $\cS'$ is strictly increasing function in $\Ro$ and 
$$\cS'(0+) < \cS'(a) < \cS'(1-) \leq \cS'(1+) < \cS'(b)$$
for all $a \in (0,1)$ and $b>1$. Thus, \eqref{eq:1opg} shown in Theorem \ref{thm:opgr} implies that 
$\rho_{k}(x) = 0$ for $x \in f_k^{-1}((-\infty, \cS'(0+)])$ and $\rho_{k}(x) = 1$ for $x \in f_k^{-1}([\cS'(1-), \cS'(1+)])$. Also, $\cS'$ is invertible in $\Ro$, therefore \eqref{eq:1opg} implies 
\begin{align*}
\rho_k(x) = (\cS')^{-1} \circ f_k(x) \hbox{ for } x \in f_k^{-1}(\cS'(\R^+\setminus\{1\})),
\end{align*}
and we conclude \eqref{eq:1lipg-gen} a.e. in $\Omega$.

\medskip

(ii) Let us show that $\rho_{k}$ is continuous in $\Omega$. 
Define $\widehat{(\cS')^{-1}} : \R \to [0, +\infty)$ by
\begin{align}
\label{eq:lipg20}
\widehat{(\cS')^{-1}} := 
\begin{cases}
0, &\hbox{ in } (-\infty, \cS'(0+)],\\ 
1, &\hbox{ in } [\cS'(1-), \cS'(1+)],\\ 
(\cS')^{-1},  &\hbox{ otherwise. }
\end{cases}
\end{align}
Note that from \eqref{eq:1lipg-gen}, we have
\begin{align}
\label{eq:lipg21-gen}
\rho_k = \widehat{(\cS')^{-1}} \circ f_k\ \ {\rm{a.e.}\ in\ } \Omega.
\end{align}
From the continuity and invertibility of $\cS'$ in $\Ro$, we conclude that $\widehat{(\cS')^{-1}}$ is continuous in $\R$. Furthermore, from Lemma~\ref{lem:lip_new}, we know that $\ophi_{k}$ is Lipschitz continuous 
and $\Phi$ is Lipschitz continuous by assumption, therefore $f_k$ is Lipschitz continuous. From \eqref{eq:lipg21-gen}, we conclude that $\rho_{k}$ is continuous in $\Omega$. 

\medskip

(iii) As $\cS$ is strictly convex and twice differentiable in $\Ro$ (by Assumption~\ref{as:gen}), $(\cS')^{-1}$ is differentiable in $S'(\R^+\setminus\{1\})$ and on this set we have
\begin{align}
\label{eq:lipg22-gen}
(\widehat{(\cS')^{-1}})'=((\cS')^{-1})' = \frac{1}{\cS'' \circ (\cS')^{-1}}.
\end{align}
Therefore, by \eqref{eq:lipg21-gen} we have \eqref{deriv:formula-gen}.

(iv) Using Assumption \ref{as:intro_main1}, from \eqref{deriv:formula-gen}, we conclude that a.e. in $\rho_{k}^{-1}(\Ro)$ we can compute
\begin{align}
\label{eq:lipg23}
| \nabla \rho_{k}| =  \frac{|\nabla f_{k}| }{\cS'' (\rho_{k})} \leq  \s_2 \rho_k^{2-m}  |\nabla f_{k}|.
\end{align}
As $f_k$ is Lipschitz continuous and $\rho_k$ is bounded (since it is continuous in $\Omega$), we conclude that $\rho_{k}$ is Lipschitz continuous in $\Omega$ if $m\in[1, 2]$.
\end{proof}

\begin{remark}\label{rem:general_deriv_rho}
Let us emphasize that the representation formula \eqref{eq:1lipg-gen} is independent of the entropy function $S$, therefore it remains the same for all energies considered in the manuscript. As a consequence, the formula \eqref{deriv:formula-gen} holds also true for all the models throughout the paper.
\end{remark}

\section{Linear diffusion with discontinuities -- a cornerstone of our analysis}
\label{sec:gfl}
In this section we show the well-posedness of \eqref{eq:main1} in the most simple case considered in this paper, i.e. when the associated internal energy is an entropy of logarithmic type. We give a fine characterization of the `critical phase' $\{\rho=1\}$ via a scalar \emph{pressure field}. In the next sections we shall see how the results and ideas from this sections will be important to build solutions for problems with more general nonlinearities.

\medskip

In this section, we assume that $\cS : \Rpz \rightarrow \R$ is defined by 
\begin{align}
\label{eq:sl}
\cS(\rho):= \begin{cases}
\rho \log\rho, &\hbox{ for } \rho \in [0,1],\\
2\rho\log\rho, &\hbox{ for } \rho\in(1,+\infty).
\end{cases}
\end{align}
Let us notice that $S$ defines a continuous superlinear function on $\Rp$ with $S'(1-)=1$ and $S'(1+)=2$.
\medskip

Our main theorem from this section can be formulated as follows.

\begin{theorem}
\label{thm:exi}
For $\rho_0 \in \sP(\Om)$ such that $\cJ(\rho_0) < +\infty$ and $S$ given in \eqref{eq:sl}, there exists $\rho \in L^1(Q)\cap {\rm{AC}}^2([0,T];\sP(\Om))$ and $p \in L^2([0,T]; H^1(\Om)) \cap L^\infty(Q)$ with $\sqrt{\rho} \in L^2([0,T]; H^1(\Om))$ such that $(\rho,p)$ is a weak solution of
\be\label{eq:exi}
\left\{
\ba{ll}
\partial_t\rho - \Delta (\rho p) -\nabla\cdot(\nabla\Phi\rho) =0, & {\rm{in}}\ (0,T)\times\Om,\\
\rho(0,\cdot)=\rho_0, & {\rm{in}}\ \Om,\\
(\nabla (\rho p )+\nabla\Phi) \cdot \n = 0 , & {\rm{in}}\ [0,T] \times \partial \Om,
\ea
\right.
\ee
in the sense of distribution. Furthermore, $(\rho,p)$ satisfies
\begin{align}
\label{eq:rel}
\begin{cases}
p(t,x)=1 &\hbox{ {\rm{a.e. in}} } \{0 < \rho(t,x) < 1\},\\
p(t,x) \in [1,2] &\hbox{ \rm{a.e. in} } \{\rho(t,x) = 1\},\\
p(t,x)=2 &\hbox{ \rm{a.e. in} } \{\rho(t,x) > 1\}.
\end{cases}
\end{align}
If in addition $\rho_0\in L^\infty(\Om)$ and $\Phi$ satisfies \eqref{as:Phi2}, then $\rho \in L^2([0,T]; H^1(\Om))\cap L^\infty(Q)$.
\end{theorem}

In the proof of the previous theorem we rely on the minimizing movements scheme associated to the gradient flow of $\cJ$, defined in \eqref{eq:step}. As technical tools, we define different interpolations between the discrete in time densities $(\rho_k)_{k=0}^N$ and pressures $(p_k)_{k=0}^N$ and obtain a weak solution of \eqref{eq:exi} by sending $\t\da 0$. 

By the definition of $p_k$ in Definition \ref{def:p}, the optimality condition \eqref{eq:1opg} in Theorem~\ref{thm:opg} applied to the energy from this section, can be simplified as follows.
\begin{lemma}
\label{lem:popt}
For all $k \in \{1,\dots,N\}$, there exists $C\in\R$ such that
\begin{align}
\label{eq:popt}
p_{k}(1+ \log\rho_{k})+ \frac{\ophi_{k}}{\tau}+\Phi=\cC \ \ \ae
\end{align}
\end{lemma}

\begin{proof}
Note that a subdifferential $\partial \cS(\rho)$ of $\cS: \Rpz \rightarrow \R$ is given by
\begin{align}
\label{eq:ds}
\partial \cS(\rho) =  
\begin{cases}
1+ \log\rho &\hbox{ for } 0< \rho <1,
\\ [1,2]	&\hbox{ for } \rho =1,
\\ 2(1+ \log\rho) &\hbox{ for } \rho >1.
\end{cases}
\end{align}
Thus, Theorem~\ref{thm:opg} and \eqref{eq:p} imply
\begin{align}
\label{eq:pk}
p_{k} = 
\begin{cases}
1 &\hbox{ in } \rho_{k}^{-1}((0,1)),\\
\cC - \frac{\ophi_{k}}{\t}-\Phi \in [1,2] &\hbox{ in } \rho_{k}^{-1}(\{1\}),\\
2 &\hbox{ in } \rho_{k}^{-1}((1,+\infty)).
\end{cases}
\hbox{ \quad a.e.}
\end{align}
Thus, we simplify \eqref{eq:1opg} into \eqref{eq:popt}.
\end{proof}

An easy consequence of the above constructions is the following result, which can be seen as a simplified version of Lemma \ref{lem:lipg-gen}.

\begin{lemma}
\label{lem:lip}
For $k \in \{1,\dots, N\}$, $\rho_{k}$ is Lipschitz continuous in $\Om$. Here, $\rho_{k}$ is given in \eqref{eq:step}.
\end{lemma}

\begin{proof}
From \eqref{eq:popt} in Lemma~\ref{lem:popt}, we have that
\begin{align}
\label{eq:lip21}
\rho_{k}(x) = \exp \left\{ \frac{1}{p_{k}(x)} \left( \cC-\frac{\ophi_{k}(x)}{\tau} -\Phi \right) - 1 \right\} \hbox{ a.e. }
\end{align}
As $p_{k}$ and $\ophi_{k}$ are Lipschitz continuous from Lemma~\ref{lem:lip_new}, $\Phi$ is Lipschitz continuous by the assumption and $p_{k}$ has a lower bound $+1$ from \eqref{eq:pk}, \eqref{eq:lip21} implies that $\rho_{k}$ is Lipschitz continuous.
\end{proof}

\subsection{Interpolations between the discrete in time densities, velocities, momenta and pressures}
As technical tools, similarly as it is done in the framework of models developed for instance in \cite{MauRouSan1, MesSan, OTAM}, we introduce two different kinds of interpolations between the objects in the title of the subsection. These interpolations actually are independent of the considered energies, and we refer back to them throughout the paper.

{\it Piecewise constant interpolations.}  Let us define $\rho^\t, p^\t : Q \to \R$ and $\v^\t, \E^\t: Q\to \R^d$ as follows
\be\label{eq:int}
\left\{
\ba{ll}
\rho^\t(t,x) := \rho_{k}(x),\\[4pt]
p^\t(t,x):=p_k(x),\\[4pt]
\ds\v^\t(t,x) := \frac{1}{\t} \nabla \ophi_k(x),\\ [4pt]
\E^\t(t,x) := \rho^\t(t,x) \v^\t(t,x)
\ea
\right.
\hbox{ for } (t,x) \in ((k-1)\t, k\t] \times \Om \hbox{ and } k \in \{1,\dots, N\},
\ee
for $(\rho_{k})_{k=1}^N$ obtained in \eqref{eq:step} and $\ophi_{k}\in\cK(\rho_k,\rho_{k-1})$ given in Theorem~\ref{thm:opg}.

\medskip

By standard arguments on gradient flows (see for instance \cite[Proposition 8.8]{OTAM}, \cite[Lemma 3.5]{MesSan}), we have the following.

\begin{lemma}
\label{lem:gfi}
It holds that
\begin{align*}
\frac{1}{2\t}\sum_{k =1}^N W_2^2(\rho_{k},\rho_{k-1}) &= \frac{1}{2\t}\sum_{k=1}^N \int_\Om  |\nabla \ophi_{k}|^2 \dd\rho_{k}(x)  \leq \cJ(\rho_0) - \inf \cJ.
\end{align*}
Furthermore, there exists a constant $C >0 $ such that for any $0\le s<t\le T$
\begin{align}
\label{eq:2gfi}
W_2(\rhot(t), \rhot(s)) &\leq C(t-s + \t)^{\frac{1}{2}}.
\end{align}
\end{lemma}

\begin{proposition} 
\label{prop:con}
Let $(\rho^\t)_{\t>0}$ and $(p^\t)_{\t>0}$  given \eqref{eq:int} and \eqref{eq:p}, respectively. We have the followings.
\begin{itemize}
\item[(1)] $(p^\t)_{\t>0}$ is uniformly bounded in $L^2([0,T]; H^1(\Om))\cap L^\infty(Q)$;
\item[(2)] $(\sqrt{\rho^\t})_{\t>0}$ is uniformly bounded in $L^2([0,T]; H^1(\Om))$;
\item[(3)] if in addition $\rho_0\in L^\infty(\Om)$ and $\Phi$ satisfies \eqref{as:Phi2}, then $(\rho^\t)_{\t>0}$ is uniformly bounded in $L^2([0,T]; H^1(\Om))\cap L^\infty(Q)$.
\end{itemize}
\end{proposition}

\begin{proof}
\textbf{Step 1.} Clearly, by construction,  $(p^\t)_{\t>0}$ is uniformly bounded. Furthermore, if $\rho_0\in L^\infty(\Om)$, then Lemma~\ref{lem:Linfty} implies that $(\rho^\t)_{\t>0}$ is uniformly bounded by a constant depending only on the data for all $t\in[0,T]$. 

\medskip

\textbf{Step 2.}
Now, let us show that $(\nabla \sqrt{\rho^\t})_{\t>0}$ and $(\nabla p^\t)_{\t>0}$ are uniformly bounded in $L^2(Q)$. Let $\ophi_k\in\cK(\rho_k,\rho_{k-1})$. Lemmas~\ref{lem:lip_new} and \ref{lem:lip} implies that $\ophi_{k}, \rho_{k}$ and $p_{k}$ are Lipschitz continuous functions, and therefore by Rademacher's theorem one can differentiate these function a.e. in $\Om$.
Note that $\{\rho_{k} \neq 1 \}$ is an open by the continuity of $\rho_{k}$ in Lemma~\ref{lem:lip} and thus \eqref{eq:pk} implies
\begin{align}
\label{eq:con20}
\nabla p_{k} = 0\ \ae \hbox{ in } \{ \rho_{k} \neq 1 \}. 
\end{align}
Therefore, we get
\begin{align}
\label{eq:con21}
\log\rho_{k} \nabla p_{k} = 0 \hbox{ and } (\rho_{k}-1) \nabla p_{k} = 0 \hbox{ \quad a.e.}
\end{align}

\medskip

Next, we claim that
\begin{align}
\label{eq:con212}
\nabla p_{k} \cdot  \nabla \rho_{k} = 0 \hbox{ \quad a.e. in }\Om.
\end{align}
From \eqref{eq:pk}, the above holds in the open set $\{\rho_k \neq 1 \}$ and in the interior of $\{\rho_k = 1 \}$, but we point out that $\partial \{\rho_k = 1 \}$ may have positive measure even though $\rho_k$ is Lipschitz continuous. In order to show \eqref{eq:con212} in $\Om$, we apply the coarea formula and \eqref{eq:con20}. As $\rho_{k}$ is Lipschitz and $\nabla p_{k}$ is in $L^1(\Om)$, we could use the coarea formula in \cite[Corollary 5.2.6]{KraPar08} and conclude that
\begin{align*}
\int_{\Om} | \nabla p_{k}| | \nabla \rho_{k}| \dd x = \int_{\R} \int_{(\rho_{k})^{-1}(s)} | \nabla p_{k} | \dd \sH^{d-1} \dd s.
\end{align*}
where $\sH^{d-1}$ stands for the $(d-1)$-dimensional Hausdorff measure. From \eqref{eq:con20}, we conclude \eqref{eq:con212}.

\medskip

Differentiating \eqref{eq:popt} and applying \eqref{eq:con21} and \eqref{eq:con212}, we have
\begin{align}
\label{eq:con22}
- \frac{\nabla \ophi_{k}}{\tau} -\nabla\Phi = \nabla (p_{k}(1+ \log\rho_{k})) = \nabla p_{k} + \frac{p_{k}}{\rho_{k}}\nabla \rho_{k} \hbox{ \quad a.e.}
\end{align}
From \eqref{eq:con22} and \eqref{eq:con212} again, we have
\begin{align}
\label{eq:con23}
2\rho_{k}  \left(\frac{|\nabla \ophi_{k}|^2}{\tau^2}+|\nabla\Phi|^2\right)  \ge | \nabla p_{k} |^2 + \frac{p_{k}^2}{\rho_{k}} |\nabla \rho_{k} |^2  \hbox{ \quad a.e.},
\end{align}
from where we can write
\begin{align*}
2\rho_{k}  \left(\frac{|\nabla \ophi_{k}|^2}{\tau^2}+|\nabla\Phi|^2\right)  \ge | \nabla p_{k} |^2 + p_{k}^2 |\nabla \sqrt{\rho_{k}} |^2  \hbox{ \quad a.e.}
\end{align*}

As $p_{k} \in [1,2]$ (from \eqref{eq:pk}),  we have 
$$\int_\Om\left(| \nabla p_{k} |^2 + |\nabla \sqrt{\rho_{k}} |^2\right)\le 2 \int_{\Om}\frac{|\nabla \ophi_{k}|^2}{\tau^2}\rho_{k}\dd x+ 2\|\nabla\Phi\|_{L^\infty}^2.$$
From Lemma~\ref{lem:gfi}, we conclude that $(\sqrt{\rho^\t})_{\t>0}$ and $(p^\t)_{\t>0}$ are uniformly bounded in $L^2([0,T]; H^1(\Om))$ for all $\t>0$.

Moreover, if $\rho_0\in L^\infty(\Om)$, we have $\|\rho_{k}\|_{L^\infty} \le \| \rho_0 \|_{L^\infty(\Om)}e^{dT\|[\Delta\Phi]_+\|_{L^\infty}}$ (from Lemma~\ref{lem:Linfty}), and therefore from \eqref{eq:con23} we get
\begin{align}
\label{eq:con24}
\int_{\Om}| \nabla p_{k} |^2\dd x + \int_\Om\frac{1}{ \| \rho_0 \|_{L^\infty(\Om)} e^{dT\|[\Delta\Phi]_+\|_{L^\infty}}} |\nabla \rho_{k} |^2\dd x \le C, 
\end{align}
from where we have $(\rho^\t)_{\t>0}$ is uniformly bounded in $L^2([0,T];H^1(\Om)).$
\end{proof}

\begin{corollary}\label{cor:ent-ent}
Let $(\rho^\t)_{\t>0}$ and $(p^\t)_{\t>0}$ be as in the previous proposition. There exist $p \in L^2([0,T]; H^1(\Om))$ and $\rho\in L^1(Q)$ such that 
\begin{align*}
\rho^\t \to \rho\ {\rm{in}}\ L^1(Q),\ {\rm{as}}\  \t\da 0,
\end{align*}
and
\begin{align*}
p^\t \weakly p\ {\rm{in}}\ L^2([0,T]; H^1(\Om)), \ {\rm{as}}\  \t\da 0.
\end{align*}
along a subsequence. If in addition $\rho_0\in L^\infty(\Om)$ and $\Phi$ satisfies \eqref{as:Phi2}, then we also have $\rho\in L^2([0,T]; H^1(\Om))$ and $\rho^\t \to \rho \ {\rm{in}}\ L^2(Q), \ {\rm{as}}\  \t\da 0.$
\end{corollary}
\begin{proof}
\medskip

The weak sequential compactness of $(p^\t)_{\t>0}$ follows from the uniform boundedness in $L^2([0,T]; H^1(\Om))$ in the previous proposition. 
Also, as $(\rho^\t)_{\t>0}$ has the `quasi-H\"{o}lder' type estimates in Lemma~\ref{lem:gfi} and $(\sqrt{\rho^\t})_{\t>0}$ is uniformly bounded in $L^2([0,T]; H^1(\Om))$, we conclude the strong compactness of $(\rhot)_{\t>0}$ in $L^1(Q)$ by a consequence of a modified version of the classical Aubin-Lions lemma in Lemma~\ref{lem:al}, ofter used in similar context (see for instance \cite[Proposition 4.8]{DifMat14} and \cite[Proposition 5.2]{Lab}). If $\rho_0\in L^\infty(\Om)$, the last statement simply follows from the similar arguments.
\end{proof}

As a consequence of the above results, we have the following.

\begin{lemma}
\label{lem:rel}
$(\rho,p)$ given in Corollary \ref{cor:ent-ent} satisfies \eqref{eq:rel}.
\end{lemma}

\begin{proof}
\textbf{Step 1.} 
Let  $(\rho^\t, p^\t)$ be defined in \eqref{eq:int} and \eqref{eq:p}. First, from \eqref{eq:pk}, we have
\begin{align}
\label{eq:rel11}
(\pt-2)(\rhot-1)_+ = (\pt-1)(\rhot-1)_{-}  = 0\ {\rm{in}}\ Q.
\end{align}
As it holds that
\begin{align}
\label{eq:rel13}
| (\rhot-1)_+ -  (\rho-1)_+ | \leq | \rhot - \rho| \hbox{ and } | (\rhot-1)_- -  (\rho-1)_- | \leq | \rhot - \rho|,
\end{align}
Proposition~\ref{prop:con} implies that both $(\rhot-1)_+ \to (\rho-1)_+$ and $(\rhot-1)_{-} \to (\rho-1)_{-}$ in $L^1(Q)$ as $\t\da 0$ (up to passing to a subsequence).

\medskip

\textbf{Step 2.} 
Let us show that for a.e. $t \in [0,T]$
\begin{align}
\label{eq:rel21}
\int_\Om (p(t,x)-2)(\rho(t,x)-1)_+ \dd x =0 \hbox{ and } \int_\Om (p(t,x)-1)(\rho(t,x)-1)_{-} \dd x = 0.
\end{align}
We only show the first one as the parallel arguments work for the second one.
From \eqref{eq:rel11}, we have
\begin{align}
\label{eq:rel22}
0 &=  \int_Q (\pt(t,x)-2)(\rhot(t,x)-1)_+ \dd x \dd t.
\end{align}
Recall that up to passing to a subsequence, $(\pt)_{\t>0}$ convergences weakly$-\star$ in $L^\infty(Q)$ (see Proposition~\ref{prop:con}) and $((\rhot(t,x)-1)_+)_{\t>0}$ converges strongly (from Step 1) in $L^1([0,T\times\Om])$ as $\t\da 0$. Combining these with \eqref{eq:rel22}, we conclude the first equation of \eqref{eq:rel21}.

\medskip

As $\pt \in [1,2]$ for $\pt$ given in \eqref{eq:p}, we have $p \in [1,2]$ a.e. in $Q$.
Thus, \eqref{eq:rel21} implies that
\begin{align*}
(p-2)(\rho-1)_+ = (p-1)(\rho-1)_{-}  = 0 \hbox{ \quad a.e.}
\end{align*}
and we conclude \eqref{eq:rel}.
\end{proof}

\begin{proposition}
\label{prop:e}
Let $\E^\t$ be given in \eqref{eq:int}. Then up to passing to a subsequence, $(\E^\t)_{\t>0}$ weakly-$\star$ converges to 
$$\E:= - \nabla(p\rho)-\nabla\Phi\rho,\ \  {\rm{in}}\ \  \sD'(Q;\R^d),$$ 
as $\t\da 0$ where  and $(\rho,p)$ is given in Corollary~\ref{cor:ent-ent}.
\end{proposition}

\begin{proof}
For any test function $\zeta \in C^\infty_c(Q;\R^d)$, we claim that up to passing to a subsequence,
\begin{align}
\label{eq:exi11}
\I:= \int_Q \zeta\cdot \dd (\E^\t - \E)  \to 0, \hbox{ as } \t \da 0.
\end{align}
From \eqref{eq:con21}, we have $\log \rhot \nabla \pt = 0$ in a.e. in $Q$ and thus it holds that
\begin{align}
\label{eq:exi12}
- \E^\t = p^\tau \nabla \rhot + \rhot(1+ \log\rhot) \nabla \pt + \nabla\Phi\rho^\t= \nabla(\rhot \pt)+\nabla\Phi\rho^\t.
\end{align}
By the weak convergence of $(\rho^\t)_{\t>0}$ to $\rho$, we already have that 
$$\int_Q\zeta\cdot\nabla\Phi\dd\rho^\t\dd t\to\int_Q\zeta\cdot\nabla\Phi\dd\rho\dd t,\ \t\da 0,$$
we only focus on the other term. By integration by parts and and from the fact that $\zeta \in C^\infty_c(Q;\R^d)$, we study thus
\begin{align*}
\I_1 = \int_Q (\rhot \pt - \rho p)  \diver \zeta\dd x\dd t.
\end{align*}

\medskip

By subtracting and adding the same term in the above equation, we get
\begin{align*}
\I_1 = \I_2 + \I_3 \hbox{ where } \I_2 = \int_Q (\rhot  - \rho ) \pt  \diver \zeta \dd x\dd t \hbox{ and } \I_3 = \int_Q \rho (\pt - p)  \diver \zeta \dd x\dd t.
\end{align*}
From the H\"{o}lder inequality, we have
\begin{align*}
|\I_2 | \leq \| \rhot  - \rho \|_{L^1(Q)} \| \pt \|_{L^\infty(Q)} \| \diver \zeta \|_{L^\infty(Q)}.
\end{align*}
As  $\rho^\t \to \rho$ in $L^1(Q)$ as $\t\da 0$ and $\| \pt \|_{L^\infty(Q)}$ is uniformly bounded (Proposition~\ref{prop:con}), we conclude $\I_2 \to 0$ as $\t \da 0$. On the other hand, as $p^\t \weaklys p \hbox{ in } L^\infty(Q)$ as $\t\da 0$ (Proposition~\ref{prop:con}), and $\rho\in L^1(Q)$ we have $\I_3 \to 0$ as $\t\da 0$ as well, and thus we conclude \eqref{eq:exi11}.
\end{proof}

To arrive to the time continuous PDE satisfied by $(\rho,p)$ from Corollary \ref{cor:ent-ent}, as technical tools (inspired from \cite{MauRouSan1,MesSan, OTAM}), we introduce a geodesic interpolation between $(\rho_k)_{k=1}^N$ and we consider the corresponding velocities and momenta as well.

More precisely, we define $\trho^\t:[0,T]\to\sP(\Om)$, $\tilde\v^\t, \tE^\t \in\sM(Q;\R^d)$ as follows: for $t \in ((k-1)\t, k\t] \hbox{ and } k \in \{1,\dots,N\}$
\be\label{eq:csi}
\left\{
\ba{ll}
\trho^\t(t,x) := \left(  \frac{k\t - t}{\t} \v^\t(t,x) + \id  \right)_\# \rho^\t(t,x),\\
\tilde\v^\t(t,x) := \v^\t(t,x) \circ \left(\frac{k\t - t}{\t} \v^\t(t,x) + \id  \right)^{-1}, \\
\tE^\t(t,x) := \trho^\t(t,x) \tilde\v^\t(t,x),
\ea
\right.
\ee
where $\rho^\t$ and $\v^\t$ are given in \eqref{eq:int}.

\medskip

Following the very same steps as in \cite[Lemma 8.9]{OTAM} and \cite[Step 2 in Theorem 3.1]{MesSan}, we have the following.

\begin{lemma}\label{lem:coi}
We have that 
\begin{itemize}
\item[(i)] $\ds(\trho^\t)_{\t>0}$ is uniformly bounded in ${\rm{AC}}^2([0,T];\sP(\Om))$;
\item[(ii)] there exists $C>0$ such that  $\ds\int_0^T\int_{\Om}|\tilde v^\t|^2\dd\tilde\rho^\t_t\dd t\le C;$
\item[(iii)] $\ds (\tE^\t)_{\t>0}$ is uniformly bounded in $\sM(Q;\R^d)$.
\end{itemize}

As a consequence, we have that along a subsequence
\begin{itemize}
\item[(iv)] $\ds\sup_{t\in[0,T]}W_2(\trho^\t_t,\rho_t)\to 0,\ {\rm{as}}\ \t\da 0,$
\item[(v)] $\ds\tE^\t \weaklys \E,\ {\rm{in}}\ \sM(Q;\R^d),\ {\rm{as}}\ \t\da 0,$
\end{itemize}
where $\rho$ is given in Proposition~\ref{prop:con} and $\E$ is given in Propositon~\ref{prop:e}.
\end{lemma}

\medskip
Now, we are ready to prove Theorem \ref{thm:exi}.
\begin{proof}[Proof of Theorem~\ref{thm:exi}]
Let us underline that the main reason for introducing the interpolations $(\trho^\t, \tE^\t)$ is that by construction, they satisfy the PDE
\be
\left\{
\ba{ll}
\partial_t\tilde\rho^\t + \nabla \cdot  \tE^\t =0, & {\rm{in}}\ (0,T)\times\Om,\\
\tilde\rho^\t(0,\cdot)=\rho_0, & {\rm{in}}\ \Om,\\
\tE^\t \cdot \n = 0 , & {\rm{on}}\ [0,T] \times \partial \Om,
\ea
\right.
\ee
in the distributional sense. Then, Lemma~\ref{lem:coi} and Proposition \ref{prop:e} allow us to conclude that $(\rho,p)$ satisfies \eqref{eq:exi} in the distributional sense. Last, from Lemma~\ref{lem:rel}, we conclude that $(\rho,p)$ satisfies \eqref{eq:rel}. The thesis of the theorem follows.
\end{proof}


\section{Linear diffusion on $\{\rho<1\}$ and porous medium type diffusion on $\{\rho>1\}$}
\label{sec:gfn}

As we will see below, in this section the diffusion coefficients and the diffusion rates are not necessarily supposed to be the same in the regions $\{\rho < 1 \}$ and $\{ \rho >1 \}$. Therefore, a technical difficulty arises, because of the lack of a simple way (as in \eqref{eq:popt}) to derive the first order necessary optimality conditions for the minimizing movement scheme. To overcome this issue, instead, we use a particular decomposition for $\cS$, which allows us to use the construction from Section \ref{sec:gfl}. 

In this section too, we impose Assumption \ref{as:gen}. If $\rho_0\notin L^\infty(\Om)$, we impose additionally \eqref{eq:g2}. Furthermore, throughout this section we suppose also the following:
$\cS : \Rpz \to \R$ satisfies
\begin{align}
\label{eq:n}
\frac{\rho^{-1}}{\s_2} \leq  \cS''(\rho) \hbox{ {\rm{in}} } (0,1)
\end{align}
for some constant $\s_2>\s_1$ for $\s_1$ given in \eqref{eq:2as22}. This corresponds to \eqref{eq:g1} with $m=1$.

A direct consequence of the above assumption is the following result.
\begin{lemma}
\label{lem:1n}
$\cS : \Rpz \to \R$ satisfies
\begin{align}
\label{eq:1n}
\cS'(0+) = -\infty.
\end{align}
\end{lemma}
\begin{proof}
Integrating \eqref{eq:n} from $\frac{1}{2}$ to $\rho$, it holds that
\begin{align*}
\cS'\left(\frac{1}{2}\right) - \cS'(\rho) \geq \frac{1}{\s_2}\left( \log \frac{1}{2} - \log \rho \right).
\end{align*}
As $\s_2 > 0$, we conclude that
\begin{align*}
\cS'(\rho) \leq \cS'\left(\frac{1}{2}\right) - \frac{1}{\s_2} \log \frac{1}{2} + \frac{1}{\s_2} \log \rho \ \ \to \ \  -\infty \ \ \ \ \ \hbox{ as } \rho \to 0^+.
\end{align*}
\end{proof}

\begin{example}
For $m>1$, $\cS : \Rpz \to \R$ given by
\begin{align*}
\cS(\rho) := 
\begin{cases}
\rho \log \rho ,  &{\rm{for}}\ \rho \in [0,1], \\
\frac{1}{m-1}(\rho^m - 1),   &{\rm{for}}\ \rho \in (1, +\infty).
\end{cases}
\end{align*}
Note that Assumption~\ref{as:gen} follows from the smoothness and strict convexity of $\cS$ in $\Ro$ and 
\begin{align*}
\cS'(1-) = 1 < \cS'(1+) = \frac{m}{m-1}.
\end{align*}
\eqref{eq:n} is obtained by
\begin{align*}
\rho \cS''(\rho) =
\begin{cases}
1 ,  &{\rm{for}\ } \rho \in (0,1), \\
m \rho^{m-1} \geq m,   &{\rm{for}\ } \rho \in (1, +\infty).
\end{cases}
\end{align*} \eqref{eq:g2} is also fulfilled with $r=m$.

In this case, $L_\cS(\rho,p)(x)$ is given by
\begin{align*}
L_\cS(\rho,p)(x) =
\begin{cases}
\rho(x), &\hbox{ {\rm{if}} } 0 < \rho(x) < 1,\\
p(x) \in \left[ 1, \frac{m}{m-1} \right], &\hbox{ {\rm{if}} } \rho(x) = 1,\\
\rho(x)^m + \frac{1}{m-1}, &\hbox{ {\rm{if}} } \rho(x) >1.
\end{cases}
\end{align*}
\end{example}

Our main theorem from this section reads as:
\begin{theorem}
\label{thm:gm}
Suppose that \eqref{eq:g2} and \eqref{eq:n} hold true. For $\rho_0 \in \sP(\Om)$ such that $\cJ(\rho_0) < +\infty$, there exists $\rho\in L^{\beta}(Q)\cap {\rm{AC}}^2([0,T];\sP(\Om))$ for $\beta$ given in \eqref{eq:beta} and $p\in L^2([0,T]; H^1(\Om)) \cap L^\infty(Q)$ with $\sqrt{\rho}\in L^2([0,T];H^1(\Om))$
such that $(\rho,p)$ is a weak solution of
\be\label{eq:1gm}
\left\{
\ba{ll}
\partial_t\rho - \Delta (L_S(\rho,p)) -\nabla\cdot(\nabla\Phi\rho) =0, & {\rm{in}}\ (0,T)\times\Om,\\
\rho(0,\cdot)=\rho_0, & {\rm{in}}\ \Om,\\
(\nabla (L_S(\rho,p) )+\nabla\Phi\rho) \cdot \n = 0 , & {\rm{in}}\ [0,T] \times \partial \Om,
\ea
\right.
\ee
in the sense of distribution. Furthermore, $(\rho,p)$ satisfies for a.e. $(t,x) \in Q$
\begin{align}
\label{eq:2gm}
\begin{cases}
p(t,x)=\cS'(1-) &\hbox{ {\rm{if}} } 0 < \rho(t,x) < 1,\\
p(t,x) \in [\cS'(1-),\cS'(1+)] &\hbox{ {\rm{if}} } \rho(t,x) = 1,\\
p(t,x)=\cS'(1+) &\hbox{ {\rm{if}} } \rho(t,x) > 1.
\end{cases}
\end{align}
If in addition $\rho_0\in L^\infty(\Om)$ and $\Phi$ satisfies \eqref{as:Phi2}, we can drop \eqref{eq:g2} from the statement and we obtain that $\rho\in L^2([0,T]; H^1(\Om))\cap L^\infty(Q).$
\end{theorem}

Let us briefly explain the outline of the proof. First, we define $\Sa$ and $\Sb : \Rpz \rightarrow \R$ by
\begin{align}
\label{eq:sa}
\Sa(\rho) := 
\begin{cases}
\cS'(1-) \rho \log\rho,  &\hbox{ for } \rho \in [0,1], \\
\cS'(1+) \rho \log\rho,   &\hbox{ for } \rho \in (1, +\infty),
\end{cases}
\end{align}
and
\begin{align}
\label{eq:sb}
\Sb(\rho) := \cS(\rho) - \Sa(\rho).
\end{align}
We show the convexity of $\Sa$ and twice differentiability of $\Sb$ in Lemma~\ref{lem:sab}. This particular decomposition will be useful when deriving optimality conditions in our minimizing movement scheme. Under \eqref{eq:n}, we are able to apply similar arguments as the ones in Section~\ref{sec:gfl}.

\medskip

We point out that Lemma~\ref{lem:1n} implies that $\rho_{k}>0$ a.e. (see Lemma~\ref{lem:postg}). Theorem~\ref{thm:opg} and \eqref{eq:n} yield that $\rho_{k}$ satisfies the following lemma.

\begin{lemma}
\label{lem:lipg}
Let $(\rho_k)_{k=1}^N$ be obtained via the minimizing movement scheme \eqref{eq:step}. For $k \in \{1,\dots,N\}$ and $\ophi_{k}\in\cK(\rho_k,\rho_{k-1})$ given in Theorem~\ref{thm:opg}, we have that
\begin{align}
\label{eq:1lipg}
\rho_{k} = 
\begin{cases}
1, &  {\rm{in}}\ f_k^{-1}([\cS'(1-), \cS'(1+)]),\\ 
(\cS')^{-1} \circ f_k,  &{\rm{otherwise}}, 
\end{cases}
\end{align} 
in $\Om,$ where  $f_k:= \cC - \frac{\ophi_{k}}{\tau}-\Phi$, and $\cS'(1\pm)$ is given in \eqref{eq:slr}.
In particular, $\rho_{k}$ is Lipschitz continuous in $\Om$ with a Lipschitz constant that might degenerate when $\t\da 0$.
\end{lemma}

\begin{proof}
This result is a direct consequence of Lemma \ref{lem:lipg-gen}. Indeed, Lemma~\ref{lem:postg} shows that $\rho_k>0$ a.e. in $\Om$, therefore  $\spt(\rho_k)=\Om$. Also, in Assumption \ref{as:intro_main1} we have $m=1$. Moreover, since Lemma \ref{lem:1n} yields that $S'(0+)=-\infty$, and since we can eventually modify $\rho_k$ on a negligible set, we can use the representation \eqref{eq:1lipg} for $\rho_k$ everywhere in $\Om$. 
\end{proof}

The following properties hold for $\Sa$ and $\Sb$.

\begin{lemma}
\label{lem:sab}
$\Sa$ is convex and continuous in $\Rp$. Also, $\Sb$ is continuously differentiable and $\Sb'$ is locally Lipschitz continuous in $\Rp$. In particular, we have
\begin{align}
\label{eq:1sab}
\Sb(1) = \cS(1) \hbox{ and } \Sb'(1)=0.
\end{align}
\end{lemma}
\begin{proof}
From convexity of $\cS$, it holds that $\cS'(1-) < \cS'(1+)$ and thus $\Sa$ is convex. It is obviously also continuous by construction.

\medskip

On the other hand, by the construction in \eqref{eq:sa}, $\Sb(\rho)$ is differentiable on $\Ro$. Let us show that $\Sb(\rho)$ is differentiable at $\rho=1$. By differentiating \eqref{eq:sa} on $\Ro$, we have that
\begin{align*}
\Sa'(\rho) = \begin{cases}
\cS'(1-) (1+ \log\rho),  &\hbox{ for } \rho \in (0,1), \\
\cS'(1+) (1+ \log\rho),   &\hbox{ for } \rho \in (1, +\infty).
\end{cases}
\end{align*}
Therefore, we conclude that 
\begin{align*}
\Sb'(1-) = \cS'(1-) - \Sa'(1-) = 0 \hbox{ and } \Sb'(1+) = \cS'(1+) - \Sa'(1+) = 0
\end{align*}
and $\Sb$ is continuously differentiable in $\Rp$. As both $\cS'$ and $\Sa'$ are locally Lipschitz in $\Ro$, $\Sb'$ is also locally Lipschitz continuous in $\Ro$. As $\Sb'$ is continuous, we conclude that $\Sb'$ is locally Lipschitz continuous in $\Rp$. Lastly, $\Sb(1) = \cS(1)$ follows from $\Sa(1)=0$.

\end{proof}

\begin{lemma}
\label{lem:og} 
Let $(\rho_k)_{k=1}^N$ be obtained via the minimizing movement scheme \eqref{eq:step} and let $(p_{k})_{k=1}^N$ be constructed in \eqref{eq:p}. For $k \in \{1,\dots,N\}$, we have that
\begin{align}
\label{eq:1og}
p_{k}(1+ \log\rho_{k}) + \Sb'(\rho_{k})+ \frac{\ophi_{k}}{\tau}+\Phi=\cC, \ \ae\ {\rm{in}}\ \Om.
\end{align} 
\end{lemma}

\begin{proof}

We first note that Lemma~\ref{lem:1n} 
implies that $\rho_k >0$ a.e. in $\Om$ (see also Lemma~\ref{lem:postg}). From Theorem~\ref{thm:opg}, we have
\begin{align}
\label{eq:pkg}
p_{k} = 
\begin{cases}
\cS'(1-), &\hbox{ in } \rho_{k}^{-1}((0,1)),\\
\cC - \frac{\ophi_{k}}{\t}-\Phi, &\hbox{ in } \rho_{k}^{-1}(\{1\}),\\
\cS'(1+), &\hbox{ in } \rho_{k}^{-1}((1,+\infty)).
\end{cases}
\end{align}
As $\Sb'(1)=0$, \eqref{eq:1og} holds in $\rho_{k}^{-1}(\{1\})$ by \eqref{eq:pkg}.

\medskip

Lastly,  from \eqref{eq:pkg}, in $\rho_{k}^{-1}(\Ro)$ we have that
\begin{align}
\label{eq:og12}
\Sa'(\rho_{k}) = p_{k}(1+ \log\rho_{k}).
\end{align}
As $\cS' = \Sa' + \Sb'$ in $\rho_{k}^{-1}(\Ro)$, we conclude \eqref{eq:1og} from Proposition \ref{prop:sul}.
\end{proof}

\begin{remark}
As $\Sb$ is differentiable, in the previous proof we also used the fact
\begin{align*}
\partial \cS = \partial \Sa + \Sb',
\end{align*}
the proof of which can be found for instance in \cite[Corollary 1.12.2]{Kru13}.
\end{remark}

Similarly as in Section \ref{sec:gfl}, we construct piecewise constant and continuous in time interpolations $(\rho^\t, \v^\t, \E^\t)$ and $(\tilde\rho^\t, \tilde\v^\t, \tE^\t)$. Similarly to Proposition~\ref{prop:con}, we can formulate the following result.

\begin{proposition} 
\label{prop:cog}
$(\rho^\t)_{\t>0}$ and $(p^\t)_{\t>0}$ satisfy the exact same bounds as in Proposition~\ref{prop:con}.
\end{proposition}

\begin{proof}
Let us notice first that the uniform boundedness of $(\pt)_{\t>0}$ in $L^\infty(Q)$ follows from the construction in \eqref{eq:pkg}.

\medskip

Let us show the other estimates from Proposition~\ref{prop:con}. Note that both $\Sb'$ and $\rho_{k}$ are locally Lipschitz continuous (as we have shown in Lemma~\ref{lem:sab} and Lemma~\ref{lem:lipg}). Thus, Lemma~\ref{lem:og} implies that
\begin{align}
\label{eq:cog11}
- \frac{\nabla \ophi_{k}}{\tau}  -\nabla\Phi =  \nabla p_{k} +  \left( \frac{p_{k}}{\rho_{k}} + \Sb''(\rho_{k}) \right) \nabla \rho_{k},\    \hbox{ a.e. in } \Om.
\end{align}
By the parallel computation as in \eqref{eq:con23}, we conclude that
\begin{align*}
2\rho_{k} \frac{|\nabla \ophi_{k}|^2}{\tau^2}+2\rho_k|\nabla\Phi|^2 \ge | \nabla p_{k} |^2 + \rho_{k} \left( \frac{p_{k}}{\rho_{k}} + \Sb''(\rho_{k}) \right)^2 |\nabla \rho_{k} |^2.
\end{align*}
From Lemma~\ref{lem:sig} below, we have
\begin{align*}
\rho_{k} \left( \frac{p_{k}}{\rho_{k}} + \Sb''(\rho_{k}) \right)^2 |\nabla \rho_{k} |^2 \geq  \frac{1}{ \s_2^2 \rho_{k}} |\nabla \rho_{k} |^2 \hbox{ a.e. in } \Om.
\end{align*}

\medskip

The rest of arguments is parallel to {\bf Step 3} in Proposition~\ref{prop:con}, thus we conclude the thesis of the proposition.
\end{proof}

\begin{corollary}\label{cor:ent-por}
Up to passing to a suitable subsequence, the sequences $(\rho^\t)_{\t>0}$ and $(p^\t)_{\t>0}$ converge in the same sense as in Corollary \ref{cor:ent-ent}.
\end{corollary}

\begin{remark}
From \eqref{eq:cog11}, we have 
\begin{align*}
2\rho_{k} \frac{|\nabla \ophi_{k}|^2}{\tau^2}+2\rho_k|\nabla\Phi|^2 \ge \frac{| \nabla (F(\rho_{k},p_{k})) |^2}{\rho_{k}}, {\ \rm{where}\ } F(\rho,p):= p\rho + \rho \Sb'(\rho) - \Sb(\rho).
\end{align*}
Then, if $\rho_0\in L^\infty(\Om)$, this observation together with the uniform $L^\infty$ bounds on $\rho^\t$ imply uniform $L^2([0,T]; H^1(\Om))$ bounds on $F(\rho^\t,p^\t)$.
\end{remark}

As the proof of Proposition~\ref{prop:con}, we rely on the coarea formula when proving the following result.

\begin{lemma}
\label{lem:sig}
For $(\rho_{k})_{k=1}^N$ and $(p_{k})_{k=1}^N$ given in \eqref{eq:step} and \eqref{eq:p}, it holds that
\begin{align}
\label{eq:1sig}
| p_{k} + \rho_{k} \Sb''(\rho_{k}) | |\nabla \rho_{k} |  \geq  \frac{1}{\s_2} |\nabla \rho_{k} | {\rm{\ a.e.\ in\ }} \Om.
\end{align}
\end{lemma}

\begin{proof}
If $x \in \{ \rho_{k} \neq 1 \}$, then \eqref{eq:og12} implies that
\begin{align}
\label{eq:sig11}
\frac{p_{k}(x)}{\rho_{k}(x)} + \Sb''(\rho_{k}(x)) = \Sa''(\rho_{k}(x)) + \Sb''(\rho_{k}(x)) = \cS''(\rho_{k}(x)).
\end{align}
From \eqref{eq:n}, we conclude 
\begin{align}
\label{eq:sig111}
| p_{k} + \rho_{k} \Sb''(\rho_{k}) | \geq  \frac{1}{\s_2} \hbox{ a.e. in } \{ \rho_{k} \neq 1 \}.
\end{align}
Recall that $\rho_{k}$ is Lipschitz continuous (cf. Lemma~\ref{lem:lipg}) and thus
\begin{align*}
\nabla \rho_{k} = 0  \ \ \ae \hbox{ in } \{ \rho_k = 1\}
\end{align*}
(see for instance \cite[Theorem 4.(iv), Section 4.2.2]{EvaGar92}). Therefore, we conclude \eqref{eq:1sig}.

\end{proof}

\begin{proof}[Proof of Theorem~\ref{thm:gm}]

As and initial observation, let us remark that by similar arguments as in Lemma \ref{lem:coi}, one obtains the same estimates for the continuous in time interpolations $(\tilde\rho^\t,\tilde\v^\t,\tE^\t)$, and by passing to the limit as $\t\da 0$, we obtain a continuity equation of the form 
$$\partial_t\rho+\diver\E=0.$$
Since the limits of $(\tilde\rho^\t,\tE^\t)$ and $(\rho^\t,\E^\t)$ are the same, it remains to identify the limit of the latter one to get the precise form of our limit equation.

\medskip

\textbf{Step 1.} From direct computation as in \eqref{eq:exi12}, we obtain that
\begin{align}
\label{eq:gm11}
- \Et =  \rhot \nabla(\Sb'(\rhot) + \pt(1+ \log\rhot)) + \rho^\t\nabla\Phi = \nabla ( \rhot \Sb'(\rhot) - \Sb(\rhot) + \Sb(1) + \pt \rhot) + \rho^\t\nabla\Phi.
\end{align}
From Proposition~\ref{prop:cog} and Corollary \ref{cor:ent-por} we can claim that
\begin{align}
\label{eq:gm22}
\nabla (\rhot \Sb'(\rhot) - \Sb(\rhot) + \Sb(1) + \pt \rhot) \to \nabla (\rho \Sb'(\rho) - \Sb(\rho) + \Sb(1) + p \rho),
\end{align}
as $\t \da 0$ in the sense of distribution.
Indeed, using the strong $L^1(Q)$ compactness of $(\rho^\t)_{\t>0}$ and the weak-$\star$ compactness of $(p^\t)_{\t>0}$ in $L^\infty(Q)$, we can pass to the limit $\rho^\t p^\t$. 
Recall that $(\rho^\t)_{\t>0}$ in uniformly bounded in $L^\beta(Q)$ for $\beta$ given in \eqref{eq:beta}. As $r < \beta$, Corollary~\ref{cor:ent-por} yields the convergence of $(\rho^\t)_{\t>0}$ in $L^r(Q)$. As the growth rate of $\rho \Sb'(\rho)$ and $\Sb(\rho)$ is $r$, we conclude that $\rhot \Sb'(\rhot) - \Sb(\rhot)\to\rho \Sb'(\rho) - \Sb(\rho)$ in $L^1(Q)$ as $\t\da 0$.

\medskip

\textbf{Step 2.} Let us show that 
\begin{align}
\label{eq:gm31}
\rho \Sb'(\rho) - \Sb(\rho) + \Sb(1)  + p \rho = L_S(\rho,p),
\end{align}
By parallel arguments as in Lemma~\ref{lem:rel}, we conclude that $(\rho,p)$ satisfies \eqref{eq:2gm}. Thus, it holds that
\begin{align}
\label{eq:gm32}
\rho \Sa'(\rho) - \Sa(\rho) = p \rho, \hbox{ a.e. in } \rho^{-1}(\Ro)
\end{align}
and we conclude \eqref{eq:gm31} a.e. in $\rho^{-1}(\Ro)$. From \eqref{eq:gm22} and \eqref{eq:gm31}, we conclude \eqref{eq:1gm}.

\medskip

Furthermore, from Lemma~\ref{lem:sab}, we obtain that
\begin{align*}
\rho \Sb'(\rho) - \Sb(\rho) + \Sb(1)  + p \rho = p &\hbox{ in } \rho^{-1}(\{1\}).
\end{align*}
and we conclude \eqref{eq:gm31} a.e. in $\rho^{-1}(\{1\})$.
\end{proof}

In particular, \eqref{eq:1gm} can be also represented in the form of a continuity equation, as we show below. 

\begin{theorem}
\label{cor:modg}
Suppose that \eqref{eq:g2} and \eqref{eq:n} hold true. Let $\rho_0$ and $(\rho,p)$ be given in Theorem \ref{thm:gm}. 
If
\begin{align}
\label{eq:0modg}
r \ge \frac{3d-4}{2d}, 
\end{align}
then $(\rho,p)$ also satisfies 
\be
\label{eq:modg}
\left\{
\ba{ll}
\partial_t\rho - \nabla \cdot \left(\rho \nabla \left( \cS'(\rho)\one_{ \{ \rho \neq 1 \} } + p\one_{\{ \rho=1 \}} \right) \right)-\nabla\cdot(\rho\nabla\Phi) =0, & {\rm{in}}\ (0,T)\times\Om,\\
\rho(0,\cdot)=\rho_0, & {\rm{in}}\ \Om,\\
\rho\left[ \nabla \left( \cS'(\rho)\one_{ \{ \rho \neq 1 \} } + p\one_{\{ \rho=1 \}} \right)+\nabla\Phi\right] \cdot \n = 0 , & {\rm{in}}\ [0,T] \times \partial \Om,
\ea
\right.
\ee
in the sense of distribution. If in addition $\rho_0\in L^\infty(\Om)$ and $\Phi$ satisfies \eqref{as:Phi2}, we can drop \eqref{eq:g2} and \eqref{eq:0modg} from the statement.
\end{theorem}

\begin{remark}
Note that \eqref{eq:0modg} is equivalent to $\beta \geq 2$ for $\beta$ given in \eqref{eq:beta} and the inequality holds for any $r \geq 1$ if $d=1$ or $d=2$.
\end{remark}

\begin{proof}
Note that \eqref{eq:2gm} and \eqref{eq:sa} imply that $\cS_a'(\rho) = p(1+ \log\rho) \hbox{ in } \rho^{-1}(\Ro)$. Furthermore, from \eqref{eq:sb} and \eqref{eq:1sab}, it holds that
\begin{align}
\label{eq:mod11} 
\I_1:= \cS'(\rho)\one_{ \{ \rho \neq 1 \} } + p\one_{\{ \rho=1 \}} = p(1+ \log\rho) + \Sb'(\rho). 
\end{align}
From $p \in L^2([0,T];H^1(\Om))$ and \eqref{eq:2gm}, we obtain
\begin{align}
\label{eq:mod12}
\rho \log \rho \nabla p = 0 \ \ae
\end{align}
From \eqref{eq:mod11} and \eqref{eq:mod12}, we have
\begin{align*}
\rho \nabla \I_1 = \rho \nabla p + p \nabla \rho + \rho \nabla(\Sb'(\rho)).
\end{align*}

\medskip

Next, we claim that 
\begin{align*}
\rho \nabla p + p \nabla \rho + \rho \nabla(\Sb'(\rho)) \in L^1(Q).
\end{align*}
Consider the first term $\rho \nabla p$. Recall from Theorem~\ref{thm:gm} that $\nabla p \in L^2(Q)$. 
If $\rho_0 \in L^\infty(\Om)$, then $\rho \in L^\infty(Q)$ from Lemma~\ref{lem:Linfty} and thus  $\rho \nabla p \in L^1(Q)$. On the other hand, if \eqref{eq:0modg} is fulfilled, then $\beta$ given in \eqref{eq:beta} is greater than or equal to 2. As $\rho \in L^\beta(Q)$ from Lemma~\ref{lem:uni}, we obtain $\rho \in L^2(Q)$ and thus $\rho \nabla p \in L^1(Q)$.

\medskip

Furthermore, as $\nabla \sqrt{\rho} \in L^2(Q)$, $\nabla \rho = 2 \rho^{\frac12}\nabla\sqrt{\rho}  \in L^1(Q)$ and the second term is in $L^1(Q)$. Lastly, 
\begin{align*}
\rho \nabla(\Sb'(\rho)) = 2 \rho^{\frac{3}{2}} \Sb''(\rho) \nabla \sqrt{\rho}.
\end{align*}
As the growth rate of $\rho^{\frac{3}{2}} \Sb''(\rho)$ is $r - \frac12$ and $r - \frac12 \leq \frac{\beta}{2}$, $\rho \in L^{\beta}(Q)$ implies $\rho^{\frac{3}{2}} \Sb''(\rho) \in L^2(Q)$ and the last term is in $L^1(Q)$.

\medskip

Lastly, we have
\begin{align*}
\rho\nabla\I_1 = \nabla (\rho p + \rho \Sb'(\rho) - \Sb(\rho) + \Sb(1)) = \nabla L_S(\rho,p)
\end{align*}
for $L_S$ given in \eqref{eq:ls}. By Theorem~\ref{thm:gm}, we conclude that $(\rho,p)$ is a weak solution of \eqref{eq:modg}.
\end{proof}


\section{Porous medium type diffusion on $\{\rho<1\}$ and general diffusion on $\{\rho>1\}$}
Similarly to the classical porous medium equation, in this section we do not expect solutions to be fully supported. As in Section \ref{sec:gfl}, let us first study an example with a particular nonlinearity. 

\subsection{Same diffusion exponent}
\label{sec:51}
In this subsection, we suppose that $\cS : \Rpz \rightarrow \R$ is defined by 
\begin{align}
\label{eq:slp}
\ds\cS(\rho):= \begin{cases}
\ds\frac{\rho^m}{m-1}, &\hbox{ for } \rho \in [0,1],\\[5pt]
\ds\frac{2 \rho^m}{m-1} - \frac{1}{m-1} , &\hbox{ for } \rho\in(1,+\infty).
\end{cases}
\end{align}
where $m > 1$.
\medskip

Our main theorem in this section can be formulated as follows.

\begin{theorem}
\label{thm:exip}
For $\rho_0 \in \sP(\Om)$ such that $\cJ(\rho_0) < +\infty$ and $S$ given in \eqref{eq:slp}, there exists 
$\rho\in L^{\beta}(Q) \cap {{\rm{AC}}^2([0,T];(\sP(\Om),W_2))}$ and $p\in L^2([0,T]; H^1(\Om))\cap L^\infty(Q)$ with $\rho^{m-\frac12} \in L^2([0,T]; H^1(\Om))$ such that $(\rho,p)$ is a weak solution of
\be\label{eq:exip}
\left\{
\ba{ll}
\partial_t\rho -  \Delta ( [(m-1)\rho^m +1]  \frac{p}{m})-\nabla\cdot(\nabla\Phi\rho) =0, & {\rm{in}}\ (0,T)\times\Om,\\[5pt]
\rho(0,\cdot)=\rho_0, & {\rm{in}}\ \Om,\\[5pt]
(\nabla ( [(m-1)\rho^m +1]  \frac{p}{m} )+\nabla\Phi\rho) \cdot \n = 0 , & {\rm{in}}\ [0,T] \times \partial \Om,
\ea
\right.
\ee
in the sense of distribution. Furthermore, $(\rho,p)$ satisfies
\begin{align}
\label{eq:relp}
\begin{cases}
p(t,x)=\frac{m}{m-1} &\hbox{ {\rm{a.e. in}} } \{0 < \rho(t,x) < 1\},\\[5pt]
p(t,x) \in \left[\frac{m}{m-1},\frac{2m}{m-1}\right] &\hbox{ \rm{a.e. in} } \{\rho(t,x) = 1\},\\[5pt]
p(t,x)=\frac{2m}{m-1} &\hbox{ \rm{a.e. in} } \{\rho(t,x) > 1\}.
\end{cases}
\end{align}
In addition, if $\rho_0\in L^\infty(\Om)$ and $\Phi$ satisfies \eqref{as:Phi2}, then $\rho\in L^\infty(Q)$ and $\rho^{m} \in L^2([0,T]; H^1(\Om))$.
\end{theorem}

Let us recall the definition of $(\rho_k)_{k=1}^N$ and $(p_k)_{k=1}^N$ from \eqref{eq:step} and \eqref{eq:p}, respectively. Let us underline that in the setting of this section due to the structure of the nonlinearity we typically expect $\spt(\rho_k)$ to be a proper subset of $\Om$, unlike in the case of Lemma~\ref{lem:postg} which was used in Section~\ref{sec:gfl} and Section~\ref{sec:gfn}. For this reason, we expect the Lipschitz continuity of $\rho_k^{m-1}$ instead of $\rho_k$. The following result is a simplified version of Lemma \ref{lem:lipg-gen}, tailored to the entropy function \eqref{eq:slp}.
\begin{lemma}
\label{lem:poptp}
For all $k \in \{1,\dots,N\}$, there exists $C\in\R$ such that
\begin{align}
\label{eq:poptp}
\rho_{k}^{m-1} p_{k}=\left( \cC - \frac{\ophi_{k}}{\tau}-\Phi \right)_{+} \ \ \ae
\end{align}
In particular, $\rho_{k}^{m-1}$ is Lipschitz continuous. Here, $\ophi_{k}$ is given in Theorem~\ref{thm:opg}.
\end{lemma}

\begin{proof}
Note that 
\begin{align}
\label{eq:dsp}
\partial \cS(\rho) =  
\begin{cases}
\frac{m}{m-1}\rho^{m-1} &\hbox{ for } 0< \rho <1,
\\[5pt] \left[\frac{m}{m-1},\frac{2m}{m-1}\right] 	&\hbox{ for } \rho =1,
\\[5pt] \frac{2m}{m-1}\rho^{m-1} &\hbox{ for } \rho >1.
\end{cases}
\hbox{ and\ \ \ }
p_{k} = 
\begin{cases}
\frac{m}{m-1} &\hbox{ in } \rho_{k}^{-1}([0,1)),\\
\cC - \frac{\ophi_{k}}{\t}-\Phi &\hbox{ in } \rho_{k}^{-1}(\{1\}),\\
\frac{2m}{m-1} &\hbox{ in } \rho_{k}^{-1}((1,+\infty)).
\end{cases}
\hbox{ \quad a.e.}
\end{align}
for $p_{k}$ given in \eqref{eq:p}.
Then, Theorem~\ref{thm:opgr} implies that
\begin{align}\label{eq:6cond}
\rho_{k}^{m-1} p_{k}  + \frac{\ophi_{k}}{\tau}+\Phi =  \cC \hbox{ a.e. on } \spt(\rho_{k})
\end{align}
for some constant $C \in \R$.

\medskip

Moreover, if $\rho_{k} = 0$ a.e. on some set $A\subset\Om$, then Theorem~\ref{thm:opgr} and $\cS'(0+) = 0$ (from \eqref{eq:dsp}) imply that
\begin{align*}
\cC - \frac{\ophi_{k}}{\tau} -\Phi \leq 0 \ \ae\ {\rm{in}}\ A,
\end{align*}
and we conclude \eqref{eq:poptp}.

\medskip

Next, recall that Lemma~\ref{lem:lip_new} yields that $\ophi_{k}$ is Lipschitz continuous and thus $\left( \cC - \frac{\ophi_{k}}{\tau} -\Phi \right)_{+}$ is Lipschitz continuous as well. As Lemma~\ref{lem:lip_new} yields the Lipschitz continuity of $p_k$, and since this has a positive lower bound $\frac{m}{m-1}$ (from \eqref{eq:dsp} and \eqref{eq:1opg}), we conclude that $\rho^{m-1}_{k}$ is also Lipschitz continuous.
\end{proof}

\begin{lemma}\label{lem:conp}
Let $(\rho^\t)_{\t>0}, (p^\t)_{\t>0}$ stand for the piecewise constant interpolations given in \eqref{eq:int} and \eqref{eq:p}, respectively. Then  
$((\rho^\t)^{m-\frac12})_{\t>0}$ and $(p^\t)_{\t>0}$ are uniformly bounded in $L^2([0,T]; H^1(\Om))$.
\end{lemma}

\begin{proof}

From Lemma~\ref{lem:poptp}, it holds that
\begin{align}
\label{eq:conp11}
\I_1: =-\rho_{k}^{\frac{1}{2}}\nabla\Phi - \rho_{k}^{\frac{1}{2}} \frac{\nabla \ophi_{k}}{\t} = \rho_{k}^{\frac{1}{2}} \nabla( \rho_{k}^{m-1} p_{k}) \ \ \ae
\end{align}
As $p_k$ and $\rho_k^{m-1}$ are Lipschitz continuous from Lemma~\ref{lem:poptp}, we have
\begin{align}
\label{eq:conp12}
\I_1 = \rho_{k}^{\frac{1}{2}} p_{k} \nabla( \rho_{k}^{m-1} )  + \rho_{k}^{m - \frac{1}{2}} \nabla p_{k} \ \ \ae \hbox{ on } \spt(\rho_{k}).
\end{align}
Furthermore, since we have the Lipschitz continuity of $\rho_k^{m-1}$ and \eqref{eq:dsp}, we apply a parallel argument as in the proof of Proposition~\ref{prop:con} and conclude that
\begin{align}
\label{eq:conp13}
(\rho_{k}^{m - \frac{1}{2}} - 1) \nabla p_{k} = 0 \hbox{ and } \nabla( \rho_{k}^{m-1} ) \cdot \nabla p_{k} = 0 \ \ \ae \hbox{ on } \Om.
\end{align}
From \eqref{eq:conp12} and \eqref{eq:conp13}, we have that
\begin{align}
\label{eq:conp14}
\I_1^2 =  p_{k}^2 |\rho_k^{\frac{1}{2}} \nabla( \rho_{k}^{m-1} )|^2  + |\nabla p_{k}|^2 \ \ \ae \hbox{ on } \spt(\rho_{k}).
\end{align}
As $p_k \geq \frac{m}{m-1}$ a.e. in $\Om$ as in \eqref{eq:dsp}, we conclude that
\begin{align}
\label{eq:conp15}
\I_1^2 \geq \left( \frac{m}{m-1} \right)^2  |\rho_k^{\frac{1}{2}}\nabla( \rho_{k}^{m-1} )|^2 + |\nabla p_{k}|^2 \ \ \ae \hbox{ on } \spt(\rho_{k}).
\end{align}

\medskip

\eqref{eq:conp13} yields that $\nabla p_{k} = 0$ a.e. on $\spt(\rho_{k})^c = \{ \rho_{k} = 0 \}$. Furthermore, as $\rho_{k}^{m-1}$ is Lipschitz continuous (see Lemma~\ref{lem:poptp}), we have
\begin{align*}
\rho_k^{\frac{1}{2}}\nabla( \rho_{k}^{m-1} ) = 0 \ \ \ae \hbox{ on } \spt(\rho_{k})^c.
\end{align*}
Therefore, \eqref{eq:conp15} holds a.e. on $\Om$.

\medskip

On the other hand, applying Lemma~\ref{lem:gfi}, it holds that
\begin{align*}
\int_0^T \int_\Om \I_1^2 \dd x\dd t  \leq 2\left(\cJ(\rho_0) - \inf \cJ\right)+T\sL^d(\Om)\|\nabla\Phi\|_{L^\infty}.
\end{align*}
As $\rho_k^{\frac{1}{2}}\nabla( \rho_{k}^{m-1}) = \frac{m-1}{m - \frac{1}{2}} \nabla( \rho_{k}^{m-\frac{1}{2}})$ and 
since $(\rho^\t)_{\t>0}$ is uniformly bounded in $L^\b(Q)$ (with $\b>m-1/2$, see Lemma \ref{lem:apriori-2star})
we conclude that $((\rho^\t)^{m-\frac12})_{\t>0}$ and $(p^\t)_{\t>0}$ are uniformly bounded in $L^2([0,T]; H^1(\Om))$ (since  $(p^\t)_{\t>0}$ is also uniformly bounded) and therefore we conclude.
\end{proof}

As a consequence of Lemma~\ref{lem:conp} and Lemma~\ref{lem:al}, we have the following convergence.
\begin{corollary}\label{cor:por-por}
Let $(\rho^\t)_{\t>0}$ and $(p^\t)_{\t>0}$ be as in the previous lemma. Then, there exists $\rho\in L^{m}(Q)$ and $p\in L^2([0,T]; H^1(\Om))$ with $\rho^{m-\frac12} \in L^2([0,T]; H^1(\Om))$,  such that 
\begin{align*}
\rho^\t \to \rho\ {\rm{in}}\ L^m(Q),\ {\rm{as}}\  \t\da 0,
\end{align*}
and
\begin{align*}
p^\t \weakly p\ {\rm{in}}\ L^2([0,T]; H^1(\Om)), \ {\rm{as}}\  \t\da 0.
\end{align*}
along a subsequence. 
\end{corollary}

\begin{proof}[Proof of Theorem~\ref{thm:exip}]

Note that \eqref{eq:dsp} implies \eqref{eq:relp} for $(\rho^\t, p^\t)$. Then, a similar argument as the one in Lemma \ref{lem:rel} together with the convergence results from Corollary~\ref{cor:por-por} reveals that $(\rho,p)$ satisfies \eqref{eq:relp}.

\medskip

Furthermore, from Lemma~\ref{lem:poptp},
we can write that
\begin{align*}
\E^\t = \rho^\t \v^\t = - \rho^\t \nabla((\rho^\t)^{m-1} p^\t) -\nabla\Phi\rho^\t = - \left\{ (m-1) p^\t (\rho^\t)^{m-1} \nabla \rho^\t + (\rho^\t)^m \nabla p^\t \right\} - \nabla\Phi\rho^\t.
\end{align*}
Note that \eqref{eq:dsp} implies
\begin{align}
\label{eq:exip22}
( (\rho^\t)^m - 1) \nabla p^\t  = 0  \ \ \ae
\end{align}
From \eqref{eq:exip22}, we conclude that
\begin{align}
\nonumber(m-1)p^\t (\rho^\t)^{m-1} \nabla \rho^\t + (\rho^\t)^m \nabla p^\t &= (m-1)p^\t (\rho^\t)^{m-1} \nabla \rho^\t + \frac{1}{m}\left\{ (m-1) (\rho^\t)^m + 1 \right\} \nabla p^\t,\\
&= \frac{1}{m} \nabla \left( [(m-1)(\rho^\t)^m +1] p^\t \right).
\label{eq:exip23}
\end{align}
As described in Proposition~\ref{prop:e}, up to passing to a subsequence and using the weak-$\star$ convergence of $(p^\t)_{\t>0}$ in $L^\infty(Q)$ and strong convergence of $((\rho^\t)^m)_{\t>0}$ in $L^1(Q)$ from Corollary~\ref{cor:por-por}, we conclude that $(\E^\t)_{\t>0}$ converges to 
\begin{align*}
\E:= - \frac{1}{m} \nabla \left( [(m-1)\rho^m +1] p \right) - \nabla\Phi\rho 
\end{align*}
in $\sD'(Q;\R^d)$, as $\t\da 0$ where $(\rho,p)$ is given in Corollary~\ref{cor:por-por}. The rest of argument is parallel to the proof of Theorem~\ref{thm:exi}.

\medskip

A last remark is that if $\rho_0\in L^\infty(\Om)$, then clearly $\rho\in L^\infty(Q)$ and thus $\rho^{m} \in L^2([0,T]; H^1(\Om))$.
\end{proof}

In particular, \eqref{eq:exip} can be also represented in the form of a continuity equation, as we show below. Note that the condition \eqref{eq:0emodg} below is equivalent to $\beta \geq 2m$.

\begin{theorem}
\label{cor:emodp}
For $S$ given in \eqref{eq:slp}, let $\rho_0$ and $(\rho,p)$ be given in Theorem \ref{thm:exip}. 
If
\begin{align}
\label{eq:0emodg}
d \leq 4m, 
\end{align}
then $(\rho,p)$ also satisfies
\be\label{eq:emodp}
\left\{
\ba{ll}
\partial_t\rho - \nabla \cdot \left(\rho \left[\nabla \left( \rho^{m-1} p \right)+\nabla\Phi\right] \right) =0, & {\rm{in}}\ (0,T)\times\Om,\\
\rho(0,\cdot)=\rho_0, & {\rm{in}}\ \Om,\\
\rho [\nabla \left(\rho^{m-1} p \right)+\nabla\Phi] \cdot \n = 0 , & {\rm{in}}\ [0,T] \times \partial \Om,
\ea
\right.
\ee
in the sense of distribution. If in addition $\rho_0\in L^\infty(\Om)$ and $\Phi$ satisfies \eqref{as:Phi2}, we can drop \eqref{eq:0emodg} from the statement.
\end{theorem}

\begin{proof}
As $p \in L^2([0,T]; H^1(\Om))$ and $(\rho,p)$ satisfies \eqref{eq:relp} from Theorem~\ref{thm:exip}, we have
\begin{align}
\label{eq:emodp10}
\nabla p  = 0 \ \ \ae \hbox{ in } \{ \rho \neq 1 \}
\end{align}
and thus
\begin{align}
\label{eq:emodp11}
(\rho^m -1) \nabla p  = 0 \ \ \ae
\end{align}
From the direct computation using \eqref{eq:emodp11}, it holds that
\begin{align*}
\I_1:=\rho \nabla \left( \rho^{m-1} p \right) &= (\rho^m)^{\frac{1}{m}}  \nabla \left( (\rho^m)^{\frac{m-1}{m}} p \right) = \rho^m \nabla p + \frac{m-1}{m} p \nabla (\rho^m).
\end{align*}

\medskip

We claim that $\I_1 \in L^1(Q)$, which is enough for the representation \eqref{eq:emodp}. Recall from Theorem~\ref{thm:exip} and Lemma~\ref{lem:uni}, $\nabla p \in L^2(Q)$, $p \in L^\infty(Q)$ and $\rho\in L^{\beta}(Q)$ for $\beta$ given in \eqref{eq:beta}. Consider the first term $\rho^m \nabla p$. If $\rho_0 \in L^\infty(\Om)$, then $\rho \in L^\infty(Q)$ and thus $\rho^m \in L^2(Q)$. If \eqref{eq:0emodg} holds (which is automatically the case if $d=1,2$), then $\beta \geq 2m$ for $\beta$ given in \eqref{eq:beta} ($r=m$ in this case) and thus $\rho^m \in L^2(Q)$. Furthermore, as $\nabla \rho^{m - \frac12} \in L^2(Q)$ and
\begin{align*}
\nabla (\rho^m) = \frac{m}{m-\frac12} \rho^{\frac12} \nabla (\rho^{m - \frac12}),
\end{align*}
so the last term is also in $L^1(Q)$.

\medskip

Lastly, it is easy to see that 
\begin{align*}
\I_1 &= \frac{(m-1)\rho^m + 1}{m}  \nabla p + \frac{m-1}{m} p \nabla (\rho^m)\\
&= \nabla \left( [(m-1)\rho^m +1]  \frac{p}{m}\right).
\end{align*}

\end{proof}

\subsection{Porous medium type diffusion on $\{\rho<1\}$ and general diffusion on $\{\rho>1\}$}

In this subsection, we suppose that Assumption~\ref{as:gen}, Assumption \ref{as:intro_main1} and Assumption \ref{as:intro-19} hold true for some $r\geq1$, for some $m>1$ and a constants $\s_1,\s_2 >0.$
Note that $\cS$ can be any function satisfying the assumptions, and in particular in the case of $r=1$, $S$ behaves as the logarithmic entropy when $\rho>1.$

\medskip

Our main theorem from this section reads as:
\begin{theorem}
\label{thm:gmpg}
Suppose that Assumption~\ref{as:gen}, Assumption \ref{as:intro_main1} and Assumption \ref{as:intro-19} hold true for $m>1$ and $r\ge 1$ such that
\begin{align}
\label{eq:mr}
m <  r + \frac{\beta}{2}
\end{align} is fulfilled for $\beta$ given in \eqref{eq:beta}.
For $\rho_0 \in \sP(\Om)$ such that $\cJ(\rho_0) < +\infty$, there exists $\rho\in L^{\beta}(Q)$ and $p\in L^2([0,T]; H^1(\Om))\cap L^\infty(Q)$ such that $(\rho,p)$ is a weak solution of
\be\label{eq:1gmpg}
\left\{
\ba{ll}
\partial_t\rho - \Delta (L_S(\rho,p)) -\nabla\cdot(\nabla\Phi\rho) =0, & {\rm{in}}\ (0,T)\times\Om,\\
\rho(0,\cdot)=\rho_0, & {\rm{in}}\ \Om,\\
(\nabla (L_S(\rho,p) )+\nabla\Phi\rho ) \cdot \n = 0 , & {\rm{in}}\ [0,T] \times \partial \Om,
\ea
\right.
\ee
in the sense of distribution. Furthermore, $(\rho,p)$ satisfies
\begin{align}
\label{eq:2gm2}
\begin{cases}
p(t,x)=\cS'(1-) &\hbox{ {\rm{if}} } 0 \leq \rho(t,x) < 1,\\
p(t,x) \in [\cS'(1-),\cS'(1+)] &\hbox{ {\rm{if}} } \rho(t,x) = 1,\\
p(t,x)=\cS'(1+) &\hbox{ {\rm{if}} } \rho(t,x) > 1.
\end{cases}
\end{align}
Here, $L_S$ is given in \eqref{eq:ls}. In particular, 
\begin{align}
\label{eq:2gmpgg}
\rho^{m - \frac12} \in L^2([0,T]; H^1(\Om)) \hbox{ if } m \leq r \hbox{ and } \rho^{m-\frac12} \in L^q([0,T]; W^{1,q}(\Om)) \hbox{ if } r < m <  r + \frac{\beta}{2}
\end{align}
for $q \in (1,2)$ given in \eqref{eq:q}. If in addition $\rho_0\in L^\infty(\Om)$ and $\Phi$ satisfies \eqref{as:Phi2}, we can drop \eqref{eq:g2} and \eqref{eq:mr} from the statement and we obtain $\rho \in L^\infty(Q)$ and $\rho^m \in L^2([0,T];H^1(\Om))$.
\end{theorem}

\begin{example} As a nonlinearity, one can consider for instance the following one.
For $m>r>1$, let $\cS : \Rpz \to \R$ given by
\begin{align*}
\cS(\rho):= \begin{cases}
\frac{\rho^m}{m-1}, &\hbox{ for } \rho \in [0,1],\\
\frac{\rho^r}{r-1} + \frac{1}{m-1} - \frac{1}{r-1}, &\hbox{ for } \rho\in(1,+\infty).
\end{cases}
\end{align*}
This clearly satisfies Assumption~\ref{as:gen} and Assumption \ref{as:intro_main1}, since 
\begin{align*}
\cS'(1-) = \frac{m}{m-1} < \frac{r}{r-1} = \cS'(1+).
\end{align*}
In this case, the operator $L_\cS(\rho,p)$ becomes
\begin{align*}
L_\cS(\rho,p)(x) =
\begin{cases}
\rho(x)^m + \frac{1}{m-1} &\hbox{ {\rm{if}} } 0 < \rho(x) < 1,\\
p(x) \in \left[ \frac{m}{m-1}, \frac{r}{r-1} \right] &\hbox{ {\rm{if}} } \rho(x) = 1,\\
\rho(x)^r + \frac{1}{r-1} &\hbox{ {\rm{if}} } \rho(x) >1.
\end{cases}
\end{align*}
\end{example}

First, using similar ideas as in Section~\ref{sec:gfn}, we choose a constant $l$ such that
\begin{align}\label{eq:l}
1 < l < \beta
\end{align}
for $\beta$ given in \eqref{eq:beta}
and split the function $S$ into $\Sa$ and $\Sb : \Rpz \to \R$ defined by
\begin{align}
\label{eq:sapg}
\Sa(\rho) := 
\begin{cases}
\frac{\cS'(1-)(\rho^l-1)}{l},  &\hbox{ for } \rho \leq 1, \\[5pt]
\frac{\cS'(1+)(\rho^l-1)}{l},   &\hbox{ for } \rho > 1, 
\end{cases}
\end{align}
and
\begin{align}
\label{eq:sbpg}
\Sb(\rho) := \cS(\rho) - \Sa(\rho).
\end{align}
Note that $\cS'(1+) > \cS'(1-)$. Then, as shown in Lemma~\ref{lem:sab}, we conclude that $\Sa$ is convex and continuous in $\Rpz$. Also, $\Sb$ is continuously differentiable and $\Sb'$ is locally Lipschitz continuous in $\Rpz$.

\medskip
\medskip

Let us recall the definition of $(\rho_k)_{k=1}^N$ and $(p_k)_{k=1}^N$ from \eqref{eq:step} and \eqref{eq:p}. Also, recall the definition of $\ophi_{k}$ given in Theorem~\ref{thm:opg}.

\begin{lemma}
\label{lem:poptpg}
For all $k \in \{1,\dots,N\}$, there exists $C\in\R$ such that
\begin{align}
\label{eq:poptpg}
\rho_{k}^{l-1} p_{k} + \Sb'(\rho_{k}) =\left( \cC - \frac{\ophi_{k}}{\tau} -\Phi \right)_{+} \ \ \ae
\end{align}
In particular, $p_{k}$ and $\rho_{k}^{m-1}$ are Lipschitz continuous in $\Om$. If in addition $m\in(1,2]$, then $\rho_{k}$ is locally Lipschitz continuous in $\spt(\rho_{k})$. 
\end{lemma}

\begin{proof}

First, \eqref{eq:poptpg} follows from the parallel argument in the proof of Lemma~\ref{lem:poptp}. 

Notice that $\ophi_k$ and $p_k$ are Lipschitz continuous (cf. Lemma~\ref{lem:lip_new}) and $f_k := \cC - \frac{\ophi_{k}}{\tau}-\Phi$ are Lipschitz continuous. From \eqref{deriv:formula-gen} in Lemma~\ref{lem:lipg-gen}, we have that
\begin{align*}
|\nabla (\rho_{k})^{m-1} | = (m-1) \rho_{k}^{m-2} |  \nabla \rho_{k}| = (m-1) \rho_{k}^{m-2} \frac{ |\nabla f_k| }{S''(\rho_k)}  \ \ae \hbox{ in } \rho_{k}^{-1}(\Ro).
\end{align*}

On the one hand, this, together with \eqref{eq:g1} from Assumption \ref{as:intro_main1} further implies
\begin{align*}
|\nabla (\rho_{k})^{m-1} | &\leq  \s_2(m-1) |\nabla f_k| \ \ \ae \hbox{ in } \{ x \in \Om : 0< \rho_k < 1\}.
\end{align*}
On the other hand, \eqref{eq:2as22} from Assumption \ref{as:intro-19} yields
\begin{align*}
 |\nabla (\rho_{k})^{m-1} | &\leq  \s_1(m-1) |\nabla f_k| \rho_k^{m-r} \leq \s_1(m-1) |\nabla f_k|    \max\{\|\rho_{k}\|_{L^\infty(\Om)}^{m-r}, 1 \}. \ \ \ae \hbox{ in } \{ x \in \Om : \rho_k > 1\}.
\end{align*}
Therefore, we conclude that $\rho_{k}^{m-1}$ is Lipschitz continuous in $\Om$.

\medskip

Lastly, 
the following identity
\begin{align*}
|  \nabla \rho_{k}| = \frac{\rho_{k}^{2-m}}{(m-1)} |\nabla (\rho_{k})^{m-1} | \ \ \ae \hbox{ in } \spt(\rho_{k})
\end{align*}
shows that $\rho_{k}$ is locally Lipschitz continuous in $\spt(\rho_{k})$, provided $m\in(1,2]$. This in particular is also a consequence of Lemma \ref{lem:lipg-gen}(iv).
\end{proof}

Proposition \ref{prop:conpg} below contains all the needed estimates and the compactness result on the sequence $(\rho^\t,p^\t)_{\t>0}$ that are necessary to pass to the limit as $\t\da 0$ and prove the main theorem of this section, i.e. Theorem \ref{thm:gmpg}. The proof of this proposition requires some intermediate results that are provided below in  Lemmas \ref{lem:sigpg}, \ref{lem:lbd} and \ref{lem:tt}.

\begin{proposition} \label{prop:conpg}
Let $(\rho^\t)_{\t>0}, (p^\t)_{\t>0}$ stand for the piecewise constant interpolations given in \eqref{eq:int} and \eqref{eq:p}, respectively. Then, $(p^\t)_{\t>0}$ is uniformly bounded in $L^2([0,T]; H^1(\Om))$.
\begin{itemize}
\item[(1)] If $r \geq m$, then $((\rho^\t)^{m-\frac12})_{\t>0}$ is uniformly bounded in $L^2([0,T]; H^1(\Om))$. 
\item[(2)] If $r < m < r + \frac{\beta}{2}$, then $((\rho^\t)^{m-\frac12})_{\t>0}$ is uniformly bounded in $L^{q}([0,T]; W^{1,q}(\Om))$ for some $q \in (1,2)$.
\item[(3)] If in addition $\rho_0\in L^\infty(\Om)$ and $\Phi$ satisfies \eqref{as:Phi2}, then $((\rho^\t)^{m})_{\t>0}$ is also uniformly bounded in $L^2([0,T];H^1(\Om))$ for any $m >1$ and $r\geq 1$.
\end{itemize}
\end{proposition}

\begin{proof}
From Lemma~\ref{lem:poptpg}, it holds that
\begin{align*}
\I_1 := -\rho_{k}^{\frac{1}{2}} \frac{\nabla \ophi_{k}}{\t} - \rho_{k}^{\frac{1}{2}}\nabla\Phi = \rho_{k}^{\frac{1}{2}} \nabla( \rho_{k}^{l-1} p_{k} + \Sb'(\rho_{k})) \ \ \ae
\end{align*}
We follow the very same steps and in the proof of Lemma~\ref{lem:poptp} (where we also use \eqref{eq:conp12} and \eqref{eq:conp13}). Therefore, we have
\begin{align}
\label{eq:conpg11}
\I_1 =  \frac{l-1}{m-1} \rho_{k}^{l - m + \frac{1}{2}} p_{k} \nabla( \rho_{k}^{m-1} )  + \rho_k^{l-\frac12} \nabla p_{k} + \rho_{k}^{\frac{1}{2}} \nabla (\Sb'(\rho_{k})) \ \ \ae \hbox{ on } \spt(\rho_{k}).
\end{align}
Note that
\begin{align}
\label{eq:conpg12}
\rho_{k}^{\frac{1}{2}} \nabla (\Sb'(\rho_{k})) = \frac{1}{m-1} \rho_{k}^{\frac{5}{2} - m } \Sb''(\rho_{k}) \nabla( \rho_{k}^{m-1} ) \ \ \ae \hbox{ on } \spt(\rho_{k}).
\end{align}
From \eqref{eq:conpg11} and \eqref{eq:conpg12}, it holds that
\begin{align*}
\I_1 = \frac{1}{(m-1)\rho_k^{m-2}} \left( (l-1) \rho_k^{l-2} p_{k} + \Sb''(\rho_{k}) \right) \rho_k^{\frac{1}{2}} \nabla( \rho_{k}^{m-1} )  + \rho_k^{l-\frac12} \nabla p_{k} \ \ \ae \hbox{ on } \spt(\rho_{k}).
\end{align*}
We can apply \eqref{eq:conp13} and conclude (since $\nabla p_k=0$ a.e. in $\{\rho_k\neq1\}$) that
\begin{align*}
\I_1^2 =  \frac{1}{(m-1)^2\rho_k^{2m-4}} \left( (l-1) \rho_k^{l-2} p_{k} + \Sb''(\rho_{k}) \right)^2 \rho_k |\nabla( \rho_{k}^{m - 1} )|^2  + | \nabla p_{k}|^2 \ \ \ae \hbox{ on } \spt(\rho_{k}).
\end{align*}

(1) If $r \geq m$, then Lemma~\ref{lem:sigpg} below implies 
\begin{align*}
\I_1^2 \geq \frac{\s_3^2}{(m-1)^2} |\nabla( \rho_{k}^{m-\frac12} )|^2  +  | \nabla p_{k}|^2 \ \ \ae \hbox{ on } \spt(\rho_{k}).
\end{align*}
for $\s_3$ given in \eqref{eq:s3}. By the parallel argument in Lemma~\ref{lem:conp}, we conclude the uniform bound in $L^2([0,T]; H^1(\Om))$. 

\medskip

(2) If $r < m < r + \frac{\beta}{2}$, then Lemma~\ref{lem:lbd} below yields the uniform bound of  $(\nabla(\rho^\t)^{m-\frac12})_{\t>0}$ in $L^q(Q)$ for $q$ given in \eqref{eq:q}. On the other hand, as $2r-1 \leq \beta$, it holds that
\begin{align*}
\left(m-\frac12\right) q = \frac{m-\frac12}{\frac{m-r}{\beta} + \frac12} = \beta \frac{2m-1}{2m - 2r + \beta} \leq \beta,
\end{align*}
As $\rho^\t$ is uniformly bounded in $L^\beta(Q)$ from Lemma~\ref{lem:uni}, $(\rho^\t)^{m-\frac12}$ is uniformly bounded in $L^q(Q)$.

\medskip

(3) 
From Lemma~\ref{lem:tt}, we conclude that
\begin{align*}
\I_1^2 \geq \frac{\s_4^2}{(m-1)^2} |\nabla(\rho_k^{m})|^2 + |\nabla p_{k}|^2 \ \ \ae \hbox{ on } \spt(\rho_{k}). 
\end{align*}
The same argument as before yields that $((\rho^\t)^{m})_{\t>0}$ is uniformly bounded in $L^2([0,T]; H^1(\Om))$.

\end{proof}

\begin{lemma}\label{lem:sigpg} 
Let us suppose that we are in the setting of Proposition \ref{prop:conpg}. If $r \geq m$, it holds that
\begin{align}
\label{eq:sigpg}
\left| \frac{1}{\rho_k^{m-2}} \left( (l-1) \rho_k^{l-2} p_{k} + \Sb''(\rho_{k}) \right)  \right| \rho_{k}^{\frac12} |\nabla( \rho_{k}^{m - 1} )| \geq \s_3|\nabla(\rho_k^{m-1/2})|,
\end{align}
where 
\begin{align}
\label{eq:s3}
\s_3:=\frac{m-1}{m-\frac12}\min\left\{\frac{1}{\s_1},\frac{1}{\s_2}\right\}. 
\end{align}
\end{lemma}
\begin{proof}

We claim that
\begin{align}
\label{eq:1sigpg}
\left| \frac{1}{\rho_k^{m-2}} \left( (l-1) \rho_k^{l-2} p_{k} + \Sb''(\rho_{k}) \right)  \right| \geq \min\left\{\frac{1}{\s_1},\frac{1}{\s_2}\right\} \hbox{ in } \{ \rho_k \neq 1 \}.
\end{align}
Recall that 
\begin{align*}
\Sa''(\rho_k) =
\begin{cases}
(l-1) S'(1-) \rho_k^{l-2} \hbox{ if } \rho_k <1,\\
(l-1) S'(1+) \rho_k^{l-2} \hbox{ if } \rho_k >1,\\ 
\end{cases}
\end{align*}
and thus by the definition of $p_k$ (see \eqref{eq:p}) we have
\begin{align}
\label{eq:sigpg13}
(l-1) \rho_k^{l-2} p_{k} + \Sb''(\rho_{k}) = \Sa''(\rho_{k}) + \Sb''(\rho_{k}) = S''(\rho_k) \ \ \ae \hbox{ in } \{ \rho_k \neq 1 \}.
\end{align}
Thus, \eqref{eq:g1} implies that
\begin{align}
\label{eq:sigpg31}
\frac{S''(\rho_k)}{\rho_{k}^{m-2}} \geq  \frac{1 }{\s_2 } \ \ \ae \hbox{ in } \{ 0< \rho_k < 1 \}.
\end{align}
Furthermore, as $r \geq m$, \eqref{eq:2as22} implies
\begin{align}
\label{eq:sigpg32}
\frac{S''(\rho_k)}{\rho_{k}^{m-2}}  \geq  \frac{\rho_{k}^{r - m} }{\s_1} \geq \frac{1 }{\s_1} \ \ \ae \hbox{ in } \{ \rho_k > 1 \}.
\end{align}
and we conclude \eqref{eq:1sigpg}.

\medskip

Recall that $\rho_{k}^{m - 1}$ is Lipschitz continuous from Lemma~\ref{lem:poptpg}. Thus, we have
\begin{align}
\label{eq:sigpg41}
\nabla  (\rho_{k}^{m - 1}) = 0  \ \ \ae \hbox{ in } \{ \rho_k = 1\}
\end{align}
(see for instance \cite[Theorem 4(iv), Section 4.2.2]{EvaGar92}). As $\rho_k^{\frac12} \nabla (\rho_k^{m-1}) = \frac{m-1}{m - \frac12} \nabla (\rho_k^{m-\frac12})$, \eqref{eq:sigpg} follows from \eqref{eq:1sigpg} and \eqref{eq:sigpg41}.
\end{proof}

\begin{lemma}\label{lem:lbd}
Let us suppose that we are in the setting of Proposition \ref{prop:conpg}. If $r < m < r + \frac{\beta}{2}$, then 
\begin{align}
\label{eq:lbd}
\left\| \frac{1}{\rho_k^{m-2}} \left( (l-1) \rho_k^{l-2} p_{k} + \Sb''(\rho_{k}) \right) |\nabla( \rho_{k}^{m - \frac{1}{2}} )| \right\|_{L^2(\Om)} \geq  C \| \nabla (\rho_{k}^{m - \frac{1}{2}}) \|_{L^q(\Om)}
\end{align}
for some $q \in (1,2)$ and a constant $C>0$.
\end{lemma}

\begin{proof}
From the relation between $r$ and $m$, the constant $q$ defined by
\begin{align}
\label{eq:q}
q:= \frac{1}{\frac{m-r}{\beta} + \frac12}
\end{align}
is in the interval $(1,2)$. As shown in \eqref{eq:sigpg13}, it holds that
\begin{align}
\label{eq:lbd11}
\I_2:=  \frac{1}{\rho_k^{m-2}} \left( (l-1) \rho_k^{l-2} p_{k} + \Sb''(\rho_{k}) \right) |\nabla( \rho_{k}^{m - \frac{1}{2}} )| =  \frac{S''(\rho_{k})}{\rho_k^{m-2}}  |\nabla( \rho_{k}^{m - \frac{1}{2}} )| \ \ \ae \hbox{ in } \{\rho_k \neq 1\}. 
\end{align}

\medskip

In $\{0<\rho_k < 1\}$, \eqref{eq:sigpg31} implies that
\begin{align*}
\left\| \I_2 \right\|_{L^2(\{0<\rho_k < 1\})} \geq  \frac{1}{\s_2} \left\| \nabla( \rho_{k}^{m - \frac{1}{2}} ) \right\|_{L^2(\{0<\rho_k < 1\})}
\end{align*}
for $\s_2$ given in \eqref{eq:g1}. As $q \in (1,2)$ and the domain is compact, the H\"{o}lder inequality yields that
\begin{align}
\label{eq:lbd14}
\left\| \I_2 \right\|_{L^2(\{0<\rho_k < 1\})} \geq  \frac{ |\Om|^{\frac12 - \frac{1}{q}} }{\s_2} \| \nabla (\rho_{k}^{m - \frac{1}{2}}) \|_{L^q(\{0<\rho_{k}<1\})}.
\end{align}

\medskip

Next, we claim that
\begin{align}
\label{eq:lbd21}
\left\| \I_2 \right\|_{L^2(\{\rho_k > 1\})} \geq  C \| \nabla (\rho_{k}^{m - \frac{1}{2}}) \|_{L^q(\{\rho_{k}>1)\}}
\end{align}
for some constant $C>0$.

From \eqref{eq:2as22} and \eqref{eq:lbd11}, it holds that
\begin{align}
\label{eq:lbd22}
\left\| \I_2 \right\|_{L^2(\{\rho_k > 1\})} =  \left\|  \rho_{k}^{2 - m} S''(\rho_{k}) \nabla( \rho_{k}^{m - \frac{1}{2}} ) \right\|_{L^2(\{\rho_k > 1\})} \geq \frac{ 1}{\s_1} \left\| \rho_{k}^{r - m} \nabla( \rho_{k}^{m - \frac{1}{2}} ) \right\|_{L^2(\{\rho_k > 1\})}.
\end{align}

On the other hand, as
\begin{align*}
\frac{1}{2} + \frac{m-r}{\beta} = \frac{1}{q},
\end{align*}
the H\"{o}lder inequality yields that
\begin{align}
\label{eq:lbd24}
\left\| \rho_{k}^{r - m} \nabla( \rho_{k}^{m - \frac{1}{2}} ) \right\|_{L^2(\{\rho_k > 1\})} \| \rho_{k}^{m-r} \|_{L^{\frac{\beta}{m-r}}(\{\rho_{k}>1)\}} \geq \| \nabla (\rho_{k}^{m - \frac{1}{2}}) \|_{L^q(\{\rho_{k}>1)\}}.
\end{align}
As $ \rho_{k}$ is uniformly bounded in $L^{\beta}(\Om)$ from Lemma~\ref{lem:apriori-2star}, $\rho_{k}^{m-r}$ is uniformly bounded in $L^{\frac{\beta}{m-r}}(\Om)$. From \eqref{eq:lbd22} and \eqref{eq:lbd24}, we conclude \eqref{eq:lbd21}.

\medskip

Lastly, as \eqref{eq:sigpg41} holds true, \eqref{eq:lbd} follows from \eqref{eq:lbd14} and \eqref{eq:lbd21}.

\end{proof}

\begin{lemma}\label{lem:tt} 
Let us suppose that we are in the setting of Proposition \ref{prop:conpg}. If $\rho_0\in L^\infty(\Om)$ and $\Phi$ satisfies \eqref{as:Phi2}, then it holds that
\begin{align}
\label{eq:1tt}
\left| \frac{1}{\rho_k^{m-2}} \left( (l-1) \rho_k^{l-2} p_{k} + \Sb''(\rho_{k}) \right)  \right| \rho_{k}^{\frac12} |\nabla( \rho_{k}^{m - 1} )| \geq \s_4  |\nabla(\rho_k^{m})|,
\end{align}
where 
\begin{align*}
\s_4:=\frac{m-1}{m}\min\left\{\frac{1}{\s_1},\frac{1}{\s_2}\right\} \min \left\{ \left( \|\rho_0\|_{L^\infty}e^{dT\|\Delta\Phi\|_{L^\infty}} \right)^{-\frac12} , \left( \|\rho_0\|_{L^\infty}e^{dT\|\Delta\Phi\|_{L^\infty}} \right)^{r-m - \frac12}\right\}. 
\end{align*}
\end{lemma}

\begin{proof}

Recall from Lemma~\ref{lem:Linfty} that if $\rho_0\in L^\infty(\Om)$, then we have 
\begin{align}
\label{eq:tt01}
\|\rho_k\|_{L^\infty}\le \|\rho_0\|_{L^\infty}e^{dT\|\Delta\Phi\|_{L^\infty}}=:C.
\end{align}
On the other hand, from \eqref{eq:sigpg13} and $\nabla( \rho_{k}^{m - 1} ) = \frac{m-1}{m} \nabla( \rho_{k}^{m})$, it holds that
\begin{align}
\label{eq:tt11}
\I_3 := \frac{1}{\rho_k^{m-2}} \left( (l-1) \rho_k^{l-2} p_{k} + \Sb''(\rho_{k})  \right) \rho_{k}^{\frac12} \nabla( \rho_{k}^{m - 1} ) = \frac{m-1}{m} \frac{S''(\rho_k)}{\rho_k^{m-\frac32}}  \nabla( \rho_{k}^{m}) \ \ \ae \hbox{ in } \{ \rho_k \neq 1\}.
\end{align}
Then, \eqref{eq:sigpg31} and \eqref{eq:tt01} yield that
\begin{align}
\label{eq:tt12}
|\I_3| &\geq \frac{m-1}{ m \s_2} \rho_k^{-\frac12} |  \nabla( \rho_{k}^{m}) | \geq \frac{m-1}{ m \s_2} C^{-\frac12} |  \nabla( \rho_{k}^{m}) | 
\ \ \ae \hbox{ in } \{ 0<\rho_k < 1\}.
\end{align}
Furthermore, \eqref{eq:1as22} and \eqref{eq:tt01} imply that
\begin{align}
\label{eq:tt13}
|\I_3| \geq \frac{m-1}{ m \s_1} \rho_k^{r-m-\frac12} |  \nabla( \rho_{k}^{m}) | \geq \frac{m-1}{ m \s_1}C^{-\frac12} \min\{  C^{r-m}, 1 \} |  \nabla( \rho_{k}^{m}) | \ \ \ae \hbox{ in } \{ \rho_k > 1\}.
\end{align}
Lastly, as \eqref{eq:sigpg41} holds, \eqref{eq:1tt} follows from \eqref{eq:tt12} and \eqref{eq:tt13}.
\end{proof}

\begin{corollary}
\label{prop:conm}
Let $(\rho^\t)_{\t>0}$ and $(p^\t)_{\t>0}$ be as in the previous proposition and \eqref{eq:mr} hold. There exists $\rho\in L^{\beta}(Q)$ and $p\in L^2([0,T]; H^1(\Om))$ such that 
\begin{align*}
\rho^\t \to \rho\ {\rm{in}}\ L^s(Q),\ {\rm{as}}\  \t\da 0,
\end{align*}
and
\begin{align*}
p^\t \weakly p\ {\rm{in}}\ L^2([0,T]; H^1(\Om)), \ {\rm{as}}\  \t\da 0,
\end{align*}
along a subsequence for any $s \in (0, \beta)$ and $\beta$ given in \eqref{eq:beta}.
\end{corollary}

\begin{proof}
Recall that Lemma~\ref{lem:uni} yields that $(\rho^\t)_{\t>0}$ is uniformly bounded in $L^{\beta}(Q)$. In both cases $r \geq m$ and $r < m < r + \frac{\beta}{2}$, Lemma~\ref{lem:al} and Proposition~\ref{prop:conpg} yield $(\rho^\t)_{\t>0}$ is precompact in $L^s(Q)$ for any $s \in (0, \beta)$.

\medskip

Indeed, first, we consider the case $r < m<r+\frac{\b}{2}$. We apply Proposition~\ref{prop:conpg}(2) and Lemma~\ref{lem:al}(1) to conclude that $(\rho^\t)_{\t>0}$ converges to $\rho$ in $L^{\left(m-\frac12\right)q^*}(Q)$ along a subsequence, where $q^*:= \frac{qd}{d-q}$ and $q\in (1,2)$ is given in Proposition~\ref{prop:conpg}(2). Note that a direct computation shows that
\begin{align*}
q^* = \frac{2r-1}{2m-1} \frac{2d}{d-2}=\frac{\b}{m-1/2}.
\end{align*}
By a similar argument, we conclude the strong convergence of $(\rho^\t)_{\t>0}$ in $L^{s}(Q)$ along a subsequence, also in the case when $r\ge m$.
\end{proof}

\begin{proof}[Proof of Theorem~\ref{thm:gmpg}]
Note that by the direct computation as in \eqref{eq:exip23} and \eqref{eq:gm11}, we have
\begin{align*}
- \E^\t &= - \rho^\t \v^\t = \rho^\t \nabla((\rho^\t)^{l-1} p^\t + \Sb'(\rho^\t))+\rho^\t\nabla\Phi\\ 
&= \nabla \left(  \frac{1}{l}((l-1)(\rho^\t)^l +1) p^\t  + \rhot \Sb'(\rhot) - \Sb(\rhot) + \Sb(1) \right)+\rho^\t\nabla\Phi. 
\end{align*}
Then, we have $- \E^\t = \nabla L_S(\rho^\t,p^\t)+\rho^\t\nabla\Phi$ for $L_S$ given in \eqref{eq:ls}. Since $l, r < \beta$ from \eqref{eq:mr}, Corollary~\ref{prop:conm} yields that $(\rho^\t)^l$, $\rhot \Sb'(\rhot)$ and $\Sb(\rhot)$ converge in $L^1(Q)$ as $\t\da 0$. As $p^\t$ is uniformly bounded, we conclude that
\begin{align*}
- \E^\t \to \nabla \left(  \frac{1}{l}((l-1)\rho^l +1) p^\t  + \rho \Sb'(\rho) - \Sb(\rho) + \Sb(1) \right) +\rho\nabla\Phi, \hbox{ as } \t \to 0,
\end{align*}
along a subsequence in $\sD'(Q;\R^d)$. Note that we have $\rho \in L^{\beta}$ from the uniform boundedness in Lemma~\ref{lem:uni}  and $p\in L^2([0,T]; H^1(\Om))\cap L^\infty(Q)$ from Proposition~\ref{prop:conpg}. As 
\begin{align}
\label{eq:gmpg23}
L_S(\rho, p) = \frac{1}{l}((l-1)\rho^l +1) p^\t  + \rho \Sb'(\rho) - \Sb(\rho) + \Sb(1),
\end{align}
for $L_S$ given in \eqref{eq:ls}, we conclude that $(\rho, p)$ satisfies \eqref{eq:1gmpg}. The rest of argument is parallel to Theorem~\ref{thm:gm}.
\end{proof}

In particular, \eqref{eq:exip} can be also represented in the form of a continuity equation, as we show below.

\begin{theorem}
\label{cor:gmpgg}
Suppose that \eqref{eq:g2} and \eqref{eq:n} hold true. Let $\rho_0$ and $(\rho,p)$ be given in Theorem \ref{thm:gmpg}. If 
\begin{align}
\label{eq:mr11}
m < r+\frac{1}{2} 
\end{align}
and
\begin{align}
\label{eq:mr2}
\beta > 2 \hbox{ and } m < \frac{\beta}{2} + \frac12 
\end{align}
hold, then $(\rho,p)$ is a weak solution of
\be\label{eq:gmpgg}
\left\{
\ba{ll}
\partial_t\rho - \nabla \cdot \left(\rho \nabla \left( \cS'(\rho)\one_{ \{ \rho \neq 1 \} } + p\one_{\{ \rho=1 \}} \right) \right) -\nabla\cdot(\rho\nabla\Phi) =0, & {\rm{in}}\ (0,T)\times\Om,\\
\rho(0,\cdot)=\rho_0, & {\rm{in}}\ \Om,\\
\rho\left[ \nabla \left( \cS'(\rho)\one_{ \{ \rho \neq 1 \} } + p\one_{\{ \rho=1 \}} \right) +\nabla\Phi\right]\cdot \n = 0 , & {\rm{in}}\ [0,T] \times \partial \Om,
\ea
\right.
\ee
in the sense of distribution. If in addition $\rho_0\in L^\infty(\Om)$ and $\Phi$ satisfies \eqref{as:Phi2}, we can drop \eqref{eq:mr2} from the statement. 
\end{theorem}

Here, note that if $d=1$ or $2$, \eqref{eq:mr2} holds true for any $m$.

\begin{proof}
\eqref{eq:mr11} implies \eqref{eq:mr} and thus \eqref{eq:2gmpgg} follows from Proposition~\ref{prop:conpg}. 
From \eqref{eq:mr2}, we can choose $l$ such that
\begin{align}
\label{eq:gmpgg11}
\max\left\{1, m - \frac12 \right\} < l \leq \frac{\beta}{2}.
\end{align}
From \eqref{eq:2gm2} and the construction in \eqref{eq:sapg}, we have 
\begin{align*}
\cS'(\rho)\one_{ \{ \rho \neq 1 \} } + p\one_{\{ \rho=1 \}} = p \rho^{l-1} + S_b'(\rho).
\end{align*}
We claim that
\begin{align}
\label{eq:gmpgg13}
\rho \nabla ( p \rho^{l-1} + S_b'(\rho) ) \in L^1(Q).
\end{align}
By the direct computation, we obtain
\begin{align}
\label{eq:gmpgg14}
\rho \nabla ( p \rho^{l-1} + S_b'(\rho) ) &= \nabla p \rho^l + p\rho \nabla (\rho^{l-1}) + \rho \nabla (\Sb'(\rho)),\\
&= \nabla p \rho^l + \frac{l-1}{m-\frac12} p \rho^{l-m+\frac12} \nabla (\rho^{m-\frac12}) + \frac{1}{m - \frac12} \Sb''(\rho) \rho^{\frac{5}{2} - m } \nabla (\rho^{m-\frac12}).
\end{align}
As $\rho \in L^\beta(Q)$ and $\beta \geq 2 l$, $\rho^l  \in L^2(Q)$ and thus $ \nabla p \rho^l \in L^1(Q)$. Now, let us consider the second term. From \eqref{eq:2gmpgg}, $\nabla (\rho^{m-\frac12}) \in L^q(Q)$ for $q$ given in \eqref{eq:q} and $m,r$ satisfying \eqref{eq:mr}. Recall that $p \in L^\infty(Q)$ and $\rho \in L^{\beta}(Q)$. As $l < \frac{\beta}{2}$ from \eqref{eq:gmpgg11}, we have
\begin{align*}
\frac{1}{q} + \frac{l-m+\frac12}{\beta} = \frac{m-r}{\beta } + \frac12 + \frac{ l - m + \frac12}{\beta} = \frac{l - r + \frac12}{\beta} + \frac12 \leq 1.
\end{align*}
Thus, from the H\"{o}lder inequality, we conclude that the second term in \eqref{eq:gmpgg14} is in $L^1(Q)$. Similarly, as $r - m +\frac12 >0$ from \eqref{eq:mr11}, $ \Sb''(\rho) \rho^{\frac{5}{2} - m } \in L^{\frac{\beta}{r-m+\frac12}}(Q)$ and 
\begin{align*}
\frac{1}{q} + \frac{r-m+\frac12}{\beta} = \frac{1}{2\beta} + \frac{1}{2} \leq 1,
\end{align*}
we conclude that the third term is in $L^1(Q)$ and \eqref{eq:gmpgg13}.

\medskip

Next, we claim that
\begin{align}
\label{eq:gmpgg21}
\rho \nabla (p \rho^{l-1} + S_b'(\rho) ) = \nabla L_S(\rho,p)
\end{align}
for $L_S$ given in \eqref{eq:ls}. 
As $p \in L^2([0,T]; H^1(\Om))$ and $(\rho,p)$ satisfies \eqref{eq:relp} from Theorem~\ref{thm:exip}, we have
\begin{align}
\label{eq:modp10}
\nabla p  = 0 \ \ \ae \hbox{ in } \{ \rho \neq 1 \}.
\end{align}
As $\rho^l \in L^2(Q)$ from Theorem~\ref{thm:gmpg} and \eqref{eq:gmpgg11}, we have
\begin{align}
\label{eq:modp11}
(\rho^l -1) \nabla p  = 0 \ \ \ae
\end{align}
From the direct computation using \eqref{eq:modp11}, it holds that
\begin{align*}
\rho \nabla \left( \rho^{l-1} p \right) &= (\rho^l)^{\frac{1}{l}}  \nabla \left( (\rho^l)^{\frac{l-1}{l}} p \right) = \rho^l \nabla p + \frac{l-1}{l} p \nabla (\rho^l)\\
&= \frac{(l-1)\rho^l + 1}{l}  \nabla p + \frac{l-1}{l} p \nabla (\rho^l)\\
&= \nabla \left( ((l-1)\rho^l +1)  \frac{p}{l}\right).
\end{align*}
Therefore, we have
\begin{align}
\label{eq:modp23}
\nabla (p \rho^{l-1} + S_b'(\rho) )  = \nabla \left( ((l-1)\rho^l +1)  \frac{p}{l} +  \rho \Sb'(\rho) - \Sb(\rho) + \Sb(1) \right).
\end{align}
From \eqref{eq:gmpg23} and \eqref{eq:modp23}, we conclude \eqref{eq:gmpgg21}.

\medskip

Lastly, note that if $\rho_0\in L^\infty(\Om)$ and and $\Phi$ satisfies \eqref{as:Phi2}, then $\rho \in L^\infty(Q)$ and thus $\rho^l  \in L^2(Q)$ for any $l>0$. Furthermore, we choose $l = m+1$, then from Proposition~\ref{prop:conpg}.(3), we conclude that $p\rho \nabla (\rho^{m}) \in L^1(Q)$. Therefore, we show \eqref{eq:gmpgg13} without \eqref{eq:mr2}.

\end{proof}


\section{Uniqueness via an $L^1$-contraction}\label{sec:uniqueness}
We construct an $L^1$ contraction result, inspired by \cite[Section 3]{DiMMes} and  \cite[Theorem 6.5]{Vaz07}. In particular, this will imply the uniqueness of the solution of \eqref{eq:1gm}-\eqref{eq:2gm} and \eqref{eq:1gmpg}-\eqref{eq:2gm2}. Let us underline the fact that because of the generality of the previous two problems, on the one hand, the techniques from \cite[Section 3]{DiMMes} do not apply directly. On the other hand, because of the presence of the critical regimes $\{\rho^i=1\},$ $i=1,2$, the construction from \cite[Theorem 6.5]{Vaz07} does not apply directly either. Therefore, we develop a careful combination of these two approaches to be able to provide an $L^1$-contraction for all the systems considered previously, with general initial data. 

\begin{theorem}\label{thm:L1contr}
Let $(\rho^1,p^1),(\rho^2,p^2)$ be solutions to {\color{blue}\eqref{eq:main1}-\eqref{eq:main2}} with initial conditions $\rho^1_0,\rho^2_0\in \sP(\Om)$ such that $\cJ(\rho_0^i)<+\infty$, $i=1,2$. Suppose that $L_S(\rho^i,p^i)\in L^2(Q)$, $i=1,2$ (or equivalently $\rho^i\in L^{2r}(Q),$ $i=1,2$). Then we have
$$\|\rho^1_t-\rho^2_t\|_{L^1(\Om)}\le \|\rho^1_0-\rho^2_0\|_{L^1(\Om)},\ \sL^1-\ae\ t\in[0,T].$$
\end{theorem}

\begin{remark}
It worth noticing that the assumption $L_S(\rho^i,p^i)\in L^2(Q)$, $i=1,2$ in the statement of the previous theorem seems quite natural in the setting of $L^1$-type contractions for porous medium equations (see \cite{Vaz07}). In our setting, because of the $L^\b(Q)$ estimates on $\rho^i$ (where $\b$ is defined in \eqref{eq:beta}) and because of the $L^r$-type growth condition on $L_S$ at $+\infty$, this assumption is fulfilled already if $\b\ge 2r$. In the same time, no such assumption is needed if the initial data is in $L^\infty(\Om)$, since in that case $L^\infty$ estimates hold true for $\rho^i_t$ for a.e. $t\in[0,T]$ (see Lemma \ref{lem:Linfty}).
\end{remark}

\begin{proof}[Proof of Theorem \ref{thm:L1contr}]

Let $(\rho^1,p^1)$ and $(\rho^2,p^2)$ be two solutions to \eqref{eq:main1}-\eqref{eq:main2} with initial data $\rho^1_0$ and $\rho^2_0$ respectively. Let $\vphi\in C^2_c((0,T]\times\Om)$ and using the notation 
$$\cI(\vphi,t):=\int_\Om\vphi_t\left(\rho^1_t-\rho^2_t\right)\dd x$$
we compute 
\begin{align*}
\frac{\dd}{\dd t}\cI(\vphi,t)=\int_{\Om}\partial_t\vphi(\rho^1-\rho^2)+\vphi\partial_t(\rho^1-\rho^2)\dd x.
\end{align*}
Now, using the equation \eqref{eq:1gm} and by integrating the above expression on  $(0,t)$, we get
\begin{align}\label{eq:deriv}
\cI(\vphi,t)&=\cI(\vphi,0)+\int_0^t\int_{\Om}\partial_s\vphi(\rho^1-\rho^2)+\Delta\vphi(L_\cS(\rho^1,p^1) - L_\cS(\rho^2,p^2))-\nabla\vphi\cdot\nabla\Phi(\rho^1-\rho^2)\dd x\dd s\\
\nonumber&=\cI(\vphi,0)+\int_0^t\int_{\Om}(L_\cS(\rho^1,p^1) - L_\cS(\rho^2,p^2))\left[A\partial_s\vphi+\Delta\vphi-A\nabla\Phi\cdot\nabla\vphi\right]\dd x\dd s,
\end{align}
where we use the notation 
\begin{align}
\label{eq:A}
A:=\frac{\rho^1-\rho^2}{L_\cS(\rho^1,p^1) - L_\cS(\rho^2,p^2)}, 
\end{align}
with the convention $A=0$, when $L_\cS(\rho^1,p^1) = L_\cS(\rho^2,p^2)$. Note that Lemma~\ref{lem:mon} below implies that if $L_\cS(\rho^1,p^1) = L_\cS(\rho^2,p^2)$ a.e., then $\rho^1 = \rho^2$ and $p^1 =p^2$ a.e. Furthermore, on this very particular set actually there is no contribution in the integral  on the right hand side of \eqref{eq:deriv}, so it is meaningful to set $A=0$ there.  Also, because of the monotonicity property of the operator $L_S$ (see Lemma \ref{lem:mon}), we have that $A\ge 0$ a.e. in $Q$.

Similarly to the arguments from \cite[Section 3]{DiMMes}, for $\zeta:\Om\to\R$ smooth with $|\zeta|\le 1$, we consider the dual backward equation as
\begin{equation}\label{eq:dual}
\left\{
\begin{array}{ll}
A\partial_t\vphi+\Delta\vphi - A\nabla\Phi\cdot\nabla\vphi=0, & {\rm{in}}\ (0,T)\times\Om,\\
\nabla\vphi\cdot\n=0, & {\rm{on}}\ (0,T)\times\partial\Om,\\
\vphi(T,\cdot)=\zeta, & {\rm{in}}\ \Om.
\end{array}
\right.
\end{equation}
Let us notice that if we are able to construct a suitable (weak) solution $\vphi$ to \eqref{eq:dual}, for which the  computations in \eqref{eq:deriv} remain valid, we can deduce the $L^1$-contraction result, after optimizing with respect to $\zeta$. In general one cannot hope for smoothness of $A$, and so \eqref{eq:dual} is degenerate. Therefore, we introduce suitable approximations which will allow to construct smooth test function. 

\medskip

Let us define two Borel sets 
$$E_1:=\{\rho^1\ge 1/2\}\cup\{\rho^2\ge 1/2\}$$ and $E_2:=Q\setminus E_1$. We suppose that both sets $E_1$ and $E_2$ have positive measures w.r.t. $\sL^{d+1}$, otherwise we  simply do not consider the negligible one in the consideration below. First, by Lemma \ref{lem:bdd}, we have that $A\mres E_1$ is bounded. Second we have the following 

\medskip

{\it Claim.} $A^{-1}\mres E_2\in L^{2}(E_2)$.

\medskip

{\it Proof of the claim.} Let us notice that we can write 
\begin{align*}
E_2&=\left(\{\rho^1<1/2\}\cap\{\rho^2\ge 1/2\}\right)\cup\left(\{\rho^1\ge1/2\}\cap\{\rho^2< 1/2\}\right)\cup\left(\{\rho^1<1/2\}\cap\{\rho^2< 1/2\}\right)\\
&:=E_2^1\cup E_2^2\cup E_2^3.
\end{align*}
 We further decompose $E_2^1:=\left(\{\rho^1<1/2\}\cap\{1/2\le \rho^2 <1\}\right)\cup\left(\{\rho^1<1/2\}\cap\{\rho^2\ge 1\}\right)=:E_1^{11}\cup E_1^{12}$.
For a.e. $q\in E_1^{11}$ we have  
\begin{align*}
A^{-1}(q)=\frac{L_S(\rho^1(q),p^1(q))-L_S(\rho^2(q),p^2(q))}{\rho^1(q)-\rho^2(q)}=\tilde \rho(q)S''(\tilde\rho(q))
\end{align*}
where $\tilde \rho(q)$ is between $\rho^1(q)$ and $\rho^2(q)$. Since restricted to $E_1^{11}$ both $\rho^1$ and $\rho^2$ are bounded by $1$, we have that $A^{-1}\mres E_1^{11}\in L^\infty(E_1^{11}).$

For a.e. $q\in E_1^{12}$ we have
\begin{align*}
A^{-1}(q)=\frac{L_S(\rho^1(q),p^1(q))-L_S(\rho^2(q),p^2(q))}{\rho^1(q)-\rho^2(q)}\le 2|L_S(\rho^1(q),p^1(q))-L_S(\rho^2(q),p^2(q))|,
\end{align*}
since restricted to this set $|\rho^1(q)-\rho^2(q)|\ge 1/2$ a.e. Therefore, by our assumption on $L_S(\rho^i,p^i)$ we have that $A^{-1}\mres E_{2}^{12}\in L^2(E_2^{12}).$ Therefore, $A^{-1}\mres E_1^1\in L^2(E_2^1)$.

Similarly, we can draw the same conclusion in the case of $E_2^2$, and so $A^{-1}\mres E_2^2\in L^2(E_2^2).$

\medskip

For a.e. $q\in E_2^3$, we conclude similarly as in the case of $E_2^{11}$, i.e. we have that 
\begin{align*}
A^{-1}(q)=\frac{L_S(\rho^1(q),p^1(q))-L_S(\rho^2(q),p^2(q))}{\rho^1(q)-\rho^2(q)}
=\tilde \rho(q)S''(\tilde\rho(q)),
\end{align*}
where $\tilde \rho(q)$ is between $\rho^1(q)$ and $\rho^2(q)$. Since restricted to $E_2^3$ both $\rho^1$ and $\rho^2$ are bounded by $1/2$, we have that $A^{-1}\mres E_2^3\in L^\infty(E_2^3).$

Therefore, combining all the previous arguments, one obtains that $A^{-1}\mres E_2\in L^2(E_2)$, and the claim follows.

\medskip

\medskip

Let $\e>0$ and let $K_1:=\|A \one_{E_1}\|_{L^\infty(Q)}$. Let $A_1^\e:= \max\{ \e, A \one_{E_1}\}$. Then, we have $\e \leq A_1^\e \leq K_1$ and $\|A_1^\e-A\one_{E_1}\|_{L^\infty(Q)} \leq \e$.
In the same time, for $0<\d\le K$ given, let $A_2^\e = A_2^\e(\delta, K)$ be smooth such that $\delta \leq (A_2^\e)^{-1} \le K$ and 
\begin{equation}\label{eq:regul}
(A_2^\e)^{-1}\to [(A\one_{E_2})^{-1}]_{\delta, K}\ {\rm{strongly\ in\ }}L^q(E_2),\ \ {\rm{as}}\ \e\da 0, 
\end{equation}
for any $q \in[1,+\infty)$ and in particular, $A_\e^{-1}\weaklys[ (A\one_{E_2})^{-1}]_{\delta, K}$ weakly-$\star$ in $L^\infty(E_2)$ as $\e\da 0$. Here, for a nonnegative function $f:Q\to[0,+\infty)$ we use the notation $f_{\d,K}:=\min\{\max\{f,\d\},K\}.$   

Now, let us define $A_\e:Q\to [0,+\infty)$ as 
$$
A_\e:=\left\{
\begin{array}{ll}
A_1^\e, & {\rm{a.e\ in}}\ E_1,\\[3pt]
A_2^\e, & {\rm{a.e.\ in}}\ E_2.
\end{array}
\right.
$$
By construction $\min\{\e;1/K\}\le A_\e\le\max\{K_1,1/\d\}.$ For $\theta>0$ let $A_\theta$ (which depends also on $\e,\d$ and $K$) be a smooth approximation of $A_\e$ such that 
\begin{align}\label{eq:lb-theta}
&\min\{\e;1/K\}\le A_\theta\le\max\{K_1,1/\d\},\ {\rm{in}}\ Q;\\
&\nonumber \e\le A_\theta\le K_1,\ {\rm{a.e.\ in}}\ E_1;\\
&\nonumber 1/K\le A_\theta\le 1/\d, \ {\rm{a.e.\ in}}\ E_2;
\end{align} 
and $A_\theta\to A_\e$ strongly in $L^q(Q)$ for any $q\in[1,+\infty)$ and in particular 
\begin{align}\label{eq:conv-theta1}
A_\theta\weaklys A_\e\  {\rm{weakly}}-\star\ {\rm{in}}\ L^\infty(Q),\ {\rm{as}\ } \theta\da 0. 
\end{align}

Moreover, we have
\begin{equation}\label{eq:conv-theta2}
A_\theta^{-1}\to [(A\one_{E_2})^{-1}]_{\d, K} \ {\rm{in}}\ L^q(E_2), \ \forall\ q\in[1,+\infty)\ \ {\rm{and}}\ \ A_\theta^{-1}\weaklys [(A\one_{E_2})^{-1}]_{\d, K} \ {\rm{in}}\ L^\infty(E_2),\ {\rm{as\ }} \max\{\theta,\e\}\da 0.
\end{equation}
To check this last claim, we argue as follows.
\begin{align*}
\|A_\theta^{-1}- [(A\one_{E_2})^{-1}]_{\d, K}\|_{L^q(E_2)}&\le \|A_\theta^{-1}-(A_2^\e)^{-1}\|_{L^q(E_2)}+\|(A_2^\e)^{-1}- [(A\one_{E_2})^{-1}]_{\d, K}\|_{L^q(E_2)}\\
&=\|(A_\theta -A_2^\e)/(A_\theta A_2^\e)\|_{L^q(E_2)}+\|(A_2^\e)^{-1}- [(A\one_{E_2})^{-1}]_{\d, K}\|_{L^q(E_2)}\\
&\le K^2\|A_\theta -A_2^\e\|_{L^q(E_2)}+\|(A_2^\e)^{-1}- [(A\one_{E_2})^{-1}]_{\d, K}\|_{L^q(E_2)}\to 0,
\end{align*}
as $\max\{\theta,\e\}\da 0$, by the construction of $A_\theta$ and $A_2^\e$. We conclude similarly about the weak-$\star$ convergence as well.

Since $\Phi\in W^{1,\infty}(\Om)$, we consider a smooth approximation of it, $(\Phi_\theta)_{\theta>0}$ such that $\nabla\Phi_\theta\to\nabla\Phi$, as $\theta\da 0$, strongly in $L^{2r'}(\Om).$

Let us consider the regularized dual equation which reads as 
\begin{equation}\label{eq:dual-reg}
\left\{
\begin{array}{ll}
\partial_t\vphi_\theta+(1/A_\theta)\Delta\vphi_\theta-\nabla\Phi_\theta\cdot\nabla\vphi_\theta=0, & {\rm{in}}\ (0,T)\times\Om,\\
\nabla\vphi_\theta\cdot\n=0, & {\rm{on}}\ (0,T)\times\partial\Om,\\
\vphi_\theta(T,\cdot)=\zeta, & {\rm{in}}\ \Om.
\end{array}
\right.
\end{equation} 

Let $\vphi_\theta$ be the smooth solution of \eqref{eq:dual-reg}, when the coefficient function is $A_\theta$ and we use this in \eqref{eq:deriv} as
\begin{align*}
\cI(\vphi_\theta,T)-\cI(\vphi_\theta,0)&=\int_0^T\int_{\Om}\partial_s\vphi_\theta(\rho^1-\rho^2)+\Delta\vphi_\theta(L_\cS(\rho^1,p^1) - L_\cS(\rho^2,p^2))-\nabla\vphi_\theta\cdot\nabla\Phi(\rho^1-\rho^2)\dd x\dd s\\
&=\int_{E_1}\partial_s\vphi_\theta(\rho^1-\rho^2)+\Delta\vphi_\theta(L_\cS(\rho^1,p^1) - L_\cS(\rho^2,p^2))-\nabla\vphi_\theta\cdot\nabla\Phi(\rho^1-\rho^2)\dd \sL^{d+1}\\
&+\int_{E_2}\partial_s\vphi_\theta(\rho^1-\rho^2)+\Delta\vphi_\theta(L_\cS(\rho^1,p^1) - L_\cS(\rho^2,p^2))-\nabla\vphi_\theta\cdot\nabla\Phi(\rho^1-\rho^2)\dd \sL^{d+1}\\
&=\int_{E_1}(L_\cS(\rho^1,p^1) - L_\cS(\rho^2,p^2))\left[A\partial_s\vphi_\theta+\Delta\vphi_\theta-A\nabla\Phi\cdot\nabla\vphi_\theta\right]\dd \sL^{d+1}\\
&+\int_{E_2}(\rho^1-\rho^2)\left[\partial_s\vphi_\theta+A^{-1}\Delta\vphi_\theta-\nabla\Phi\cdot\nabla\vphi_\theta\right]\dd\sL^{d+1}=:\cI_1+\cI_2.
\end{align*}
It remains to show that both $|\cI_1|$ and $|\cI_2|$ can be made arbitrary small. Because $\phi_\theta$ solves \eqref{eq:dual-reg} with the coefficient function $A_\theta$, we have
\begin{align*}
\cI_1&=\int_{E_1}(L_\cS(\rho^1,p^1) - L_\cS(\rho^2,p^2))\left[A\partial_s\vphi_\theta+\Delta\vphi_\theta-A\nabla\Phi\cdot\nabla\vphi_\theta\right]\dd \sL^{d+1}\\
&-\int_{E_1}(L_\cS(\rho^1,p^1) - L_\cS(\rho^2,p^2))A\left[\partial_s\vphi_\theta+A_\theta^{-1}\Delta\vphi_\theta-\nabla\Phi_\theta\cdot\nabla\vphi_\theta\right]\dd \sL^{d+1}\\
&=\int_{E_1}(L_\cS(\rho^1,p^1) - L_\cS(\rho^2,p^2))(A_\theta-A)A_\theta^{-\frac12} A_\theta^{-\frac12}\Delta\vphi_\theta\dd \sL^{d+1}\\
&+\int_{E_1}(L_\cS(\rho^1,p^1) - L_\cS(\rho^2,p^2))A\nabla\vphi_\theta\cdot(\nabla\Phi_\theta-\nabla\Phi)\dd \sL^{d+1}.
\end{align*}
From here, by \eqref{eq:lb-theta} we have 
\begin{align*}
|\cI_1&|\le \e^{-\frac12}\|A_\theta^{-\frac12}\Delta\vphi_\theta\|_{L^2(Q)}\left(\int_{E_1}|L_\cS(\rho^1,p^1) - L_\cS(\rho^2,p^2)|^2|A_\theta-A|^2\dd\sL^{d+1}\right)^{\frac12}\\
&+\int_0^T\int_\Om|\rho^1-\rho^2||\nabla\vphi_\theta||\nabla\Phi_\theta-\nabla\Phi|\dd x\dd t.
\end{align*}
By Lemma \ref{lem:estim_reg}(1), the summability assumption on $\rho^i\in L^{2r}(Q)$ and the approximation $\nabla\Phi_\theta\to\nabla\Phi$, in $L^{2r'}(\Om)$ as $\theta\da 0$, we conclude that the second term in the previous inequality tends to 0 as $\theta\da 0$. By Lemma \ref{lem:estim_reg}(2), we have that  $\|A_\theta^{-\frac12}\Delta\vphi_\theta\|_{L^2(Q)}\le C$ for some constant independent of $\theta$ and $\e$. Furthermore, by \eqref{eq:conv-theta1}, by the summability assumption on $L_S(\rho^i,p^i)$ and by the construction of $A_1^\e$, for $\theta$ small enough we have
\begin{align*}
&\int_{E_1}|L_\cS(\rho^1,p^1)- L_\cS(\rho^2,p^2)|^2|A_\theta-A|^2\dd\sL^{d+1}\\
&\le 2 \int_{E_1}|L_\cS(\rho^1,p^1) - L_\cS(\rho^2,p^2)|^2|A_\theta-A_1^\e|^2\dd\sL^{d+1}+2\int_{E_1}|L_\cS(\rho^1,p^1) - L_\cS(\rho^2,p^2)|^2|A_1^\e-A|^2\dd\sL^{d+1}\\
&\le \e^2 + C\e^2, 
\end{align*}
for some constant independent of $\e,\theta, K$ and therefore by the arbitrariness of $\e$, we conclude that $\cI_1=0.$  

\medskip

In the case of $\cI_2$ we argue as follows.
\begin{align*}
\cI_2&=\int_{E_2}(\rho^1-\rho^2)\left[\partial_s\vphi_\theta+A^{-1}\Delta\vphi_\theta-\nabla\Phi\cdot\nabla\vphi_\theta\right]\dd \sL^{d+1}\\
&-\int_{E_2}(\rho^1-\rho^2)\left[\partial_s\vphi_\theta+A_\theta^{-1}\Delta\vphi_\theta-\nabla\Phi_\theta\cdot\nabla\vphi_\theta\right]\dd\sL^{d+1}\\
&=\int_{E_2}(\rho^1-\rho^2)(A^{-1}-A_\theta^{-1})A_\theta^{\frac12} A_\theta^{-\frac12}\Delta\phi_\theta\dd \sL^{d+1}\\
&+\int_{E_2}(\rho^1-\rho^2)\nabla\vphi_\theta\cdot(\nabla\Phi_\theta-\nabla\Phi)\dd \sL^{d+1}\\
&=\int_{E_2}(\rho^1-\rho^2)(A^{-1}-A^{-1}_{\d,K})A_\theta^{\frac12} A_\theta^{-\frac12}\Delta\phi_\theta\dd \sL^{d+1}+\int_{E_2}(\rho^1-\rho^2)(A^{-1}_{\d,K}-A_\theta^{-1})A_\theta^{\frac12} A_\theta^{-\frac12}\Delta\phi_\theta\dd \sL^{d+1}\\
&+\int_{E_2}(\rho^1-\rho^2)\nabla\vphi_\theta\cdot(\nabla\Phi_\theta-\nabla\Phi)\dd \sL^{d+1}\\
&=:\cI_{21}+\cI_{22}+\cI_{23}.
\end{align*}
In the case of $\cI_{23}$, we argue exactly as in the case of the second term of $\cI_1$ to conclude that this term tends to 0 as $\theta\da 0$. As for the other terms,
let us notice that by the definition of  $A_{\delta, K}^{-1}$ (on $E_2$), we have that
\begin{align*}
\left| A^{-1}- A_{\delta, K}^{-1} \right| = 
\begin{cases}
\d &\hbox{ a.e. in } \{0 \leq A^{-1} < \delta\}\cap E_2,\\
0 &\hbox{ a.e. in } \{\delta \leq A^{-1} \leq K\}\cap E_2,\\
A^{-1} - K &\hbox{ a.e. in }  \{K \leq A^{-1}\}\cap E_2,
\end{cases}
\end{align*}
and thus
\begin{align}
\label{eq:uni41}
\left| A^{-1}- A_{\delta, K}^{-1} \right| \leq  \delta +  (A^{-1}-K)_{+}, \ {\rm{a.e.\ in\ }}E_2.
\end{align}
Therefore, since $A_\theta^{\frac12}\le\d^{-\frac12}$, we obtain
\begin{align*}
|\cI_{21}|\le \|A_\theta^{-\frac12}\Delta\phi_\theta\|_{L^2(Q)} \d^{-\frac12}\left(\d\|\rho^1-\rho^2\|_{L^2(E_2)}+\|(\rho^1-\rho^2)(A^{-1}-K)\|_{L^2(\{K \leq A^{-1}\}\cap E_2)}\right)\to 0,
\end{align*}
as $K\to+\infty$ and $\d\da 0$ (in this order). This is true indeed, by Lemma \ref{lem:estim_reg}(2) and by the fact that 
\begin{align*}
\int_{\{K \leq A^{-1}\}\cap E_2}(\rho^1-\rho^2)^2(A^{-1}-K)^2\dd\sL^{d+1}&\le\int_{\{K \leq A^{-1}\}\cap E_2}(\rho^1-\rho^2)^2(A^{-1})^2\dd\sL^{d+1}\\
&\le \int_{\{K \leq A^{-1}\}\cap E_2}(L_S(\rho^1,p^1)-L_S(\rho^2,p^2))^2\dd\sL^{d+1}.
\end{align*}
Since $A^{-1}\in L^2(E_2)$, by Chebyshev's inequality $\sL^{d+1}(\{K \leq A^{-1}\}\cap E_2)\to 0,$ as $K\to+\infty,$ so by the summability of $L_S^2(\rho^i,p^i)$ we deduce that for $K$ large enough last term in the last inequality is smaller than $\d^2$. Therefore, by the arbitrariness of $\d$, we conclude that $\cI_{21}$ has to be zero.

\medskip

To show that $|\cI_{22}|$ can be made arbitrary small, using again $A_\theta^{\frac12}\le \d^{-\frac12}$ a.e. on $E_2$ and  Lemma \ref{lem:estim_reg}(2),  we have 
\begin{align*}
|\cI_{22}|^2\le \d^{-1}C\int_{E_2}(\rho^1-\rho^2)^2(A^{-1}_{\d,K}-A_\theta^{-1})^2\dd\sL^{d+1}.
\end{align*}
By the fact that $A^{-1}_{\d,K},A_\theta^{-1}\in L^\infty(E_2)$, $\rho^1,\rho^2\in L^2(E_2)$ and by the weak-$\star$ convergence of $A_\theta^{-1}$ to $A^{-1}_{\d,K}$ in $L^\infty(E_2)$, we conclude that for $\theta$ small enough, the r.h.s. of the previous inequality is smaller than $\d$, therefore by the arbitrariness of $\d$ we conclude that $\cI_{22}=0.$ 
\end{proof}

The next three lemmas (in an implicite or explicite form) are used in the proof of Theorem \ref{thm:L1contr}.

\begin{lemma}
\label{lem:mon}
Let $(\rho^1,p^1),(\rho^2,p^2):\Om\to\R^2$ satisfy \eqref{eq:main2}. Then $L_S$ (defined in \eqref{eq:ls}) defines a monotone operator in the sense that
\begin{align}
\label{eq:mon11}
{\rm{if\ }} \rho^1(x) < \rho^2(x), {\rm{\ then\ }} L_\cS(\rho^1,p^1)(x) < L_\cS(\rho^2,p^2)(x).
\end{align}
In particular, for $x \in \Om$, if 
\begin{align}
\label{eq:1mon}
L_\cS(\rho^1,p^1)(x) = L_\cS(\rho^2,p^2)(x),
\end{align}
then $\rho^1(x) = \rho^2(x)$ and $p^1(x) = p^2(x)$.
\end{lemma}
\begin{proof}

First of all, if we have \eqref{eq:1mon} and $\rho^1(x) = \rho^2(x)$, then \eqref{eq:ls} and \eqref{eq:main2} imply $p^1(x) = p^2(x)$. Thus, it is enough to show that $\rho^1(x) = \rho^2(x)$. 

Let us show now that $L_S$ is a monotone operator in the sense of
\eqref{eq:mon11}. First, note that $\rho \mapsto \rho \cS'(\rho) - \cS(\rho)$ is strictly increasing in $\Ro$. Indeed, from the strict convexity of $\cS$ (Assumption~\ref{as:gen}), one has that
\begin{align}
\label{eq:mon12}
\partial_\rho (\rho \cS'(\rho) - \cS(\rho)) = \rho \cS''(\rho) > 0 \hbox{ in } \Ro.
\end{align} 
Therefore, \eqref{eq:mon11} holds if $\rho^1(x), \rho^2(x) \in (0,1)$ or $\rho^1(x), \rho^2(x) \in (1,+\infty)$. Thus, it remains to treat the remaining cases.

\medskip

Consider the case that $\rho^1(x) = 1 < \rho^2(x)$. Let us recall that $\cS$ and $\cS'$ are continuous in $\Rp$ and $\Rp \setminus \{1\}$, respectively (by Assumption~\ref{as:gen}). As $\rho \mapsto \rho \cS'(\rho) - \cS(\rho)$ is strictly increasing in $(1,+\infty)$, we have that
\begin{align}
\label{eq:mon21}
L_\cS(\rho^2,p^2) &= \rho^2(x) \cS'(\rho^2(x)) - \cS(\rho^2(x)) + \cS(1) > \lim_{\rho \to 1+} \rho \cS'(\rho) - \cS(\rho) + \cS(1) = \cS'(1+)\\
&\nonumber\ge p^1(x)=L_S(\rho^1,p^1)(x).
\end{align}
So, from \eqref{eq:mon21} and \eqref{eq:main2}, we conclude \eqref{eq:mon11}.  

Similar arguments show \eqref{eq:mon11} in the case when $\rho^1(x) < \rho^2(x) = 1$. 

Lastly, by combining the inequalities in \eqref{eq:mon11} for two previous cases, $\rho^1(x) = 1 < \rho^2(x)$ or $\rho^1(x) < 1 = \rho^2(x)$, we conclude \eqref{eq:mon11} for $\rho^1(x) < 1 < \rho^2(x)$.
\end{proof}

\begin{lemma}\label{lem:bdd}
We differentiate two cases.
\begin{itemize}
\item[(1)]Assume $m=1$ for $m$ given in \eqref{eq:g1}. Let $(\rho^1,p^1)$ and $(\rho^2,p^2)$ satisfy \eqref{eq:main2}. Then we have
\begin{align}
\label{eq:1bdd}
0 \le A \le\max\left\{\s_1,\s_2\right\}, \ \ae\ {\rm{in}}\ Q,
\end{align}
where $A=A(\rho^1,p^1, \rho^2,p^2)$ is given in \eqref{eq:A} and $\s_1,\s_2$  are from Assumption \eqref{eq:g1}-\eqref{eq:g2}.
\item[(2)] Let $m>1$. If there exist $c_0>0$ and a Borel set $E\subseteq Q$  such that $\rho^1,\rho^2\ge c_0$ a.e. on $E$, then $A\mres E\in L^\infty(E)$ and $\ds A\le \max\left\{\s_1,\frac{\s_2}{c_0^{m-1}}\right\}$ a.e. in $E$, where $A=A(\rho^1,p^1, \rho^2,p^2)$ is given again in \eqref{eq:A}.
\end{itemize}
\end{lemma}

\begin{proof}
Let us recall the definition of $L_S$ from \eqref{eq:ls}, i.e.
$$L_\cS(\rho,p)(t,x):= \left[ \rho(t,x) \cS'(\rho(t,x)) - \cS(\rho(t,x)) + \cS(1) \right] \one_{ \{ \rho \neq 1 \} }(t,x) + p(t,x) \one_{\{ \rho=1 \}}(t,x).$$ The non-negativity of $A$ follows from the monotonicity of $L_\cS$ shown in Lemma~\ref{lem:mon}. We fix $q=(t,x) \in Q$ a Lebesgue for $\rho^1,\rho^2, p^1, p^2$ and assume that $\rho^1(t,x) \geq \rho^2(t,x)$. If $q\in \{\rho^1=1\}\cap\{\rho^2=1\}$ there is nothing to check, since $A(q)=0$ in both cases.

Let us show (1).

{\it Case 1.} If  $q\in \left(\{\rho^1>1\}\cap\{\rho^2>1\}\right)\cup\left(\{\rho^1<1\}\cap\{\rho^2<1\}\right)$ we have that 
\begin{align*}
\rho^1(q)\cS'(\rho^1(q)) - \cS(\rho^1(q))-\rho^2(q)\cS'(\rho^2(q)) + \cS(\rho^2(q))=\tilde\rho S''(\tilde\rho)(\rho^1(q)-\rho^2(q))\ge\min\left\{\frac{1}{\s_1},\frac{1}{\s_2}\right\}(\rho^1(q)-\rho^2(q)),
\end{align*}
where $\tilde\rho$ is a constant between $\rho^1(q)$ and $\rho^2(q)$. Therefore, we get that $A(q)\le\max\left\{\s_1,\s_2\right\}.$

\medskip

{\it Case 2.} If $q\in \{\rho^1>1\}\cap\{\rho^2=1\}$ we have  from \eqref{eq:main2} that
\begin{align*}
\rho^1(q)\cS'(\rho^1(q)) - \cS(\rho^1(q)) + S(1) - p^2(q) \geq \rho^1(q)\cS'(\rho^1(q)) - \cS(\rho^1(q)) - (\cS'(1+) -  \cS(1)).
\end{align*}
As $\rho \mapsto \rho \cS'(\rho ) - \cS(\rho )$ is continuous in $[1, \rho^1(q)]$ and differentiable in $(1, \rho^1(q))$, the mean value theorem yields that
\begin{align*}
\rho^1(q)\cS'(\rho^1(q)) - \cS(\rho^1(q)) - p^2(q) \geq \tilde\rho S''(\tilde\rho)(\rho^1(q)-1) \ge\frac{1}{\s_1}(\rho^1(q)-1),
\end{align*}
where $\tilde\rho$ is between $1$ and $\rho^1(q)$. Parallel arguments show \eqref{eq:1bdd} on the region $\{\rho^1=1\}\cap\{\rho^2<1\}.$

\medskip

{\it Case 3.} If $q\in\{\rho^1>1\}\cap\{\rho^2<1\}$ from similar arguments as in Case 2, we have that
\begin{align*}
\rho^1(q)\cS'(\rho^1(q)) - \cS(\rho^1(q)) - (\cS'(1+) -  \cS(1)) \geq \frac{1}{\s_1}(\rho^1(q)-1)
\end{align*}
and
\begin{align*}
(\cS'(1-) -  \cS(1)) - [\rho^2(q)\cS'(\rho^2(q)) - \cS(\rho^1(q))]  \geq \frac{1}{\s_2}(1 - \rho^2(q)).
\end{align*}
As $\cS'(1+) \geq \cS'(1-)$, we conclude that
\begin{align*}
L_\cS(\rho^1,p^1)(q) - L_\cS(\rho^2,p^2)(q)\ge \s_1(\rho^1(q)-1) + \s_2(1 - \rho^2(q)) = \min\left\{\frac{1}{\s_1},\frac{1}{\s_2}\right\}(\rho^1(q) - \rho^2(q)).
\end{align*}

The proof of (2) follows the very same steps as the one of (1). By the lower bound $c_0>0$ on the densities in $E$, we conclude that $\ds A\le \max\left\{\s_1,\frac{\s_2}{c_0^{m-1}}\right\}.$

\end{proof}

\begin{lemma}\label{lem:estim_reg}
Let $\e>0$ and let $\vphi_\e$ be a smooth solution to \eqref{eq:dual-reg}. Then there exists a constant $C=C(T,\|\nabla\zeta\|_{L^2})>0$ such that
\begin{itemize}
\item[(1)] $\sup_{t\in[0,T]}\|\nabla\vphi_\e\|_{L^2(\Om)}\le C$;
\item[(2)] $\|A_\e^{-\frac12}\Delta\vphi_\e\|_{L^{2}(Q)}\le C$.
\end{itemize}
\end{lemma}
\begin{proof}
The proof of this results follows the same lines as the one of \cite[Lemma 3.1]{DiMMes}, therefore we omit it.
\end{proof}

\begin{corollary}

\label{cor:uni}
Let $\rho_0 \in \sP(\Om)$ satisfy $\cJ(\rho_0)<+\infty$. A solution pair to \eqref{eq:main1}-\eqref{eq:main2} such that $L_S(\rho,p) \in L^2(Q)$ is uniquely determined by $\rho_0$.
\end{corollary}
\begin{proof}
From the contraction result in Theorem \ref{thm:L1contr} we deduce the uniqueness of $\rho$. Now suppose that there exists to pressure fields $p^1,p^2$ solving \eqref{eq:1gm} with the same $\rho$. Taking the difference of these two equations we get 
$$\Delta(L_\cS(\rho,p^1) - L_\cS(\rho,p^2))=0,\ {\rm{in}}\ \sD'((0,T)\times\Om).$$
For a.e. $t\in[0,T]$ and for any $\vphi\in C^2_c(\Om)$ we have that 
$$0=\int_\Om (L_\cS(\rho_t,p^1_t) - L_\cS(\rho_t,p^2_t)) \Delta\vphi\dd x=\int_{\{\rho_t=1\}}(p^1_t-p^2_t)\Delta\vphi\dd x,$$
where in the last equality we used the fact that $p^1_t=p^2_t$ a.e. in $\{\rho_t<1\}\cup \{\rho_t>1\}.$ By the arbitrariness of $\vphi$ we conclude that $p^1_t=p^2_t$ a.e. on $\{\rho_t=1\}$ and therefore the uniqueness of $p$ follows.

\end{proof}


\section{Discussions}\label{sec:disc}

\subsection{The emergence of the `critical region' $\{\rho=1\}$ -- an example}
We consider $d=1$ and we show that the critical region $\{\rho_t=1\}$ is of positive measure, whenever the two regions $\{\rho_t>1\}$ and $\{\rho_t<1\}$ are also of positive measure. We will see that this also implies that the critical region is expected to emerge for positive times, even if $\sL^1(\{\rho_0=1\})=0$ (and if $\sL^1(\{\rho_0<1\})>0$ and $\sL^1(\{\rho_0>1\})>0$). This phenomenon corresponds to the growth of the critical region for self-organized criticality in \cite{BanJan}.

\begin{proposition} 
\label{prop:eme}
Let $\Om \subset \R$ and $(\rho,p)$ be given in Theorem~\ref{thm:exi}. If $t\in(0,T)$ is a Lebesgue point both for $t\mapsto\rho_t$ and $t\mapsto p_t$ with 
\begin{equation}\label{ass:71}
\sL^1(\{\rho_t<1\})>0\ \ {\rm{and}}\ \  \sL^1(\{\rho_t>1\})>0
\end{equation}
then $\sL^1(\{\rho_t=1\})>0$.  
\end{proposition}

\begin{proof}
Let us show that $p(t, \cdot) \in C^{0,\frac{1}{2}}(\Om)$ for a.e. $t \in [0,T]$. 
From Theorem~\ref{thm:exi} we know that $\partial_x p \in L^2(Q)$.
As a consequence, we have that
\begin{align*}
\int_0^T \osc_{x \in [a,b]} p(t,x)\dd t \leq \int_0^T \int_a^b |\partial_x p(t,x)|\dd t \leq \| \partial_x p \|_{L^2(Q)} T^{\frac{1}{2}} |b-a|^{\frac{1}{2}}.
\end{align*}
Thus, $p \in L^1(0,T; C^{0,\frac{1}{2}}(\Om))$ and we conclude. 

\medskip

Let $t\in(0,T)$ be a Lebesgue point for both $t\mapsto p_t$ and $t\mapsto\rho_t$ such that $\sL^1(\{\rho_t<1\})>0$ and $\sL^1(\{\rho_t>1\})>0$. Then \eqref{ass:71} and \eqref{eq:rel} imply that there exists $\{U_i\}_{i \in \{1,2\}}$ subsets of $\Om$ such that $\sL^1(U_i) > 0$ and $p_t = i$ a.e. in $U_i$ for $i \in \{ 1,2\}$. As $p_t$ is continuous in $\Om$ for a.e. $t \in [0,T]$, there exists a point $x_0 \in \Om$ such that $p_t(x_0) = 3/2$. Since $N:=p_t^{-1}\left(\left(5/4, 7/4\right)\right)$ is a nonempty open set, $N$ has a positive measure. From \eqref{eq:rel}, we have that $N \subset \{\rho_t = 1 \}$ 
and thus we conclude.
\end{proof}

\begin{remark}
A similar result can be stated in higher dimensions as well, based on the fact that Sobolev functions cannot take finitely many values, except if they are constants.
\end{remark}

\subsection{Formal derivation of a free boundary problem -- an example}

Next, we formally derive the free boundary motion corresponding to the particular problem in \eqref{eq:exi}-\eqref{eq:rel}. For the analysis, we assume that $\rho$ and $p$ are continuous in $Q$ and smooth in $\{ p\rho < 1 \}$, $\{1 < p\rho < 2\}$ and $\{ p\rho > 2 \}$, which also have smooth boundaries.
Under this assumption, we deduce the following free boundary problem,
\begin{align}
\label{eq:for}
\begin{cases}
\partial_t\rho - \Delta \rho -\nabla\cdot(\nabla\Phi\rho)=0,  \quad  &\hbox{ in }  \{ p\rho < 1 \},\\
\rho =1,  \quad  &\hbox{ in }  \{1 < p\rho < 2\},\\
\partial_t\rho - 2 \Delta \rho-\nabla\cdot(\nabla\Phi\rho) =0,  \quad  &\hbox{ in }  \{ p\rho > 2 \},
\end{cases}
\hbox{ and }
\begin{cases}
p =1,  \quad  &\hbox{ in }  \{ p\rho < 1 \},\\
 - \Delta p =\Delta\Phi,  \quad  &\hbox{ in }  \{1 < p\rho < 2\},\\
p =2,  \quad  &\hbox{ in }  \{ p\rho > 2 \},
\end{cases}
\end{align}
with boundary conditions
\begin{align}
\label{eq:2for}
\begin{cases}
| D(p\rho)^{1+} | - | D(p\rho)^{1-} | &=0,  \quad  \hbox{ on }  \{ p\rho = 1 \},\\
| D(p\rho)^{2+} | - | D(p\rho)^{2-} | &=0,  \quad  \hbox{ on }  \{ p\rho = 2 \},
\end{cases}
\end{align}
where for any $f:Q \to \R$ and $c\in \R$, denotes
\begin{align*}
Df^{c\pm}(t,x):= \lim\limits_{\substack{(s,y) \to (t,x), \\ (s,y) \in \{\pm(f-c)>0\}}} Df(s,y).
\end{align*}

\medskip

As the condition \eqref{eq:rel} implies
\begin{align}
\label{eq:for11}
\begin{cases}
\rho<1,\quad p=1  \quad  &\hbox{ in }  \{ p\rho < 1 \},\\
\rho=1,\quad 1<p<2,  \quad  &\hbox{ in }  \{1 < p\rho < 2\},\\
\rho>1,\quad p=2,  \quad  &\hbox{ in }  \{ p\rho > 2 \},\\
\end{cases}
\end{align}
the first system of equations in \eqref{eq:for} is a direct consequence of Theorem~\ref{thm:exi}. Next, we consider the second system of equations in \eqref{eq:for}. For a test function $\xi \in \cC_c^\infty(Q)$ such that $\xi$ is compactly supported in $\{ 1< p \rho < 2 \}$, \eqref{eq:for11} implies 
\begin{align}\label{eq:for21}
0 &= \int_Q - \rho \partial_t\xi + D(p\rho) \cdot D\xi +D\Phi\cdot D\xi \dd x \dd t = \int_Q -  \partial_t\xi + (\rho Dp + p D\rho+D\Phi) \cdot D\xi \dd x \dd t\\
\nonumber & =  \int_Q (Dp+D\Phi) \cdot D\xi \dd x \dd t.
\end{align}
Thus, we conclude that  $- \Delta p =\Delta\Phi$ in $\{1 < p\rho < 2\}$. The other cases follow from \eqref{eq:for11}.

\medskip

Lastly, let us find the boundary condition \eqref{eq:2for} on $\{ p\rho=1\}$ and $\{ p\rho=2\}$. As \cite[Theorem 3.1]{EvaPor04}, we deduce the condition based on integration by parts. Note that the boundary condition \eqref{eq:2for} can be regarded as Rankine-Hugoniot conditions. For a test function $\xi \in \cC_c^\infty(Q)$, \eqref{eq:exi} implies that
\begin{align}
\label{eq:for31}
0 &= \int_{Q} - \rho \partial_t\xi + [D(p\rho)+\rho D\Phi] \cdot D\xi\dd x\dd t = \int_{\{p\rho<1\}} - \rho \partial_t\xi + [D(p\rho)+\rho D\Phi] \cdot D\xi \dd x \dd t\\ 
&+ \int_{\{1 < p\rho<2\}} - \rho \partial_t\xi + [Dp+D\Phi] \cdot D\xi dx dt + \int_{\{p\rho>2 \}} - \rho \partial_t\xi + [D(p\rho)+D\phi] \cdot D\xi \dd x \dd t.\nonumber
\end{align}
For  a set $N = \{p\rho<1\}, \{1 < p\rho<2\}$ or $\{p\rho>2 \}$, the smoothness of $p$ and $\rho$ in $N$ and \eqref{eq:for} imply that
\begin{align}
\label{eq:for32}
\int_{N} - \rho \partial_t\xi + [D(p\rho)+\rho D\Phi] \cdot D\xi \dd x \dd t &= \int_{\partial N} (- \rho \n_t + [D(p\rho)+\rho D\Phi] \cdot \n_x) \xi \dd\sH^d. 
\end{align}
where $\n_t$ and $\n_x$ are the outward normal vectors on $\partial N$ in $x$ and $t$ directions, respectively. From \eqref{eq:for31} and \eqref{eq:for32}, we conclude that
\begin{align}
\label{eq:for33}
0 = \int_{\partial \{ p\rho<1 \}} [D(p\rho)^{1-} - D(p\rho)^{1+}] \cdot \n_x \xi \dd\sH^d + \int_{\partial \{ p \rho >2 \}} [D(p\rho)^{2+} - D(p\rho)^{2-}] \cdot \n_x \xi \dd\sH^d.
\end{align}
By the arbitrariness of $\xi$, \eqref{eq:for33} implies that
\begin{align*}
\begin{cases}
[D(p\rho)^{1+} - D(p\rho)^{1-}] \cdot \n_x &= 0 \hbox{ on } \{p\rho=1\},\\
[D(p\rho)^{2+} - D(p\rho)^{2-}] \cdot \n_x &= 0 \hbox{ on } \{p\rho=2\}. 
\end{cases}
\end{align*}
As $\n_x$ is parallel to $D(p\rho)$ on the level set of $p\rho$, we conclude \eqref{eq:2for}.

\subsection{A nontrivial stationary solution  -- an example}

In this subsection, in one spacial dimension, we study stationary solutions to our problems. For simplicity, let us consider $\Om := (0,l) \subset  \R$ for $l > 0$. Let $(\rho,p)$ be a solution to \eqref{eq:exi}-\eqref{eq:rel} with potential $\Psi(x) = 2x$, where we have associated energy functional,
\begin{align*}
\cJ(\rho(x)) = \int_0^l \cS(\rho(x))\dd x + \int_0^l \Psi(x) \rho(x) \dd x,
\end{align*}
where $S$ is given in \eqref{eq:sl}.
From Theorem~\ref{thm:exi}, there exists a solution $(\rho,p)$ of
\be\label{eq:1sta1}
\left\{
\ba{ll}
\partial_t\rho - \partial^2_x (\rho p) - 2\partial_x \rho =0, & {\rm{in}}\ (0,T)\times (0,l),\\
\rho(0,\cdot)=\rho_0, & {\rm{in}}\ (0,l),\\
\partial_x (\rho p ) + 2\rho= 0 , & {\rm{in}}\ [0,T] \times \partial (0,l),
\ea
\right.
\ee
and $(\rho,p)$ satisfies \eqref{eq:rel}. A stationary problem associated to \eqref{eq:1sta1} is as follows: find $\rho,p : [0,l] \to \R$ satisfying \eqref{eq:rel} and 
\be\label{eq:2sta}
\left\{
\ba{ll}
\partial^2_x (\rho p) + 2\partial_x \rho =0, & {\rm{in}}\ (0,l),\\
\partial_x (\rho p ) + 2\rho= 0 , & {\rm{at}}\ x=0 \hbox{ and } x=l.
\ea
\right.
\ee
The solution $(\rho,p)$ can be also characterized as minimizers of the the free energy $\cJ$. Writing down the optimality conditions (using Lemma \ref{lem:popt}) we have
\begin{align}
\label{eq:sta11}
\rho =
\begin{cases}
\exp \left( A - x \right) &\hbox{ in } \left[0, A \right) \cap [0,l],\\
1 &\hbox{ in } \left[ A, A + \frac{1}{2} \right] \cap [0,l],\\
\exp \left( 2A+1 - 2x \right) &\hbox{ in } \left(A + \frac{1}{2},l  \right] \cap [0,l],
\end{cases}
\end{align}
where $A$ is chosen to satisfy 
\begin{align}
\label{eq:sta12}
\int_0^l \rho \dd x = 1. 
\end{align}
Depending on the value $l$, some cases in \eqref{eq:sta11} may not be present (see Figure~\ref{fig:sta}). 

\begin{figure}[h]
	\centering
	\begin{subfigure}[t]{0.25\textwidth}
\begin{tikzpicture}[xscale=1.3,yscale=0.6]
\draw[scale=2,->] (0,0) -- (1.6,0);
\draw[scale=2,->] (0,0) -- (0,2.3);
\draw[scale=2,dashed] (0,1) -- (1.6,1);
\draw[scale=2,domain=0:0.28,smooth,variable=\x,blue] plot ({\x},{exp(0.28-\x)});
\draw[scale=2,domain=0.28:0.78,smooth,variable=\x,blue] plot ({\x},{1});
\draw[scale=2,domain=0.78:1,smooth,variable=\x,blue] plot ({\x},{exp(2*0.28+1-2*\x)});
\node[left] at (0,2) {$1$};
\node[below] at (2,0){$1$};
\draw[scale=2] (1,-0.05) -- (1,0.05); 
\end{tikzpicture}	
		\caption{$l=1$}
	\end{subfigure}
	\hspace{1.5cm}
	\begin{subfigure}[t]{0.25\textwidth}
\begin{tikzpicture}[xscale=1.3,yscale=0.6]
\draw [scale=2,->] (0,0) -- (1.6,0);
\draw [scale=2,->] (0,0) -- (0,2.3);
\draw [scale=2,dashed] (0,1) -- (1.6,1);
\draw[scale=2,domain=0:0.57,smooth,variable=\x,blue] plot ({\x},{exp(0.57-\x)});
\draw[scale=2,domain=0.57:0.8,smooth,variable=\x,blue] plot ({\x},{1});
\node [left] at (0,2) {$1$};
\node[below] at (1.6,0){$0.8$};
\draw[scale=2] (0.8,-0.05) -- (0.8,0.05);
\end{tikzpicture}	
		\caption{$l=0.8$}
	\end{subfigure}
	\hspace{1.5cm}
	\begin{subfigure}[t]{0.25\textwidth}
\begin{tikzpicture}[xscale=1.3,yscale=0.6]
\draw [scale=2,->] (0,0) -- (1.6,0);
\draw [scale=2,->] (0,0) -- (0,2.3);
\draw [scale=2,dashed] (0,1) -- (1.6,1);
\draw[scale=2,domain=0:0.6,smooth,variable=\x,blue] plot ({\x},{exp(-\x)/(1-exp(-0.6))});
\node [left] at (0,2) {$1$};
\node[below] at (1.2,0){$0.6$};
\draw[scale=2] (0.6,-0.05) -- (0.6,0.05);
\end{tikzpicture}	
		\caption{$l=0.6$}
	\end{subfigure}	
	\caption{Stationary solutions}
	\label{fig:sta}
\end{figure}
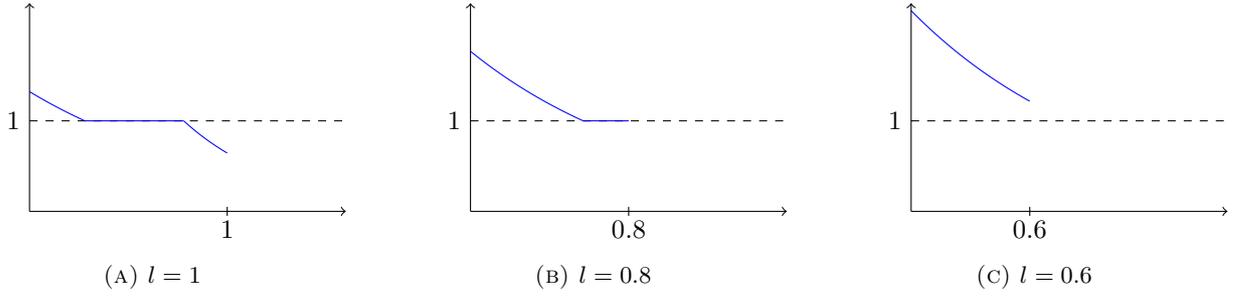

\begin{lemma}
If $l > \ln \left( \frac{3}{2} \right) + \frac{1}{2} $, then all three sets $\{\rho < 1 \}$, $\{ \rho = 1 \}$ and $\{\rho > 1\}$ have positive measure.
\end{lemma}

\begin{proof}
We consider the continuous function $f : \R \to \R$ defined by $$f(x) := \exp(x) - \frac{1}{2} \exp(2x+1-2l) - 1.$$ We claim that there exists a solution of $f(x) = 0$ in $(0, l- \frac{1}{2})$. First, $f(0) = - \frac{1}{2}\exp(1-2l) < 0$. Also, as $l > \ln \left( \frac{3}{2} \right) + \frac{1}{2}$, it holds that
\begin{align*}
f\left( l - \frac{1}{2}\right) =  \exp\left( l - \frac{1}{2} \right) - \frac{3}{2} > 0.
\end{align*}
Therefore, by the continuity of $f$ and by the intermediate value theorem, we can find $A \in (0, l- \frac{1}{2})$ such that $f(A) = 0$. For given $A$, we recall $\rho$ from \eqref{eq:sta11}. Since we have
\begin{align}
\label{eq:ex13}
A, \ \ A+\frac{1}{2} \in (0, l),
\end{align}
all three sets $\{\rho < 1 \}$, $\{ \rho = 1 \}$ and $\{\rho > 1\}$ have positive measure. 

\medskip

It remains to show that with this choice of $A$, $\rho$ given in \eqref{eq:sta11} satisfies \eqref{eq:sta12}. From \eqref{eq:sta11} and \eqref{eq:ex13}, it holds that
\begin{align*}
\int_0^l \rho(x) \dd x &= \int_{0}^{A} \exp \left( A - x \right) \dd x + \int_{A}^{A+\frac{1}{2}} 1 \dd x + \int_{A + \frac{1}{2}}^{l} \exp \left( 2A+1 - 2x \right)  \dd x\\
&=  (\exp(A) - 1) +  \frac{1}{2} + \left( \frac{1}{2} -\frac{1}{2} \exp(2A + 1 - 2l) \right) = f(A) +1 = 1.
\end{align*}
and we conclude.
\end{proof}

\begin{remark}
By a parallel argument as above, one can check that a set $\{\rho > 1\}$ will have zero measure if $l \in \left(0, \ln \left( \frac{3}{2} \right) + \frac{1}{2} \right]$. To see this, let us differentiate two cases. 

{\it Case 1.} If $l \in (0, \ln 2]$, then we have
\begin{align*}
\rho(x) = \frac{\exp(-x)}{1 - \exp(-l)} \quad  {\rm{in}}\ [0,l].
\end{align*} 
{\it Case 2.} If $l \in \left(\ln 2, \ln \left( \frac{3}{2} \right) + \frac{1}{2} \right]$, then it holds that
\begin{align*}
\rho(x) =
\begin{cases}
\exp \left( A - x \right) &{\rm{in}\ } \left[0, A \right),\\
1 &{\rm{in\ }} \left[ A, l \right].
\end{cases}
\end{align*}
where $A \in [l-\frac{1}{2}, l)$ is a solution of $\exp(A) - A = 2-l$.
\end{remark}

\medskip

\section{Acknowledgements}
The authors thank Jos\'e A. Carrillo, Inwon Kim and Filippo Santambrogio for their interests in this project. They also thank Michael R\"ockner for pointing out some interesting references on related problems. We thank the two anonymous referees for their comments and questions that helped us to improve the presentation of our results. The first author was partially supported by NSF grants DMS-1566578 and the Kwanjeong Educational Foundation. The second author was partially supported by the Air Force under the Grant AFOSR MURI FA9550-18-1-0502.


\appendix

\section{Optimal transport toolbox}

Let us recall now some basic definitions and results from the theory of optimal transport. Let $\Pi (\mu,\nu)$ be the set of all Borel probability measure $\pi$ on $\Om \times \Om$ such that
\begin{align*}
\pi(A \times \Om) = \mu(A), \quad \pi(\Om \times B) = \nu(B) \hbox{ for all measurable subsets } A,B \subset \Om.
\end{align*}
For $\mu, \nu \in \sP_2(\Om)$ we define the $2$-Wasserstein or Monge-Kantorovich distance as
\begin{align}
\label{def:w}
W_2(\mu,\nu) := \min \left\{ \int_{\Om \times \Om} |x-y|^2 \dd \gamma : \gamma \in \Pi (\mu,\nu) \right\}^{\frac{1}{2}}.
\end{align}

For $\phi : \Om \to \R$ measurable, we use the notations 
\begin{align*}
\phi^+(x) := \max\{\phi(x),0\}, \phi^-(x) := \max\{-\phi(x),0\} \hbox{ and } \phi^c(x):= {\rm{ess}}\inf\limits_{y \in \Om} \left\{\frac{1}{2}|x-y|^2 - \phi(y)\right\}
\end{align*}
where $x \in \Om.$

\subsection{Basic facts from optimal transport}

Let us recall the definition and properties of Kantorovich potentials and optimal transport maps. There results are well-known in the literature, we refer for instance to \cite{OTAM} for the proofs of the statements.

\begin{definition}
\label{def:kan}
Let $\mu,\nu\in\sP(\Om)$ be given.
$\,$
\begin{enumerate}
\item 
We say that  $\ophi : \Om \rightarrow \R$ is a Kantorovich potential from $\mu$ to $\nu$ if $( \ophi, \ophi^c )$ is a maximizer of the \emph{Kantorovich} problem:
\small
\begin{align*}
\sup \left\{ \int_\Om \phi \dd\mu + \int_\Om \psi \dd\nu : (\phi, \psi) \in L^1_\mu(\Om)\times L^1_\nu(\Om), \phi(x) + \psi(y) \leq \frac{1}{2}|x-y|^2,\  \mu\otimes\nu-{\rm{a.e.}}\ (x,y)\in \Om\times\Om \right\}.
\end{align*}
\normalsize
We denote the set of Kantorovich potential from $\mu$ to $\nu$ by $\cK(\mu,\nu)$.
\item
We say that  a Borel map $T : \Om \rightarrow \Om$ is a optimal transport map from $\mu$ to $\nu$ if $T$ is a minimizer of the following problem:
\begin{align*}
\inf \left\{ \int_\Om |x-T(x)|^2 \dd\mu : T_{\#} \mu =\nu \right\}.
\end{align*}
Here, $(T_{\#} \mu)(A):= \mu(T^{-1}(A))$ for any Borel set $A \subseteq \Om$.
\end{enumerate}
\end{definition}

\begin{lemma} [\cite{OTAM}]
\label{lem:kan} For $\mu \in \sP^{\rm{ac}}(\Om)$ and $\nu\in\sP(\Om)$, there exists a Lipschitz continuous Kantorovich potential $\ophi$ and an optimal transport map $T$ from $\mu$ to $\nu$. Also, it holds that
\begin{align}
\label{eq:1kan}
x - T(x) = \nabla \ophi(x) {\rm{\ for\ a.e.\ }}x\in\spt(\mu) {\rm{\ and\ }} W_2(\mu,\nu) = \| \nabla \ophi \|_{L^2_\mu}.
\end{align}
\end{lemma}

\begin{lemma}{\cite[Theorem 1.3]{Vil03},\cite[Proposition 1.11]{OTAM}}\label{lem:kandual}
Let $\mu,\nu \in \sP(\Om)$. 
Define $\cL:L^1_\mu(\Om)\times L^1_\nu(\Om)\to\R$ as
\begin{align}
\label{eq:ce}
\cL(\phi,\psi) := \int_{\Om} \phi \dd\mu + \int_{\Om} \psi \dd\nu.
\end{align}
Then, it holds that
\begin{align*}
&\frac{1}{2} W_2^2(\mu,\nu) = \max \left\{ \cL(\phi,\psi) : (\phi,\psi) \in C_b(\Om)\times C_b(\Om),\ \phi(x) + \psi(y) \leq \frac{1}{2}|x-y|^2 {\rm{\ for\ all\ }} x,y\in \Om \right\},\\
& = \sup \left\{ \cL(\phi,\psi) : (\phi,\psi) \in L^1_\mu(\Om) \times L^1_\nu(\Om),\ \phi(x) + \psi(y) \leq \frac{1}{2}|x-y|^2 {\rm{\ for\ }}\mu\otimes\nu-{\rm{a.e.}}\ (x,y)\in \Om\times\Om \right\}.
\end{align*}
\end{lemma}

\begin{proposition}
\label{prop:dual}
For $r \in [1,+\infty]$, let $\mu \in L^r(\Om) \cap \sP(\Om)$ and $\nu \in \sP(\Om)$.
Then, it holds that
\begin{align}
\label{eq:1dual}
\sup\limits_{\phi \in L^{r'}(\Om)}  \cL(\phi,\phi^c) = \frac{1}{2} W_2^2(\mu,\nu)
\end{align}
where $r' := \frac{r}{r-1}$ ($r'=1$ if $r=+\infty$ and $r'=+\infty$ if $r=1$) and $\cL$ is given in \eqref{eq:ce}.
\end{proposition}

\begin{proof}

\textbf{Step 1.}
Let us show that
\begin{align}
\label{eq:dual11}
\frac{1}{2} W_2^2(\mu,\nu) = \I_1
\end{align}
where
\begin{align}
\label{eq:dual12}
\I_1 &:= \sup \left\{ \cL(\phi,\psi) : (\phi, \psi) \in L^{r'}(\Om) \times L^1_\nu(\Om),\ \phi(x) + \psi(y) \leq \frac{1}{2}|x-y|^2 {\rm{\ for\ }}\mu\otimes\nu-{\rm{a.e.}}\ (x,y)\in \Om\times\Om \right\}.
\end{align}
By H\"{o}lder's inequality, it holds that
\begin{align}
\label{eq:dual13}
\| \phi \|_{L^1_\mu(\Om)} = \int_\Om |\phi(x)| \mu(x)\dd x \leq \| \phi \|_{L^{r'}(\Om)} \| \mu \|_{L^{r}(\Om)}.
\end{align}
As $\mu \in L^r(\Om) \cap \sP(\Om)$, we conclude that
\begin{align*}
L^{r'}(\Om) \subset L^1_\mu(\Om) \hbox{ and thus } C_b(\Om) \times C_b(\Om) \subset L^{r'}(\Om) \times L^1_\nu(\Om) \subset L^1_\mu(\Om) \times L^1_\nu(\Om).
\end{align*}
From Lemma~\ref{lem:kandual}, we conclude \eqref{eq:dual11}.

\medskip

\textbf{Step 2.}
It remains to show that
\begin{align}
\label{eq:dual21}
\sup\limits_{\phi \in L^{r'}(\Om)} \cL(\phi,\phi^c) = \I_1
\end{align}
for $\I_1$ given in \eqref{eq:dual12}. Indeed, let us notice that by density we have 
$$\sup_{\phi\in L^{r'}(\Om)}\cL(\phi,\phi^c)=\sup_{\phi\in C_b(\Om)}\cL(\phi,\phi^c)=\max_{\phi\in C_b(\Om)}\cL(\phi,\phi^c),$$ and the latter two quantities are finite by \cite[Proposition 1.11]{OTAM}. Therefore the thesis of the proposition follows. 
\end{proof}

\subsection{Some properties of minimizers in the minimizing movements scheme and optimality conditions}

\begin{lemma}\label{lem:postg}
For $\rho_{k}$ given in \eqref{eq:step} and $\cS$ satisfying \eqref{eq:1n}, it holds that $\rho_{k}>0$ a.e.
\end{lemma}

\begin{proof} 
The proof is inspired by \cite[Lemma 8.6]{OTAM}. The difference is that we consider the sub-differential of $\cS$ instead of its derivative.

\medskip

\textbf{Step 1.} For simplicity, let us use the notation $\mu := \rho_{k}$ and consider a competitor 
\begin{align}
\label{eq:mu1}
\mu_1:= \frac{1}{\sL^d(\Om)}. 
\end{align}
Define $\mu_\e:= (1-\e)\mu + \e \mu_1$ for $\e \in (0,1)$. 
From convexity of Wasserstein distance, we have
\begin{align*}
\cI_1:= \cJ(\mu) - \cJ(\mu_\e) &\leq \frac{1}{2\t}W_2^2(\mu_\e,\rho_{k-1}) - \frac{1}{2\t}W_2^2(\mu,
\rho_{k-1}) \leq \e \left\{ \frac{1}{2\t}W_2^2(\mu_1,\rho_{k-1}) - \frac{1}{2\t}W_2^2(\mu,
\rho_{k-1}) \right\}.
\end{align*}
The compactness of $\Om$ implies 
\begin{align}
\label{eq:post1}
\cI_1 &\leq C_1\e \hbox{ for some } C_1>0.
\end{align}

\medskip

\textbf{Step 2.} Set $A:= \{ x \in \Om : \mu >0\}$ and $B:= \{ x \in \Om : \mu = 0 \}$. Let us show that $\sL^d(B)=0$. 
For sufficiently small $\e>0$, it holds that $\e \mu_1<1$ and thus
\begin{align*}
\cI_1 
& = \int_A \cS(\mu(x)) - \cS(\mu_\e(x))+\Phi[\mu(x)-\mu_\e(x)]\dd x + (\cS(0) - \cS(\e \mu_1)) \sL^d(\cB)-\e\frac{1}{\sL^d(\Om)}\int_B\Phi\dd x. 
\end{align*}
By convexity of $\cS$, it holds that 
\begin{align*}
\cI_1 
&\geq \e \int_A [\xi_\e(x)+\Phi] (\mu(x) - \mu_1)\dd x + (\cS(0) - \cS(\e \mu_1)) \sL^d(B)-\e\frac{1}{\sL^d(\Om)}\int_B\Phi\dd x ,
\end{align*}
where $\xi_\e(x) \in \partial \cS(\mu_\e(x))$.

\medskip

 From \eqref{eq:post1}, we conclude that for all $\xi_\e(x) \in \partial \cS(\mu_\e(x))$
\begin{align}
\label{eq:post2}
\cI_2 := \int_A [\xi_\e(x)+\Phi] (\mu(x) - \mu_1)\dd x + \frac{1}{\e}(\cS(0) - \cS(\e \mu_1)) \sL^d(\cB) \leq C_1+C.
\end{align}

\medskip

Note that by the convexity of $S$, its subdifferential is monotone, therefore for all $\e \in [0,1]$, 
\begin{align*}
(\xi_\e(x) - \xi_1)  (\mu_{\e}(x) - \mu_1) \geq 0,
\end{align*}
and thus
\begin{align}
\label{eq:post3}
\xi_\e(x)  (\mu(x) - \mu_1) \geq \xi_1 (\mu(x) - \mu_1),
\end{align}
for a.e. $x\in\Om$ where $\xi_1 \in \partial \cS(\mu_1)$.
Therefore, 
\begin{align*}
\cI_2 \geq \int_A [\xi_1+\Phi] (\mu(x) - \mu_1) \dd x +  \frac{1}{\e}(\cS(0) - \cS(\e \mu_1)) \sL^d(B).
\end{align*}
Since $\cS'(0+) = -\infty$ from \eqref{eq:1n}, the right hand side blows up as $\e$ goes to zero unless $\sL^d(B)=0$. As $\cI_2$ is bounded by $C_1+C$ from \eqref{eq:post2}, we conclude that $\sL^d(B)=0$, and thus $\mu > 0$ a.e.
\end{proof}

\section{Some results from convex analysis}\label{sec:opt}

For a Banach space $\fX$ and $F : \fX \rightarrow \Rinfb$, we say that $F^* : \fX^* \rightarrow \Rinfb$ is a Legendre transform of $F$ if 
\begin{align}
\label{def:lt}
F^*(y) := \sup\limits_{ x \in \fX }{\{ \langle x, y \rangle_{\fX,\fX^*} - F(x) \}} \hbox{ for } y \in \fX^*. 
\end{align}
Here, $\fX^*$ stands for the topological dual space of $\fX$. We will denote by $C_b(\Om)$ the space of bounded continuous functions in $\Om$. In the derivation of optimality conditions associated to the minimizing movement schemes, in Section \ref{sec:opt}, we use subdifferential calculus in $L^r(\Om)$ ($r\in[1,+\infty]$) spaces. Let us recall some basic results on this. 

Let us recall the definition of subdifferentials on $L^r(\Om)^*$ for $r \in [1,+\infty]$.

\begin{definition}
\label{def:subdif}
\cite[(1.9), (1.10) \& (1.13)]{Roc71}
For $\psi : \R \rightarrow \Rinf$, $r \in [1,+\infty]$ and $\Psi : L^r(\Om) \rightarrow \Rinf$ defined by  
\begin{align}
\label{eq:psi}
\Psi(\mu) := \int_\Om \psi(\mu(x))\dd x, 
\end{align}
we say that $\xi \in L^r(\Om)^*$  belongs to the subdifferential of $\Psi$ at $\mu \in L^r(\Om)$ if 
\begin{align}
\label{eq:2psi}
\Psi(\nu)  \geq \Psi(\mu)+  \langle \xi, \nu - \mu \rangle_{L^r(\Om)^*, L^r(\Om)}
\end{align}
for every $\nu \in L^r(\Om)$. 
We denote by $\partial \Psi(\mu)$ the set of subdifferentials of $\Psi$ at the point $\mu \in L^r(\Om)$.
\end{definition}

\begin{definition}\cite[Definition 1.3.1]{EkeTem}
\label{def:gam}
Let $\fX$ be a Banach space. The set of functions $F : \fX \to \Rinfb$ which are pointwise supremum of a family of continuous affine function is denoted by $\Gamma(\fX)$.
\end{definition}

\begin{lemma}\cite[Proposition 1.3.1]{EkeTem}
\label{lem:equ}
 The following properties are equivalent to each other:
\begin{enumerate}
\item
$F \in \Gamma(\fX)$.
\item
$F$ is a convex lower semicontinuous function from $\fX$ into $\Rinfb$ and if $F$ takes the value $-\infty$, then $F$ is identically equal to $-\infty$.
\end{enumerate}
\end{lemma}

\begin{lemma}\cite[Proposition 1.5.6]{EkeTem}
\label{lem:subsum}
If $F_1, F_2 \in \Gamma(\fX)$ and if there exists $\hmu \in \fX$ such that $F_1(\hmu), F_2(\hmu) < +\infty$ and  either $F_1$ or $F_2$ is continuous at $\hmu$, then it holds that
\begin{align*}
\partial F_1(\mu) + \partial F_2 (\mu) = \partial (F_1 + F_2)(\mu) \hbox{ for all } \mu \in \fX.
\end{align*}
\end{lemma}

\section{An Aubin-Lions lemma and some of its consequences}\label{sec:appendix_aubin-lions}

In \cite{RosSav} the authors presented the following version of the classical Aubin-Lions lemma (see \cite{Aub}):

\begin{theorem}{\cite[Theorem 2]{RosSav}}\label{thm:aubin_refined}
Let $B$ be a Banach space and $\cU$ be a family of measurable $B$-valued function. Let us suppose that there exist a normal coercive integrand $\fF:(0,T)\times B\to [0,+\infty]$, meaning that 
\begin{itemize}
\item[(1)] $\fF$ is $\sB(0,T)\otimes\sB(B)$-measurable, where $\sB(0,T)$ and $\sB(B)$ denote the $\s$-algebgras of the Lebesgue measurable subsets of $(0,T)$ and of the Borel subsets of $B$ respectively;
\item[(2)] the maps $v\mapsto \fF_t(v):=\fF(t,v)$ are l.s.c. for a.e. $t\in(0,T)$;
\item[(3)] $\{v\in B:\fF_t(v)\le c\}$ are compact for any $c\ge 0$ and for a.e. $t\in(0,T),$
\end{itemize}
and a l.s.c. map $g:B\times B\to [0,+\infty]$ with the property
$$\left[u,v\in D(\fF_t),\ g(u,v)=0\right]\Rightarrow u=w,\ \text{for }\ae\ t\in(0,T).$$
If $$\sup_{u\in\cU}\int_0^T\fF(t,u(t))\dd t<+\infty\ \ \text{and}\ \ \lim_{h\da 0}\sup_{u\in\cU}\int_0^{T-h}g(u(t+h),u(t))\dd t=0,$$
then $\cU$ is relatively compact in $\sM(0,T; B).$
\end{theorem}

Many recent papers (including \cite{KimMes18,Lab}) on gradient flows in the Wasserstein space used the previous theorem to gain pre-compactness of interpolated curves. In our setting we use the following result.

\begin{lemma}
\label{lem:al}
Let $T>0$ and let $q\in[1,+\infty)$ and $n>0$ be such that $nq^*> 1$, where $q^*:= \frac{qd}{d-q}$ (with the convention $q^*\in(0,+\infty)$ is arbitrary if $q\ge d$, and therefore, $n>0$ and $nq^*>1$ can also be arbitrary). Suppose that $(\rho^\t)_{\t>0}$ is a sequence of curves on $[0,T]$ with values in $\sP(\Om)$ and suppose that there exists $C>0$ such that 
\begin{equation}\label{assumption-AL}
W_2^2(\rho^\t_t,\rho^\t_s)\le C|t-s+\t|,\ \forall\ 0\le s<t\le T
\end{equation}
and
$((\rho^\t)^{n})_{\t>0}$ is uniformly bounded in $L^{q}([0,T]; W^{1,q}(\Om))$ by $C$. We suppose moreover that there exists $\b\ge 1$ such that  $\|\rho^\t_t\|_{L^\beta(\Om)}\le C$ for a.e. $t\in [0,T]$.
\begin{itemize}
\item[(1)] Then, $(\rho^\t)_{\t>0}$ is pre-compact in $L^{\g}(Q)$, with $1\le\g\le\b$ if $\b<nq^*$ and $1\le \gamma<nq^*$, if $\b\ge nq^*$.
\item[(2)] If in addition, $(\rho^\t)_{\t>0}$ is uniformly bounded in $L^{\b_2}(Q)$ for some $\b_2>\g$ (where $\g$ is given in (1)), then $(\rho^\t)_{\t>0}$ is pre-compact in $L^{\g_2}(Q)$, for any $1\le\g_2<\b_2$. 
\end{itemize}
\end{lemma}

\begin{proof}
Let us use the previously stated Aubin-Lions lemma, i.e. Theorem \ref{thm:aubin_refined}. 
Let $1\le \a<q^*$ be fixed (that we set up later) and let us set $B:=L^{n\a}(\Om),$ $\fF:L^{n\a}(\Om)\to[0,+\infty]$ defined as 
$$
\fF(\rho):=\left\{
\ba{ll}
\|\rho^{n}\|_{W^{1,q}(\Om)}, & \text{if } \rho^n\in W^{1,q}(\Om),\ \rho\in\sP(\Om),\\[5pt]
+\infty, & \text{otherwise}
\ea
\right.
$$
and $g:L^{n\a}(\Om)\times L^{n\a}(\Om)\to[0,+\infty]$ defined as 
$$
g(\mu,\nu):=\left\{
\ba{ll}
W_2(\mu,\nu), & \text{if } \mu,\nu \in \sP(\Om),\\[5pt]
+\infty, & \text{otherwise}.
\ea
\right.
$$
In this setting, $(\rho^\t)_{\t>0}$ and $\fF$ satisfy the assumptions of Theorem \ref{thm:aubin_refined}. Indeed, from the assumption, one has in particular that $\ds\int_0^T\|(\rho^{\t}_t)^{n}\|_{W^{1,q}(\Om)}^q\dd t\le C.$ The injection $W^{1,q}(\Om)\hookrightarrow L^{\a}(\Om)$ is compact for any $1\le \a< q^*$, the injection $i:s\mapsto s^{\frac{1}{n}}$ is continuous from $L^\a(\Om)$ to $L^{n\a}(\Om)$ and the sub-level sets of $\rho\mapsto\|\rho^{n}\|_{W^{1,q}(\Om)}$ are compact in $L^{n\a}(\Om)$. 

Moreover,  by the fact that $g$ defines a distance on  $D(\fF)$ and from \eqref{assumption-AL}, one has that $g$ also satisfies the assumptions from Theorem \ref{thm:aubin_refined}, hence the implication of the theorem holds and one has that $\left(\rho^{\t}\right)_{\t\ge 0}$ is pre-compact in $\sM(0,T;L^{n\a}(\Om)).$ Let us notice that \eqref{assumption-AL} implies that there exists $\rho\in C([0,T];\sP(\Om))$ such that up to passing to a subsequence $(\rho_\t)_{\t>0}$ converges uniformly (w.r.t. $W_2$) to $\rho$ as $\t>0$. Up to passing to another subsequence, $\rho$ is the limit also in $\sM(0,T;L^{n\a}(\Om)).$

From our assumption, we know that $\|\rho^\t_t\|_{L^\beta(\Om)}\le C$ for a.e. $t\in [0,T]$. Now, if $\b<nq^*$, then setting $\a$ such that $n\a=\beta$, Lebesgue's dominated convergence theorem implies the strong pre-compactness of $(\rho^\t)_{\t>0}$ in $L^\b(Q)$. Otherwise, Lebesque's dominated convergence implies the strong pre-compactness in $L^\g(Q)$ for any $1\le \g< nq^*.$ This concludes the proof of (1).

\medskip

To show (2), we notice that (1) already implies that $\rho^\t\to\rho$, strongly in $L^\g(Q)$ as $\t\da 0$ and in particular a.e. in $Q$. Furthermore, by the by the uniform bounds in $L^{\b_2}(\Om)$, with $\b_2>\g$, for any $1\le\g_2<\b_2$ we have that  
$$\int_{Q}(\rho^\t)^{\g_2}\dd x\dd t\le \left(T\sL^d(\Om)\right)^{1-\frac{\g_2}{\b_2}}\|\rho^\t\|_{L^{\b_2}}^{\g_2},$$
which implies that $(\rho^\t)^{\g_2}$ is uniformly integrable on $Q$. Therefore, Vitali's convergence theorem yields the claim. 
\end{proof}

\bibliographystyle{alpha}
\bibliography{alpar_2}{}

\end{document}